\documentclass{article}

\usepackage{arxiv}

\usepackage[utf8]{inputenc} 
\usepackage[T1]{fontenc}    
\usepackage[hidelinks]{hyperref}       
\usepackage{url}            
\usepackage{booktabs}       

\usepackage{mathtools}
\usepackage{amsfonts}       
\usepackage{nicefrac}       
\usepackage{microtype}      
\usepackage{lipsum}		
\usepackage{graphicx}
\usepackage{amsthm}
\usepackage{amssymb}
\usepackage{pifont}
\usepackage{xcolor}
\usepackage{caption}
\usepackage{multirow}

\usepackage{hyperref}       
\usepackage{booktabs}       
\usepackage{amsfonts}       
\usepackage{multicol}
\usepackage{amsmath}
\usepackage{algorithm}
\usepackage{algorithmic}
\usepackage{subfig}

\newtheorem{theorem}{Theorem}[section]
\newtheorem{lemma}{Lemma}[section]

\newtheorem{assumption}{Assumption}[section]
\newtheorem{definition}{Definition}[section]
\newtheorem{corollary}{Corollary}[section]

\title{Mirror descent method for stochastic \\ multi-objective optimization}

\date{} 					

\author{ {Linxi Yang}\\
	School of Mathematical Sciences\\
	Sichuan University\\
Sichuan China, 610064 \\
	\texttt{leoyanglinxi@gmail.com} \\
	\And
	{Liping Tang}  \\
	National Center for Applied Mathematics in Chongqing\\
	Chongqing Normal University\\
Chongqing China, 401331 \\
	\texttt{tanglipings@163.com}
\And
{Jiahao Lv}\\
	School of Mathematical Sciences\\
	Sichuan University\\
Sichuan China, 610064 \\
	\texttt{jiahaolv554@gmail.com} \\
\And
{Yuehong He}\\
	School of Mathematical Sciences\\
	Sichuan University\\
Sichuan China, 610064 \\
	\texttt{heyuehong1111@163.com} \\
\And
{Xinmin Yang}\thanks{Corresponding author}\\
	National Center for Applied Mathematics in Chongqing\\
	Chongqing Normal University\\
Chongqing China, 401331 \\
	\texttt{xmyang@cqnu.edu.cn} \\
}

\renewcommand{\shorttitle}{\textit{arXiv} Template}


\hypersetup{
pdftitle={A template for the arxiv style},
pdfsubject={~},
pdfauthor={~},
pdfkeywords={~},
}

\hypersetup{
    colorlinks=true,
    linkcolor=blue,
    urlcolor=blue,
    citecolor=blue
}

\usepackage{cleveref}
\usepackage{bm}
\definecolor{darkgreen}{RGB}{0,100,0} 

\begin{document}
\maketitle

\begin{abstract}
Stochastic multi-objective optimization (SMOO) has recently emerged as a powerful framework for addressing machine learning problems with multiple objectives.
The bias introduced by the nonlinearity of the subproblem solution mapping complicates the convergence analysis of multi-gradient methods.
In this paper, we propose a novel SMOO method called the Multi-gradient Stochastic Mirror Descent (MSMD) method, which incorporates stochastic mirror descent method to solve the SMOO subproblem, providing convergence guarantees.
By selecting an appropriate Bregman function, our method enables analytical solutions of the weighting vector and requires only a single gradient sample at each iteration.
We demonstrate the sublinear convergence rate of our MSMD method under four different inner and outer step setups.
For SMOO with preferences, we propose a variant of MSMD method and demonstrate its convergence rate.
Through extensive numerical experiments, we compare our method with both stochastic descent methods based on weighted sum and state-of-the-art SMOO methods. Our method consistently outperforms these methods in terms of generating superior Pareto fronts on benchmark test functions while also achieving competitive results in neural network training.
\end{abstract}


\section{Introduction}
Multi-objective optimization (MOO) problem requires the simultaneous optimization of multiple objective functions. In general, it is infeasible to find a solution that optimizes all objectives simultaneously. Instead, only a Pareto optimal set can be obtained, where improving one objective necessarily worsens others.
Consequently, solving MOO problems presents significant challenges, prompting researchers to propose various strategies for identifying the Pareto optimal set \cite{fukuda2014survey}.
A notable class of methods is the descent (multi-gradient) method, which manipulates gradients to determine a direction that decreases all objectives simultaneously.
This method, known as the multiple gradient descent algorithm (MGDA), was first proposed by Fliege and Svaiter \cite{fliege2000steepest}, and it determines a common descent direction by solving a strongly convex subproblem. Building on this algorithm, various multi-gradient methods have been developed, including projected gradient methods \cite{fukuda2013inexact,drummond2004projected,zhao2021convergence} and Newton and quasi-Newton methods \cite{lapucci2023limited,fliege2009newton,gonccalves2022globally}.

In recent years, the MOO problem has gained significant attention across various real-world applications,  including online advertisement modeling \cite{ma2018modeling, lin2019pareto} , semantic segmentation \cite{long2015fully} and autonomous driving systems \cite{huang2019apolloscape}. Against this backdrop, the MOO problem is often referred to as a multi-task learning (MTL) problem, where different tasks are associated with distinct objectives.
Consequently, several gradient manipulation-based methods have emerged.
For instance, Sener et al. \cite{sener2018multi} initially approached MTL problem as a MOO problem, employing MGDA for direct resolution. Yu et al. \cite{yu2020gradient} introduced PCGrad based on Schmidt orthogonalization and Liu et al. \cite{liu2021conflict} presented CAGrad, incorporating constraints to restrict descent direction.
However, the MTL problems in machine learning are typically large-scale, these descent methods require computing the gradients of all objective functions over the entire dataset at each iteration. This makes traditional convergence analysis of multi-gradient methods impractical due to the computational burden. A similar challenge arises in single-objective machine learning problems, which led to the development of the stochastic gradient descent (SGD) method.
SGD mitigates this issue by using a subset of the gradient (batch gradient) in place of the full gradient, ensuring convergence over a sufficient number of iterations. While SGD has been extensively studied and proven effective in the optimization with single-objective  \cite{bottou2018optimization, lan2020first, robbins1951stochastic}, the application of stochastic methods in optimization with multiple objectives remains not fully understood.

We refer to MOO method using stochastic gradients as stochastic multi-objective optimization (SMOO) method. In most multi-gradient methods of MOO derive a descent direction by computing a convex combination of gradients, with the weighting parameters determined by solving strongly convex subproblems. However, the nonlinearity in the solution mapping of these subproblems can introduce biased weighting parameters, even when the gradients of each objective are assumed to be unbiased. This introduces a significant challenge for convergence analysis in SMOO methods.
To address this challenge, Liu and Vicente et al. proposed a stochastic multiple gradient (SMG) method \cite{liu2021stochastic} for convex objectives. This method directly substitutes the full gradient in the MGDA with a batch gradient. It guarantees convergence by assuming that the solution mapping of the subproblems exhibits Lipschitz continuity and by gradually increasing the batch sample size.
However, Zhou et al. have demonstrated in their work \cite{zhou2022convergence} that this assumption of continuity is not appropriate. They introduced a correlation-reduced multi-objective gradient manipulation (CR-MOGM) method in more general non-convex scenarios, which eliminates the need for such continuity assumptions.
CR-MOGM incorporates a tracking variable to adjust the weighting parameter of the update direction, reducing the correlation between the weight parameter and the stochastic batch gradients. This adjustment facilitates the identification of a descent direction that advances towards the Pareto set. As a result, CR-MOGM achieves a convergence rate result of ${O\left( {{K^{ - 1/4}}} \right)}$, where $K$ denotes the total number of iterations. Similarly, Fernando et al. introduced the multi-objective gradient with correction (MoCo) method \cite{fernando2023mitigating}, which utilizes a tracking variable incorporating historical information. MoCo differs from CR-MOGM in that it leverages historical gradient information as a correction for the current step gradient, enhancing the stability of the update direction.
All of these SMOO methods rely on obtaining exact solutions for the subproblems. In contrast, Xiao et al. \cite{xiao2024direction} presented the stochastic direction-oriented multi-objective gradient descent (SDMGrad) method, which employs a single-objective stochastic strategy to solve subproblems.  Differing from previous exact methods, this approach constructs subproblems similarly to deterministic MGDA but employs SGD for solving and obtains bias descended updating directions.
By setting the number of subproblem iteration as $S = {K^2}$, SDMGrad achieves a convergence rate of ${O\left( {1/\sqrt K } \right)}$.
However, this result is guaranteed through triple sampling of i.i.d. stochastic gradients, which can introduce considerable computational overhead, especially when faced with numerous iterations or expensive gradient sampling.

In most multi-gradient methods, subproblems are addressed by transforming the original min-max problem into a constrained optimization problem. This transformation is motivated by the fact that the resulting subproblem can be expressed as a minimum $\ell_2$-norm problem with simplex constraints, which can be efficiently solved using the Frank-Wolfe method \cite{beck2017first,sener2018multi}.
However,  as previously discussed,  analyzing the convergence of this approach becomes challenging in stochastic scenarios. Therefore, we consider directly solving the saddle point subproblem, as such problems in stochastic settings have been thoroughly studied and exhibit good convergence results \cite{beznosikov2023stochastic, lan2020first,yoon2021accelerated,tseng1995linear}.
Among these, variants of mirror descent (also known as Bregman gradient) method \cite{cesa2012mirror, duchi2012ergodic} are particularly prominent due to their adaptability to different geometries via the selection of appropriate Bregman functions, enabling faster convergence rates.
The stochastic mirror descent (SMD) method, initially introduced by Nemirovskij et al. in \cite{nemirovskij1983problem}, and extended to convex optimization in \cite{nemirovski2009robust} achieving sublinear convergence under appropriate assumptions. More recent studies have explored its application to nonsmooth nonconvex finite-sum problems \cite{li2022variance}, nonconvex nonconcave problems \cite{mertikopoulos2018optimistic} and weakly convex minmax problems \cite{rafique2022weakly}.

Inspired by SDMGrad and aiming to circumvent the overhead caused by triple sampling, we adopt the SMD as our subproblem optimizer, and propose a novel SMOO method called Multi-gradient Stochastic Mirror Descent (MSMD) method.
Our MSMD method offers several advantages. First, to the best of our knowledge, it is the first stochastic multi-gradient method to employ a saddle point optimizer for solving SMOO subproblems. This integration effectively reduces bias in updating directions during subproblem solving, thereby mitigating the challenge of analyzing the convergence in SMOO.
Secondly, by selecting an appropriate Bergman function, our method provides an analytical solution at each subproblem iteration, requiring only one stochastic gradient sample, resulting in a small computational burden per iteration. This makes our method well-suited for large-scale MOO problems, such as neural network training.
Furthermore, under reasonable assumptions, we demonstrate that our method achieves a convergence rate comparable to the state-of-art SMOO methods. By selecting different inner step size $\gamma_s$ for subproblem updates and outer step size $\alpha_k$ for the descent of the SMOO objective function, we establish convergence proofs under four distinct step size settings, as detailed in Table \ref{tab:convergence}. When the outer step size is varying, our method achieves a convergence rate of $O\left( {1/\ln \left( {K + 1} \right)} \right)$. Otherwise, it attains a faster convergence rate of ${O\left( {1/\sqrt K } \right)}$, comparable to that of SDMGrad. Finally, we evaluated our MSMD method on benchmark MOO test functions and widely recognized neural network training tasks. The numerical results demonstrate that MSMD is highly competitive when compared to both weighted methods with single-objective stochastic optimizer and state-of-the-art SMOO methods.

\begin{table}[htbp]
  \centering
  \renewcommand\arraystretch{1.5}
  \caption{Convergence rate results}
    \begin{tabular}{|c|c|c|c|}
    \hline
    ${\alpha_k}$&  ${\gamma_s}$ & Corollary &   Convergence rate \\
    \hline
    fixed& fixed& \ref{cor:fixa_fixg} &${O\left( {1/\sqrt K } \right)}$\\
    \hline
    fixed&varying& \ref{cor:fixa_alterg} &${O\left( {1/\sqrt K } \right)}$  \\
    \hline
    varying&fixed & \ref{cor:altera_fixg}&$O\left( {1/\ln \left( {K + 1} \right)} \right)$\\
    \hline
    varying&varying & \ref{cor:altera_alterg}&$O\left( {1/\ln \left( {K + 1} \right)} \right)$\\
     \hline
    \end{tabular}%
\label{tab:convergence}%
\end{table}%

The structure of this paper is as follows: Section \ref{sec:preliminaries} provide fundamental definitions, notations, deterministic and stochastic MOO methods, along with SMD method. In Section \ref{sec:ourmethod} we formally introduce our MSMD method, followed by a convergence analysis in Section \ref{sec:convergences}. Variants of MSMD with preferences are discussed in Section \ref{sec:preferences}.  Experimental results, showcasing the Pareto front on benchmark MOO test functions and the performance on MTL, are presented in Section  \ref{sec:experment}. Finally, we conclude in the last section.

\section{Preliminaries}
\label{sec:preliminaries}
In this section, we provide a comprehensive overview of the fundamental background knowledge related to MOO and its stochastic variations, along with an introduction to classical algorithms and their convergence analyses. Subsequently, we briefly explain the fundamentals of SMD method.

We provide notation throughout this work. For any natural number $n$, ${\left[ n \right] = \left\{ {1, \cdots ,n} \right\}}$ represents the set starting at $1$ and ending at $n$, $\mathbb{R}_+$ stands for all positive real numbers, $\mathbb{R}^n$ denotes the $n$-dimensional real space and ${\Delta ^m}$ represents the standard simplex in ${{\mathbb{R}^m}}$ given by
\[{\Delta ^m}: = \left\{ {x \in {\mathbb{R}^m}|\sum\limits_{i = 1}^m {{x_i} = 1,}~ x \ge 0} \right\}.\]
Moreover, ${\left\langle { \cdot , \cdot } \right\rangle }$ stand for the usual canonical inner product in ${\mathbb{R}^n}$, i.e. ${\left\langle {x,y} \right\rangle  = \sum\limits_{i = 1}^n {{x_i}{y_i}}  = {x^T}y}$. Likewise ${{\left\| x \right\|_2} = \sqrt {\left\langle {x,x} \right\rangle } }$ denote Euclidean norm, ${{\left\| x \right\|_\infty } = \max \left\{ {\left| {{x_1}} \right|, \cdots ,\left| {{x_n}} \right|} \right\}}$ denote infinite norm and ${{\left\| x \right\|_1} = \sum\limits_{i = 1}^n {\left| {{x_i}} \right|} }$ is $\ell_1$-norm. Meanwhile we can define the dual norm of any general norm ${{\left\| x \right\| }}$ on ${\mathbb{R}^n}$ to be ${{\left\| x \right\|_*} = \mathop {\sup }\limits_{\left\| y \right\| < 1} {x^T}y }$.  We denote that the metric projection onto the set $X$ by ${\prod\nolimits_X {\left(  \cdot  \right)} }$. Note that ${\prod\nolimits_X {\left( \cdot \right)} }$ is a nonexpanding operator, for example in the case of Euclidean norms we have ${\prod\nolimits_X {\left( x \right)}  = \mathop {\arg \min }\limits_{y \in X} {\left\| {x - y} \right\|_2}}$ and
\[{\left\| {\prod\nolimits_X {\left( x \right)}  - \prod\nolimits_X {\left( y \right)} } \right\|_2} \le {\left\| {x - y} \right\|_2}.\]
The notation ${{\rm O}\left(  \cdot  \right)}$ stands for the infinitesimal of the same order. By ${\left\lceil x \right\rceil }$ we denote the ceil function for the smallest integer greater than or equal to ${x \in \mathbb{R}}$. $\xi$ denote the stochastic vector with probability distribution $P$ and supported on set ${\Xi  \subset {\mathbb{R}^n}}$.

\subsection{Deterministic Multi-objective optimization}
\label{sec:moo}
MOO usually considers the following unconstrained optimization problem:
\begin{align}\label{eq:moo}
\mathop {\min }\limits_{x \in {\mathbb{R}^n}} F\left( x \right) = {\left( {{f_1}\left( x \right),{f_2}\left( x \right), \cdots ,{f_m}\left( x \right)} \right)^T},
\end{align}
where ${{f_i}\left( \cdot \right):{\mathbb{R}^n} \to \mathbb{R},i  \in \left[ m \right]}$ are continuously differentiable functions and $m\ge2$ is the number of objectives.
We denote the Jacobian matrix associated with $F$ as ${\nabla F\left(  \cdot  \right) = {\left( {\nabla {f_1}\left(  \cdot  \right),\nabla {f_2}\left(  \cdot  \right), \cdots ,\nabla {f_m}\left(  \cdot  \right)} \right)^T} \in {\mathbb{R}^{m \times n}}}$, where each ${\nabla {f_i}\left(  \cdot  \right)}$ represents the gradient of the $i$-th objective function.
In MOO, the objective function value is represented as a vector, which may lead to situations where two points cannot be directly compared. Therefore, it is crucial to consider the partial order relationship between vectors in ${\mathbb{R}^m}$. For any two points ${u,v \in {\mathbb{R}^m}}$, we define the following partial order relations:
 \[\begin{array}{l}
u < v \Leftrightarrow {u_i} < {v_i},\forall i  \in \left[ m \right],\\
u \le v \Leftrightarrow {u_i} \le {v_i},\forall i  \in \left[ m \right].
\end{array}\]
We say that $u$ dominates $v$ if $u \le v$ and $u \neq v$. The concept of multi-objective optimality is based on this dominance relation.

\begin{definition}(Pareto optimal point)
\label{def:Pareto optimal}
For any point ${x' \in {\mathbb{R}^n}}$ is Pareto optimal for problem (\ref{eq:moo}) if there dose not exist ${x \in {\mathbb{R}^n}}$ such that $F(x)$ dominates ${F(x')}$.
\end{definition}

In practice, achieving Pareto optimality is often challenging. Therefore, a more practical alternative is to consider weak Pareto optimality, which imposes a less stringent condition and is easier to implement.

\begin{definition}(weakly Pareto optimal point)
\label{def:weakly Pareto optimal}
For any point ${x' \in {\mathbb{R}^n}}$ is weakly Pareto optimal for problem (\ref{eq:moo}) if there dose not exist ${x \in {\mathbb{R}^n}}$ such that ${F(x)<F(x')}$.
\end{definition}

The Pareto set is defined as the collection of all Pareto optimal points, while the corresponding function values of these solutions form the Pareto front in ${\mathbb{R}^m}$. Under the assumption of differentiability, the concept of Pareto stationarity can be introduced. Specifically, if the objective functions in Problem (\ref{eq:moo}) exhibit convexity, Pareto stationarity becomes a sufficient condition for Pareto optimality.

\begin{definition}(Pareto stationary)
\label{def:Pareto stationary}
For any point ${x' \in {\mathbb{R}^n}}$ is Pareto stationary for problem (\ref{eq:moo}) if
\[\mathop {\min }\limits_{d \in {\mathbb{R}^n}} \mathop {\max }\limits_{i  \in \left[ m \right]} \nabla {f_i}{\left( {x'} \right)^T}d \ge  0.\]
\end{definition}
This definition indicates that if $x'$ is not Pareto stationary, there exist a direction $d$ that serves as a local descent direction of $F$ at $x'$.
The following lemma demonstrates the relationship between the concepts of Pareto stationarity, Pareto optimality and convexity.

\begin{lemma}(\cite{fliege2000steepest}, Theorem 3.1)
The following statements hold:

(i) if $x'$ is locally weakly Pareto optimal, then x is Pareto stationary for problem (\ref{eq:moo});

(ii) if $F$ is convex and $x'$ is Pareto stationary for problem (\ref{eq:moo}), then $x'$ is weakly Pareto optimal;

(iii) if $F$ is twice continuously differentiable, ${{\nabla ^2}{f_i}(x) \succ 0,i\in \left[ m \right]}$ and all ${x \in {\mathbb{R}^n}}$, and $x'$ is Pareto stationary for problem (\ref{eq:moo}), then $x'$ is Pareto optimal.
\end{lemma}

In contrast to Pareto optimal point, Pareto stationary reflects local properties and is therefore frequently used as a condition for identifying local minima in non-convex MOO problems \cite{fliege2019complexity}.

In the subsequent sections, we will introduce MGDA, a classical gradient-based MOO method. Similar to the gradient descent method used in single objective optimization, MGDA iteratively updates based on gradients. At each iteration ${k \in \left[ K \right]}$, MGDA seeks to identify a direction $d_k$ that enables simultaneous descent of all objective functions by considering the following approximation:
\[\mathop {\min }\limits_{d \in {\mathbb{R}^n}} \mathop {\max }\limits_{i \in \left[ m \right]} {f_i}({x_k}) - {f_i}({x_k} + {\alpha _k}{d}) \approx  - {\alpha _k}\mathop {\min }\limits_{d \in {\mathbb{R}^n}} \mathop {\max }\limits_{i \in \left[ m \right]} \nabla {f_i}{\left( {{x_k}} \right)^T}{d}.\]
Adding a quadratic regularization term, we can get a unconstrained scalar-valued minimization problem
\[\mathop {\min }\limits_{d \in {\mathbb{R}^n}} \mathop {\max }\limits_{i \in \left[ m \right]} \nabla {f_i}{\left( {{x_k}} \right)^T}d + \frac{1}{2}\left\| d \right\|_2^2.\]
Since the objective function of this problem is proper, closed and strongly convex, there always exists a unique optimal solution. This solution corresponds to the steepest descent direction $d_k$ that we are seeking and can be obtained by solving the following optimization problem:
\begin{align}\label{eq:dk}
{d_k} = \mathop {\arg \min }\limits_{d \in {\mathbb{R}^n}} \mathop {\max }\limits_{i \in \left[ m \right]} \nabla {f_i}{\left( {{x_k}} \right)^T}d + \frac{1}{2}\left\| d \right\|_2^2.
\end{align}
Note that in order to eliminate the non-differentiability of the objective function in (\ref{eq:dk}), we reformulate the problem as follows:
\begin{align}\label{eq:primal_p}
\begin{array}{l}
\min \tau \\
s.t.\nabla {f_i}{\left( {{x_k}} \right)^T}d + \frac{1}{2}\left\| d \right\|_2^2 - \tau  \le 0,i \in \left[ m \right],
\end{array}
\end{align}
which is a convex constrained problem with ${\left( {\tau ,d} \right) \in \mathbb{R} \times {\mathbb{R}^n}}$.
The function
$\phi \left( {d,\lambda } \right):{\mathbb{R}^n} \times {\Delta ^m} \to \mathbb{R}$, with ${\lambda  = \left( {{\lambda _1}, \cdots ,{\lambda _m}} \right)}$ is defined as:
\[\phi \left( {d,\lambda } \right) = \sum\limits_{i = 1}^m {{\lambda _i}} \left( {\left\langle {\nabla {f_i}\left( {{x_k}} \right),d} \right\rangle  + \frac{1}{2}\left\| d \right\|_2^2} \right) = \left\langle {\nabla F{{\left( {{x_k}} \right)}}\lambda ,d} \right\rangle  + \frac{1}{2}\left\| d \right\|_2^2,\]
where ${\nabla {f_i}\left( {{x_k}} \right)}$ represents the Jacobian of the objective functions at iteration $k$. Thus, subproblem (\ref{eq:dk}) can be formulated as a convex-concave saddle point problem:
\begin{align}\label{eq:Deterministic sub}
\mathop {\arg \min }\limits_{d \in {\mathbb{R}^n}} \mathop {\max }\limits_{\lambda  \in {\Delta ^m}} \phi \left( {d,\lambda } \right).
\end{align}
Let ${\left( {{d_k},{\lambda _k}} \right)}$ denote the saddle point of the above problem. According to the strong duality theorem, the primal and dual optimization problems are solvable with equal optimal values, leading to the following relationship:
\begin{align}\label{eq:moo_optimal}
\left( {{d_k},{\lambda _k}} \right) = \mathop {\arg \min }\limits_{d \in {\mathbb{R}^n}} \left[ {\mathop {\max }\limits_{\lambda  \in {\Delta ^m}} \phi \left( {d,\lambda } \right)} \right] = \mathop {\max }\limits_{\lambda  \in {\Delta ^m}} \left[ {\mathop {\arg \min }\limits_{d \in {\mathbb{R}^n}} \phi \left( {d,\lambda } \right)} \right].
\end{align}

The optimal solution pairs ${\left( {{d_k},{\lambda _k}} \right)}$  constitute the set of saddle points of ${\phi \left( {d,\lambda } \right)}$ over ${{\mathbb{R}^n} \times {\Delta ^m}}$.
By applying the Karush-Kuhn-Tucker (KKT) conditions, we can determine the optimal ${\lambda_k}$ by solving a minimal $\ell_2$-norm element with simplex constraints:
\[{\lambda _k} = \mathop {\arg \min }\limits_{\lambda  \in {\Delta ^m}} \left\| {\nabla F{{\left( {{x_k}} \right)}}\lambda } \right\|_2^2.\]
Finally, the common descent direction is given by
\begin{align}\label{eq:dk_optimal}
 {{d_k} =  - \sum\limits_{i = 1}^m {{\lambda _i}\nabla {f_i}\left( {{x_k}} \right)} }.
\end{align}
In addition to solving the dual problem to obtain the direction of steepest descent, there exists a class of gradient-based MOO methods that directly modify gradients. A prominent example is PCGrad \cite{yu2020gradient}, which addresses conflicting gradients by projecting them onto the normal plane of another gradient to obtain the update direction. Specifically, for two conflicting gradients ${\nabla {f_i}\left( {{x_k}} \right)}$ and ${\nabla {f_j}\left( {{x_k}} \right)}$, ${i,j \in \left[ m \right]}$, ${i \ne j}$  a conflict arises when ${\nabla {f_i}{\left( {{x_k}} \right)^T}\nabla {f_j}\left( {{x_k}} \right) < 0}$. Let the projected gradient be denoted by ${\upsilon _i^{PC}\left( {{x_k}} \right) = \nabla {f_i}\left( {{x_k}} \right)}$. The projected gradients are then iteratively updated using the following rule:
\[\upsilon _i^{PC}\left( {{x_k}} \right) = \upsilon _i^{PC}\left( {{x_k}} \right) - \frac{{\upsilon _i^{PC}{{\left( {{x_k}} \right)}^T}\nabla {f_j}\left( {{x_k}} \right)}}{{\left\| {\nabla {f_j}\left( {{x_k}} \right)} \right\|_2^2}}\nabla {f_j}\left( {{x_k}} \right),\]
and this update continues until the gradients are no longer in conflict with ${\upsilon _i^{PC}\left( {{x_k}} \right)}$. While gradient-based methods for deterministic MOO have been extensively studied, their stochastic variants have not yet been thoroughly investigated.
\subsection{Stochastic Multi-objective Optimization}
Deterministic MOO can be effectively addressed by gradient manipulation algorithms, which not only provide favorable numerical results but also come with solid theoretical guarantees \cite{fukuda2014survey}. However, for large-scale MOO problems, such as those encountered in deep learning, obtaining the full gradient ${{\nabla F\left( {{x_k}} \right)}}$ is computationally infeasible. As a result, training with mini-batch gradients has become a practical solution \cite{zhou2022convergence}.
In this context, we can introduce a stochastic gradient approach, where the gradient is computed based on a stochastic vector ${\xi }$, representing the stochasticity of mini-batch selection. Specifically, we define the stochastic gradient as:
\[\nabla F{\left( {{x_k};\xi } \right)^T} = {\left( {\nabla {f_1}\left( {{x_k};\xi } \right),\nabla {f_2}\left( {{x_k};\xi } \right), \cdots ,\nabla {f_m}\left( {{x_k};\xi } \right)} \right)^T} \in {\mathbb{R}^m} \times {\mathbb{R}^n} \times \Xi. \]
Let ${{\xi _{\left[ S \right]}} = \left\{ {{\xi _1}, \cdots ,{\xi _S}} \right\}}$ denote the history of the process from ${{{\xi _1}}}$ to ${{{\xi _S}}}$.
We assume that the stochastic gradient is an unbiased estimator of the full gradient, i.e., ${\mathbb{E}\left[ {\nabla F\left( {{x_k};\xi } \right)} \right] = \nabla F\left( {{x_k}} \right) = \nabla {F_k}}$. We define the function ${\phi \left( {d,\lambda ;\xi } \right)}$ as:
\[\phi \left( {d,\lambda ;\xi } \right) = \left\langle {\nabla F{{\left( {{x_k};\xi } \right)}}\lambda ,d} \right\rangle  + \frac{1}{2}\left\| d \right\|_2^2,\]
which is linear with respect to the stochastic gradient ${{\nabla F\left( {{x_k},\xi } \right)}}$. Therefore, the expectation of this function is also unbiased, i.e. ${\mathbb{E}\left[ {\phi \left( {d,\lambda ;\xi } \right)} \right] = \phi \left( {d,\lambda } \right)}$.
As a result, the problem from the deterministic case (\ref{eq:Deterministic sub}) can now be written in the stochastic context as:
\begin{align}\label{eq:dk_1}
\mathop {\arg \min }\limits_{d \in {\mathbb{R}^n}} \mathop {\max }\limits_{\lambda  \in {\Delta ^m}} \left\{ {\mathbb{E}\left[ {\phi \left( {d,\lambda ;\xi } \right)} \right] = \left\langle {\mathbb{E}\left[ {\nabla F\left( {{x_k};\xi } \right)} \right]\lambda ,d} \right\rangle  + \frac{1}{2}\left\| d \right\|_2^2} \right\}.
\end{align}
For such a stochastic SMOO problem, Liu and Vicente et al. \cite{liu2021stochastic} were among the first to explore a stochastic variant of the MGDA. By directly replacing the full gradient in Problem (\ref{eq:Deterministic sub}) with a single sampled stochastic gradient, they proposed the following formulation:
\[\mathop {\arg \min }\limits_{d \in {\mathbb{R}^n}} \mathop {\max }\limits_{\lambda  \in {\Delta ^m}} \left\langle {\nabla F\left( {{x_k};\xi } \right)\lambda ,d} \right\rangle  + \frac{1}{2}\left\| d \right\|_2^2.\]
Here, the update for $x_k$ is along the direction $d_k$, given by:
\begin{align}\label{eq:smg}
{d_k\left( \xi  \right)} = \nabla F\left( {{x_k},\xi } \right){\lambda _k}\left( \xi  \right)~~s.t.~~{\lambda _k}\left( \xi  \right) \in \mathop {\arg \min }\limits_{\lambda  \in {\Delta ^m}} \left\| {\nabla F\left( {{x_k},\xi } \right)\lambda } \right\|_2^2.
\end{align}
However, due to the nonlinearity of the solution mapping with respect to the stochastic gradient, direct substitution can introduce biased multi-gradient estimates. Specifically ${\mathbb{E}\left[ {{d_k}\left( \xi  \right)} \right] = \mathbb{E}\left[ {\nabla F\left( {{x_k};\xi } \right){\lambda _k}\left( \xi  \right)} \right] \neq \nabla F\left( {{x_k}} \right){\lambda _k}}$, leading to biased updates.
To mitigate this bias, Liu et al. \cite{liu2021stochastic} proposed a method that reduces the variance of the stochastic gradient by linearly increasing the batch size with each iteration. However, when discussing the convergence of this modified MGDA method, Liu et al. made an unreasonable assumption that the solution mapping of  $\lambda_k\left( \xi  \right)$ in (\ref{eq:smg}) is Lipschitz continuous with respect to $\nabla F({x_k;\xi})$, which has been proven incorrect in \cite{zhou2022convergence}.
In response, to develop SMOO methods with convergence guarantees, several recent studies have introduced alternative techniques for updating the weight ${\lambda_k}$. Yang et al. \cite{yang2023learning}, building on \cite{svaiter2018multiobjective}, demonstrated that ${\lambda_k\left( \xi  \right)}$ is Hölder continuity with respect to the gradient ${{\nabla F\left( {{x_k};\xi } \right)}}$. This result facilitates a convergent SMOO method achieved through the dynamic augmentation of sample size, defined as:
\begin{align*}
{N_k} = \max \left\{ {{N_B},{k^q}} \right\},
\end{align*}
where the rate of sample increase is determined by the parameter $q$, and $N_B$ is a threshold for the dynamic sample size. With this dynamic step setting, the update direction $d_k\left( \xi  \right)$ is expressed as:
\[{d_k}\left( \xi  \right) = \frac{1}{{{N_k}}}\sum\limits_{j = 1}^{{N_k}} {\nabla F\left( {{x_k},{\xi _{k,j}}} \right)} {\lambda _k}\left( \xi  \right)\;\;s.t.\;\;{\lambda _k}\left( \xi  \right) \in \mathop {\arg \min }\limits_{\lambda  \in {\Delta ^m}} \left\| {\frac{1}{{{N_k}}}\sum\limits_{j = 1}^{{N_k}} {\nabla F\left( {{x_k},{\xi _{k,j}}} \right)} \lambda } \right\|_2^2.\]
By obtaining a sufficiently small $q$, this method easily has not much computational overhead and has convergence guarantees.
Unlike methods that reduce the variance of the stochastic gradient by increasing the sample size at each step, Fernando et al. \cite{fernando2023mitigating} propose MoCo method, a "tracking" variable iteration method, using historical information to correct the current stochastic gradient. They update ${x_k}$ based on this approximate gradient
\[{y_{k-1,i}} = \prod\limits_C {\left( {{y_{k - 2,i}} - {\beta _{k - 2}}\left( {{y_{k - 2,i}} - \nabla {f_i}\left( {{x_{k - 2}},{\xi _{k - 2}}} \right)} \right)} \right)} ,i \in \left[ m \right]\]
where $C$ is a bounded set and $\beta_{k-2}$ is the step size. In order to ensure the convergence of this method, they adding regular terms to the subproblems of $\lambda_k$ to obtain a unique solution have
\[{d_k}\left( \xi  \right) =  - {Y_k}{\lambda _{k,\rho }}~~s.t~~{\lambda _{k,\rho }} = \prod\limits_{{\Delta ^m}} {\left( {{\lambda _{{k-1},\rho }} - {\gamma _{k-1}}\left( {Y_{k-1}^T{Y_{k-1}} + \rho I} \right){\lambda _{k - 1,\rho }}} \right)} ,\]
where ${\gamma_{k-1}}$ is the step size,  ${{Y_{k-1}} = \left( {{y_{k-1,1}}, \cdots {y_{{k-1},m}}} \right) \in {\mathbb{R}^{n \times m}}}$ is the tracking variable for the gradient, the ${\rho  >0}$ denoted the regularization constant.

Similarly getting information from history iteration, Zhou et al. \cite{zhou2022convergence} started from the correlation between the parameter ${\lambda_k}$ and the update direction ${d_k}$, considered the history weight ${\left\{ {{ \lambda _1}, \cdots ,{ \lambda _{k - 1}}} \right\}}$, and proposed correlation-reduced stochastic multi-objective gradient manipulation (CR-MOGM) to reduction those connection.
\[{ \lambda _k}\left( \xi  \right) = {\beta _k}{ \lambda _{k - 1}} + \left( {1 - {\beta _k}} \right){\lambda _k}\left( \xi  \right)\;\;s.t\;\;{\lambda _k}\left( \xi  \right) \in \mathop {\arg \min }\limits_{\lambda  \in {\Delta ^m}} \left\| {\nabla F\left( {{x_k};\xi } \right)\lambda } \right\|_2^2,\]
the update direction is ${{d_k}\left( \xi  \right) =  - \nabla F\left( {{x_k},\xi } \right){{ \lambda }_k\left( \xi  \right)}}$ and the parameter $\beta_k$ controls the dependence between ${{\lambda _k}\left( \xi  \right)}$ and ${d_k\left( \xi  \right)}$. By choosing $\beta_k$ close enough to 1, the correlation can be reduced, thus correcting the wrong direction generated by the stochastic gradient into a positive direction and allowing the convergence of this method.
\subsection{Stochastic mirror descent method}
\label{sec:smd}
Mirror descent is a first-order optimization method specified by a distance generating function ${\omega \left(  \cdot  \right):X \to \mathbb{R}}$ . ${\omega \left(  \cdot  \right)}$  is continuously differentiable and strongly convex modulus ${\eta }$ with respect to ${\left\|  \cdot  \right\|}$ , i.e.
\[{\left( {x' - x} \right)^T}\left( {\nabla \omega \left( {x'} \right) - \nabla \omega \left( x \right)} \right) \ge \eta {\left\| {x' - x} \right\|^2},\]
where ${{X^o} = \left\{ {x \in X|\partial \omega \left( x \right) \ne \phi } \right\}}$ is convex and always contains the relative interior of $X$. The nonnegative function ${V\left( { \cdot , \cdot } \right):{X^o} \times X \to {\mathbb{R}_ + }}$ called Bregman distance \cite{bregman1967relaxation} which is defined as follows:
\begin{align}\label{eq:V}
V\left( {x,u} \right) = \omega \left( u \right) - \left[ {\omega \left( x \right) + \nabla \omega {{\left( x \right)}^T}\left( {u - x} \right)} \right], \forall x \in {X^o},u \in X,
\end{align}
note that ${V\left( {x, \cdot } \right)}$ is also a strongly convex modulus $\eta$ with respect to the norm ${\left\|  \cdot  \right\|}$. We define the prox-mapping ${{P_x}\left(  \cdot  \right):{\mathbb{R}^n} \to {X^o}}$
associated with point ${x \in {X^o}}$ and function ${\omega \left( x \right)}$ as
\begin{align}\label{eq:pxy}
{P_x}\left( u \right) = \mathop {\arg \min }\limits_{u \in X} \left\{ {{u^T}\left( {u - x} \right) + V\left( {x,u} \right)} \right\}.
\end{align}
It is easy to see that this prox-mapping is well-defined, since the function $\omega$ is continuous on compact set $X$ and ${V\left( {x, \cdot } \right)}$ is strongly convex on $X^o$, the optimal of (\ref{eq:pxy}) is unique and belong to ${X^o}$.
Based on the aforementioned definition, we now present how the SMD method can be utilized to solve the convex-concave stochastic saddle point problem. Consider the following minimax (also known as saddle point) problem on nonempty bounded closed convex sets ${X \subset {\mathbb{R}^n}}$ and ${Y \subset {\mathbb{R}^m}}$ with stochastic vector ${\xi  \in \Xi  \subset {\mathbb{R}^n}}$,
\begin{align}\label{eq:smd_phi}
\mathop {\min }\limits_{x \in X} \mathop {\max }\limits_{y \in Y} \left\{ {\phi \left( {x,y} \right) = \mathbb{E}\left[ {\phi \left( {x,y;\xi } \right)} \right]} \right\}.
\end{align}
Assume that the function ${{\phi \left( {x,y;\xi } \right)}}$ is differentiable with respect to both $x$ and $y$, and the expectation ${{\phi \left( {x,y} \right)}}$ is convex in ${X \subset {\mathbb{R}^n}}$  and concave in ${Y \subset {\mathbb{R}^m}}$, we can obtain the gradient as
\[\nabla \phi \left( {x,y;\xi } \right) = \left[ {\begin{array}{*{20}{c}}
{\nabla {\phi _x}\left( {x,y;\xi } \right)}\\
{ - \nabla {\phi _y}\left( {x,y;\xi } \right)}
\end{array}} \right].\]
Specifically, we set ${{\omega _x}\left(  \cdot  \right):X \to \mathbb{R}}$ modulus ${\eta_x}$ with respect to ${{\left\|  \cdot  \right\|_x}}$ and ${{\omega _y}\left(  \cdot  \right):Y \to \mathbb{R}}$ modulus ${\eta_y}$ with respect to ${{\left\|  \cdot  \right\|_y}}$ be the distance generating functions on $X$ and $Y$ respectively. Followed by the definition of
\begin{align}\label{eq:D}
{D_{{\omega _x},X}} = \sqrt {\mathop {\max }\limits_{x \in X} {\omega _x}\left( x \right) - \mathop {\min }\limits_{x \in X} {\omega _x}\left( x \right)} ,
\end{align}
and simplifying the notation ${\left( {x,y} \right)}$ to ${z \in Z = X \times Y}$, we equip on $Z$ the norm
\[\left\| z \right\| = \left\| {\left( {x,y} \right)} \right\| = \sqrt {\frac{{{\eta _x}}}{{2D_{{\omega _x},X}^2}}\left\| x \right\|_x^2 + \frac{{{\eta _y}}}{{2D_{{\omega _y},Y}^2}}\left\| y \right\|_y^2}, \]
and its dual norm is
\[{\left\| z \right\|_*} = {\left\| {\left( {x,y} \right)} \right\|_*} = \sqrt {\frac{{2D_{{\omega _x},X}^2}}{{{\eta _x}}}\left\| x \right\|_{*,x}^2 + \frac{{2D_{{\omega _y},Y}^2}}{{{\eta _y}}}\left\| y \right\|_{*,y}^2} .\]
For ${z \in Z}$ we can get distance generating function ${{\omega _z}\left( z \right)}$ modulus $\eta_z = 1$ with respect to the norm ${\left\|  \cdot  \right\|}$ as:
\[{\omega _z}\left( z \right) = \frac{{{\omega _x}\left( x \right)}}{{2D_{{\omega _x},X}^2}} + \frac{{{\omega _y}\left( y \right)}}{{2D_{{\omega _y},Y}^2}}.\]
Drawing on Definition (\ref{eq:D}) and a straightforward computation, we can derive ${{D_{{\omega _z},Z}} = 1}$. By utilizing the aforementioned setup and definitions (\ref{eq:V}) and (\ref{eq:pxy}), we can write the Bergman distance associated with ${{\omega _z}\left( z \right)}$ on ${Z}$ as ${V\left( {z,u} \right):{Z^o} \times Z}$, along with the corresponding prox-mapping ${{P_z}\left(  \cdot  \right):{\mathbb{R}^{n + m}} \to {Z^o}}$.
The upper bound of dual norm of the gradient ${\nabla \phi \left( {x,y;\xi } \right)}$ can be assumed by two positive constants ${M_{*,x}^2}$ and ${M_{*,y}^2}$, i.e.
\begin{center}
$\mathbb{E}\left[ {\left\| {\nabla {\phi _x}\left( {x,y;\xi } \right)} \right\|_{*,x}^2} \right] \le M_{*,x}^2$ and $\mathbb{E}\left[ {\left\| {\nabla {\phi _y}\left( {x,y;\xi } \right)} \right\|_{*,y}^2} \right] \le M_{*,y}^2.$
\end{center}

Let
\[M_*^2 = \frac{{2D_{{\omega _x},X}^2}}{{{\eta _x}}}M_{*,x}^2 + \frac{{2D_{{\omega _y},Y}^2}}{{{\eta _y}}}M_{*,y}^2,\]
we have
\[\mathbb{E} {\left\| {\nabla \phi \left( {z;\xi } \right)} \right\|_{*}^2} \le M_*^2.\]

With step size ${{{\gamma_s}}}$ , we can defined the updates of SMD as:
\begin{align}\label{eq:iterative updates}
{z_{s + 1}} = {P_{{z_s}}}\left( {{\gamma_s}\nabla \phi \left( {{z_s};\xi } \right)} \right),
\end{align}
Similar to \cite{nemirovski2009robust}, we consider the search point $ z_P^S$ as a moving average of the approximate solutions, with ${s = P, \cdots ,S}$ yields
\begin{align}\label{eq:approximate solution}
{ z_{P}^S} = \frac{{\sum\limits_{s = P}^{S} {{\gamma_s}} {z_{s}}}}{{\sum\limits_{s = P}^{S} {{\gamma_s}} }}.
\end{align}
We refer to (\ref{eq:iterative updates}) and (\ref{eq:approximate solution}) as the SMD method. With a suitable distance generating function, assuming that the norm of $z$ is bounded, it can be obtained that the sequence generated by (\ref{eq:iterative updates}) and (\ref{eq:approximate solution}) diminishes the dual gap of the objective function at a sublinear rate.
Based on the fact that $V$ is strongly convex, the following lemma can be easily obtained from the optimality conditions.
\begin{lemma} (\cite{nemirovski2009robust} Lemma 2.1)
\label{lemma:1}
For any ${u \in X}$, $x \in {X^o}$ and ${y \in {\mathbb{R}^n}}$, one has
\begin{align}\label{eq:lemma1_1}
V\left( {{P_x}\left( y \right),u} \right) \le V\left( {x,u} \right) + {y^T}\left( {u - x} \right) + \frac{{\left\| y \right\|_*^2}}{{2\eta }}.
\end{align}
\end{lemma}
In the above setting, Nemirovski et al. in \cite{nemirovski2009robust} obtained the following convergence rate results.
\begin{lemma}\cite{nemirovski2009robust}
\label{lemma:epslion}
For any ${x \in \left[ X \right]}$, ${y \in \left[ Y \right]}$, the following inequality holds:
\[\begin{array}{l}
\mathbb{E}\left[ {\mathop {\max }\limits_{y \in Y} \phi \left( { x_1^S,y} \right) - \mathop {\min }\limits_{x \in X} \phi \left( {x, y_1^S} \right)} \right] \le {\left[ {\sum\limits_{s = 1}^S {{\gamma _s}} } \right]^{ - 1}}\mathbb{E}\left[ {2 + \frac{5}{2}M_*^2\sum\limits_{s = 1}^S {\gamma _s^2} } \right].
\end{array}\]
\end{lemma}
By employing different step size ${\gamma_s}$ settings, Nemirovski et al. \cite{nemirovski2009robust} achieved sublinear convergence rate results. The proofs for our MSMD method draw significant inspiration from their work. We then provide a comprehensive description of how we utilize the SMD method to solve the SMOO problem. Furthermore, we demonstrate that incorporating SMD iterations effectively addresses the challenge of unbiased estimation of subproblem solutions in SMOO proofs.
\section{Our Method}
\label{sec:ourmethod}
We propose a novel SMOO method based on the SMD method, named Multiple gradient Stochastic Mirror Descent method.
First we make some standard assumptions, consistent with those used in existing SMOO methods \cite{zhou2022convergence,liu2021stochastic,fernando2023mitigating}. The first assumption is that the objective function is lower bounded, differentiable and the gradient is $L-$Lipschitz continuous.
\begin{assumption}\label{ass_f}
For all objective functions ${{{f_i}\left( x_k \right)}}$, ${i \in [M]}$ and iterates ${k \in [K]}$, the following statements hold:

(a) ${{{f_i}({x_k})}}$ is bounded from below
\[\mathbb{E}\left[ {{f_i}({x_k})} \right] \ge {F_{\inf }} >  - \infty. \]

(b) ${f_i(x_k)}$ is differentiable at any point $x_k$.

(c) ${\nabla {f_i}\left( {{x_k}} \right)}$ is Lipschitz continuous with constant $L$.
\end{assumption}

This assumption ensures that the gradient does not change too rapidly, which is critical for controlling the convergence behavior of the algorithm. Similar to the setting in single-objective stochastic optimization \cite{hardt2016train, johnson2013accelerating}, for each objective ${i \in [m]}$ we assume that unbiased stochastic estimates of the objective gradients can be obtained.
For any objective gradient, we assume that the Euclidean norm and the variance of stochastic gradients are bounded by two positive constants, respectively.
\begin{assumption}\label{ass_gradient}
For all objective functions ${{{f_i}\left( x_k \right)}}$, ${i \in [m]}$, iterates ${k \in [K]}$ and i.i.d stochastic variable ${\xi  \in \Xi }$, we have access to the individual stochastic gradients ${{\nabla {f_i}\left( {{x_k},{\xi}} \right)}}$ which are unbiased estimates of ${{\nabla {f_i}\left( {{x_k}} \right)}}$. Specifically, the gradients satisfies the unbiasedness condition:
\begin{align*}
{\nabla {f_i}\left( {{x_k}} \right) = {\mathbb{E}}\left[ {\nabla {f_i}\left( {{x_k},\xi } \right)} \right]},
\end{align*}
Moreover, the Euclidean norm of the stochastic gradients are bounded by a positive constant ${{C_f\ge 0}}$ i.e.
\[\mathbb{E} {\left\| {\nabla {f_i}({x_k},\xi )} \right\|_2^2} \le C_f^2,\]
and the variance are also bounded by ${\delta \ge 0}$ i.e.
\begin{align}\label{eq:bound_var_f}
  \mathbb{E}{\left\| {\nabla {f_i}\left( {{x_k};\xi } \right) - \nabla {f_i}\left( {{x_k}} \right)} \right\|_2^2}  \le {\delta ^2}.
\end{align}
\end{assumption}
These assumptions ensure that the stochastic gradient estimates remain well-behaved in terms of their magnitude and variance, which are crucial for analyzing the convergence properties of the SMOO methods. By leveraging the upper bound of the stochastic gradient for each objective function ${{f_i}\left(  \cdot  \right)}$ in Assumption \ref{ass_gradient}, we can readily derive an upper bound for the stochastic gradient matrix and its variance
\begin{align}
\label{eq:bound_marx_F}
\mathbb{E}\left\| {\nabla {F_k}\left( {{\xi}} \right)} \right\|_2^2 \le mC_f^2 ~~\mbox{and}~~\mathbb{E}\left\| {\nabla {F_k}\left( {{\xi}} \right) - \nabla {F_k}} \right\|_2^2 \le m{\delta ^2}.
\end{align}


Recalling subsection \ref{sec:moo}, the goal of MOO is to identify a comment descent direction $d_k$, defined by (\ref{eq:dk_optimal}). Thus we can obtain an upper bound followed by
\begin{equation}\label{eq:dk_Pound}
{\left\| {{d_k}} \right\|_2} = {\left\| {\nabla F\left( {{x_k}} \right){\lambda _k}} \right\|_2} \le {\left\| {\sum\limits_{i = 1}^m {{\lambda _{k,i}}} \mathbb{E}\left[ {\nabla {f_i}\left( {{x_k};\xi } \right)} \right]} \right\|_2} \le \mathbb{E}\left[\sum\limits_{i = 1}^m {{\lambda _{k,i}}} {\left\| {\nabla {f_i}\left( {{x_k},\xi } \right)} \right\|_2}\right] \le C_f.
\end{equation}
where ${{\lambda _k} = \left( {{\lambda _{k,1}}, \cdots ,{\lambda _{k,m}}} \right)^T}$, ${{{\lambda _{k,i}} \in \mathbb{R}}}$ denotes the $i$-th element of $\lambda_k$. At this point, problem (\ref{eq:dk_1}) can be reformulated without loss of optimality as:
\begin{align}\label{eq:dk_c}
\mathop {\arg \min }\limits_{d \in C} \mathop {\max }\limits_{\lambda  \in {\Delta ^m}} \mathbb{E}\left[ {\phi \left( {d,\lambda ;\xi } \right)} \right],
\end{align}
where ${C = \left\{ {d \in {\mathbb{R}^n}|{{\left\| d \right\|}_2} \le {C_f}} \right\}}$. For any ${\left( {d,\lambda ,\xi } \right) \in C \times {\Delta ^m} \times \Xi }$, the gradients of the subproblem are given by:
\begin{align}\label{eq:grad}
\nabla \phi \left( {d, \lambda ;\xi } \right) = \left[ {\begin{array}{*{20}{c}}
{\nabla {\phi _d}\left( {d, \lambda ;\xi } \right)}\\
{-\nabla {\phi _\lambda }\left( { d,\lambda ;\xi } \right)}
\end{array}} \right] = \left[ {\begin{array}{*{20}{c}}
{\nabla {F_k}\left( \xi  \right) \lambda  +  d}\\
{ - \nabla {F_k}\left( \xi  \right) d}
\end{array}} \right],
\end{align}
under Assumption \ref{ass_gradient}, it follows that the gradient of the subproblem is unbiased, i.e., ${\mathbb{E}\left[ {\nabla \phi \left( { d, \lambda ;\xi } \right)} \right] = \nabla \phi \left( {d,\lambda } \right)}$.

In fact, the Bregman distance generating function plays a crucial role in the SMD method. The SMOO subproblem can be viewed as a saddle point problem that minimizes $d$ over a bounded set $C$ defined by Euclidean norms, while maximizes ${\lambda}$ over a standard simplex ${{\Delta ^m}}$. As suggested by \cite{nemirovski2009robust}, choosing the $\ell_1$-norm setup for $\lambda$ is considered valid. This setup not only expedites the convergence rate but also ensures the existence of an analytical solution for the approximate mapping of ${\lambda}$.
The generating function for the $\ell_1$-setup is given by:
\begin{align}\label{eq:wlamda}
{\omega _\lambda }\left( x \right) = \sum\limits_{i = 1}^m {{x_i}} \ln {x_i}:{\Delta ^m} \to \mathbb{R},
\end{align}
which corresponds to the modulus $\eta_\lambda = 1$.
The associated mirror descent step is straightforward to implement, with the Bregman distance function defined as:
\[{V_\lambda }\left( {x,y} \right) = \sum\limits_{i = 1}^n {{y_i}\ln \frac{{{y_i}}}{{{x_i}}}} .\]

For $d$, we choose the Euclidean norm as the generating function, given by:
\begin{align}\label{eq:wd}
{\omega _d}\left( x \right) = \frac{1}{2}\left\| x \right\|_2^2:C \to \mathbb{R},
\end{align}
with modulus $\eta_d = 1 $ with respect to $\|\cdot\|_2$.  Using this setup, the Bregman distance for ${{\omega _d}}$ can be expressed as:
\[{V_d}\left( {x,y} \right) = \frac{1}{2}\left\| {x - y} \right\|_2^2.\]
This setup allows us to implement ${{X^o} = X}$ and ${{Y^o} = Y}$ easily.
According to Definition (\ref{eq:D}), the constants of $d$ and $\lambda$ can be expressed as ${{D_{{\omega _d},C}} = {C_f}/\sqrt 2 }$ and ${D_{{\omega _\lambda },{\Delta ^m}}} = \sqrt {\ln m}$.
To establish a more rigorous framework, we equip $\mathbb{R}^n \times \mathbb{R}^m$ with the norm:
\begin{align}\label{eq:norm}
\left\| {\left( {d,\lambda } \right)} \right\| \le {\left( {\frac{{\left\| d \right\|_2^2}}{{2R_d^2}} + \frac{{\left\| \lambda  \right\|_1^2}}{{2R_\lambda ^2}}} \right)^{1/2}},
\end{align}
such that the dual norm can be defined as:
\begin{align}\label{eq:dual_norm}
{\left\| {\left( {d,\lambda } \right)} \right\|_*} \le {\left( {2R_d^2\left\| d \right\|_2^2 + 2R_\lambda ^2\left\| \lambda  \right\|_\infty ^2} \right)^{1/2}},
\end{align}
where
\begin{align*}
R_d^2 = \frac{{D_{{\omega _d},C}^2}}{{{\eta _d}}} = \frac{C_f^2}{2} ~~\mbox{and} ~~R_\lambda ^2 = \frac{{D_{{\omega _\lambda },{\Delta ^m}}^2}}{{{\eta _\lambda }}} = \ln m.
\end{align*}

Given that the initial point for $d$ in the set $C$ is selected as:
\[{d_{k,0}} = \mathop {{\mathop{\rm argmin}\nolimits} }\limits_{d \in C} d = {(0, \ldots ,0)^T},\]
and for $\lambda$ in ${{\Delta ^m}}$ as:
\[{\lambda _{k,0}} = \mathop {{\mathop{\rm argmin}\nolimits} }\limits_{\lambda  \in {\Delta ^m}} \lambda  = {m^{ - 1}}{(1, \ldots ,1)^T}.\]
We can get the iterative format of SMD in solving subproblem (\ref{eq:dk_1}).
At each iterative step $s\in \left[ S \right]$, with stochastic variables ${{\xi _s} \in \Xi }$ and step size $\gamma_s$, the iteration can be expressed as follows:
\begin{align*}
\left\{ {\begin{array}{*{20}{c}}
{{d_{k,s + 1}} = {P_{{d_{k,s}}}}\left( {{\gamma _s}\nabla {\phi _d}\left( {{d_{k,s}},{\lambda _{k,s}};{\xi _s}} \right)} \right)\mathop { = \arg \min }\limits_{d \in C} \left\{ {{\gamma _s}\nabla \phi {{\left( {{d_{k,s}},{\lambda _{k,s}};{\xi _s}} \right)}^T}\left( {d - {d_{k,s}}} \right) + \frac{1}{2}\left\| {d - {d_{k,s}}} \right\|_2^2} \right\},}\\
{{\lambda _{k,s + 1}} = {P_{{\lambda _{k,s}}}}\left( { - {\gamma _s}\nabla {\phi _\lambda }\left( {{d_{k,s}},{\lambda _{k,s}};{\xi _s}} \right)} \right)\mathop { = \arg \min }\limits_{\lambda  \in {\Delta ^m}} \left\{ { - {\gamma _s}\nabla {\phi _\lambda }{{\left( {{d_{k,s}},{\lambda _{k,s}};{\xi _s}} \right)}^T}\left( {\lambda  - {\lambda _{k,s}}} \right) + \sum\limits_{i = 1}^n {{\lambda _i}\ln \frac{{{\lambda _i}}}{{{\lambda _{k,s,i}}}}} } \right\},}
\end{array}} \right.
\end{align*}
where ${{\lambda _{k,s}} = \left( {{\lambda _{k,s,1}}, \cdots ,{\lambda _{k,s,m}}} \right)^T}$, ${{{\lambda _{k,s,i}} \in \mathbb{R}}}$ represents the $i$-th element of $\lambda_{k,s}$.
At this point, the proximal mapping for $d$ reduces to a projection onto the bounded set $C$, while the computation for $\lambda$ can be performed in ${{\rm O}(m)}$ operations based on the explicit formula, as shown below:
\begin{align*}
\left\{ {\begin{array}{*{20}{c}}
{{d_{k,s + 1}} = {\Pi _C}\left( {{d_{k,s}} - {\gamma _s}\nabla {\phi _d}\left( {{d_{k,s}},{\lambda _{k,s}}} \right)} \right),}\\
{{\lambda _{k,s + 1,i}} = \frac{{{\lambda _{k,s,i}}{e^{{{\left[ {{\gamma _s}\nabla {\phi _\lambda }\left( {{d_{k,s}},{\lambda _{k,s}}} \right)} \right]}_i}}}}}{{\sum\limits_{j = 1}^m {{\lambda _{k,s,j}}{e^{{{\left[ {{\gamma _s}\nabla {\phi _\lambda }\left( {{d_{k,s}},{\lambda _{k,s}}} \right)} \right]}_j}}}} }},i \in \left[ m \right].}
\end{array}} \right.
\end{align*}
Substituting the gradient (\ref{eq:grad}) into the iterative format above yields
\begin{align}\label{eq:dks}
{d_{k,s + 1}} = {\Pi _C}\left( {{d_{k,s}} - {\gamma_s}\left( {\nabla {F_k}\left( {{\xi _s}} \right){\lambda _{k,s}} + {d_{k,s}}} \right)} \right),
\end{align}
and
\begin{align}\label{eq:lamdaks}
{\lambda _{k,s + 1,i}} = \frac{{{\lambda _{k,s,i}}{e^{{\gamma _s}\nabla {f_i}\left( {{x_k};{\xi _s}} \right){d_{k,s}}}}}}{{\sum\limits_{j = 1}^m {{\lambda _{k,s,j}}{e^{{\gamma _s}\nabla {f_j}\left( {{x_k};{\xi _s}} \right){d_{k,s}}}}} }},i \in \left[ m \right].
\end{align}
Thus, the iterative updates for ${d _{k,s + 1}}$ and ${\lambda _{k,s + 1}}$ in our MSMD are derived.  Notably, the constraint set $C$, being a simple bounded region under the Euclidean norm, allows for easy computation of its projection. The more challenging simplex constraint can also be solved analytically using the update formula in (\ref{eq:lamdaks}). As a result, each iteration of the proposed MSMD method is computationally efficient, avoiding significant resource overhead, ensuring both convergence and simplicity in implementation.

Similar to the discourse in subsection \ref{sec:smd}, we define ${z = \left( {d,\lambda } \right)}$ in the set ${Z = C \times {\Delta ^m}}$, where $C$ is the bounded set for $d$, and ${{\Delta ^m}}$ is the simplex for $\lambda$. We can further define the Bergman function ${{\omega _z}\left( z \right)}$ on ${Z}$, the Bergman distance ${V\left( {z,u} \right):{Z^o} \times Z \to \mathbb{R}}$ and the prox-mapping ${{P_z}\left( \zeta  \right):{\mathbb{R}^{n + m}} \to {Z^o}}$, where ${{Z^o} = Z}$. At this point we have the modulus $\eta_z = 1$.
The gradient of $\nabla {\phi _z}\left( {{z_{k,s}};{\xi _s}} \right)$ can be expressed as:
\[\nabla {\phi _z}\left( {{z_{k,s}};{\xi _s}} \right) = \left[ {\begin{array}{*{20}{c}}
{\nabla {\phi _d}\left( {{d_{k,s}},{\lambda _{k,s}};{\xi _s}} \right)}\\
-{\nabla {\phi _\lambda }\left( {{d_{k,s}},{\lambda _{k,s}};{\xi _s}} \right)}
\end{array}} \right].\]
Thus, the update of variable $z$ can be written as:
\begin{align}\label{eq:zks}
{z_{k,s + 1}} = {P_{{z_{k,s}}}}\left( {{\gamma_s}\nabla {\phi _z}\left( {{z_{k,s}};{\xi _s}} \right)} \right) = \left[ {\begin{array}{*{20}{c}}
{{d_{k,s + 1}}}\\
{{\lambda _{k,s + 1}}}
\end{array}} \right].
\end{align}
This update formula provides a unified way of handling both the primal variable $d$ and the dual variable $\lambda$.
A weighted average of the iterations from $s=K$ to $s=S$ is utilized as the final iteration, formulated as follows:
\begin{align}\label{eq:avg_zks}
{ z_{k,P}^S} = \frac{{\sum\limits_{s = P}^{S} {{\gamma_s}} {z_{k,s}}}}{{\sum\limits_{s = P}^{S} {{\gamma_s}} }},
\end{align}
where ${ z_{k,P}^S = {\left[ {\begin{array}{*{20}{c}}
{{{\left( { d_{k,P}^S} \right)}^T}}&{{{\left( { \lambda _{k,P}^S} \right)}^T}}
\end{array}} \right]^T}}$
represents the moving average of ${{d_{k,s}}}$ and ${{\lambda_{k,s}}}$. Building on this, we have the iterative formulation of our method
\begin{align}\label{eq:update_xk}
{x_{k + 1}} = {x_k} + {\alpha _k}{{ d}_{k,P}^S},
\end{align} where ${\alpha _k}$ is the step size. Consequently, we summarize our MSMD in Algorithm \ref{alg:smoo}.
\begin{algorithm}[h]
\begin{normalsize}
\caption{\text {Multi-gradient Stochastic Mirror Descent}}
\begin{algorithmic} \label{alg:smoo}
\REQUIRE  bounded set ${C}$, ratio ${r \in \left( {0,1} \right)}$, number of iteration ${K}$, initial point ${{x_0} \in {\mathbb{R}^n}}$, initial update direction ${{d_{0,0}} \in C}$, initial weighting parameter ${{\lambda _{0,0}} \in {\Delta ^m}}$
\FOR{${k = 0, \cdots ,K-1}$}
\STATE Compute ${\alpha_k}$, $S$
\STATE ${P = \left\lceil {rS} \right\rceil }$
\FOR{${s = 0, \cdots ,S-1}$}
\STATE Compute ${\gamma_s}$
\STATE Sample data ${{{\xi _s}}}$
\STATE Update  ${{z_{k,s + 1}}}$ by (\ref{eq:zks})
\ENDFOR
\STATE Compute approximate solution ${ z_{k,P}^S}$ by (\ref{eq:avg_zks})
\STATE Update ${{x_{k + 1}} = {x_k} + {\alpha _k}{{ d}_{k,P}^S}}$
\ENDFOR
\ENSURE $x_K$
\end{algorithmic}
\end{normalsize}
\end{algorithm}

\section{Convergences}
\label{sec:convergences}
In this section, we present a comprehensive theoretical proof for our proposed MSMD method. We begin with the following lemma, which bounds the dual norms of the gradient ${{\nabla {\phi _z}\left( {{z_{k,s}};{\xi _s}} \right)}}$ for subproblem (\ref{eq:dk_1}), as well as the difference between the stochastic gradient and the full gradient ${{\nabla {\phi _z}\left( {{z_{k,s}};{\xi _s}} \right) - \nabla {\phi _z}\left( {{z_{k,s}}} \right)}}$.
\begin{lemma}
\label{lemma:bound}
Suppose Assumption \ref{ass_gradient} holds. For the norm ${\left\|  \cdot  \right\|}$ and dual norm ${{\left\|  \cdot  \right\|_*}}$ defined by (\ref{eq:norm}) and (\ref{eq:dual_norm}), we have the expectation which satisfy
\begin{align}\label{eq:bound_gra_phi}
\mathbb{E}\left\| {\nabla {\phi _z}\left( {{z_{k,s}};{\xi _s}} \right)} \right\|_*^2 \le \left( {2 + \ln m} \right)2mC_f^4,
\end{align}
and
\begin{align}\label{eq:bound_var_phi}
\mathbb{E}\left\| {\nabla {\phi _z}\left( {{z_{k,s}};{\xi _s}} \right) - \nabla {\phi _z}\left( {{z_{k,s}}} \right)} \right\|_*^2 \le \left( {1 + 2m\ln m} \right)C_f^2{\delta ^2}.
\end{align}
\end{lemma}
\begin{proof}
Setting ${\lambda  = {\lambda _{k,s}}}$ and ${d = {d_{k,s}}}$ in (\ref{eq:grad}), with Assumption \ref{ass_gradient} and the fact that ${\lambda  \in {\Delta ^m}}$, we can get
\begin{align*}
\mathbb{E}\left\| {\nabla {F_k}\left( \xi  \right){\lambda _{k,s}} + {d_{k,s}}} \right\|_2^2 &\le 2\mathbb{E}\left\| {\nabla {F_k}\left( \xi_s  \right){\lambda _{k,s}}} \right\|_2^2 + 2\mathbb{E}\left\| {{d_{k,s}}} \right\|_2^2\\
& \le 2\mathbb{E}\left\| {\sum\limits_{i = 1}^m {{\lambda _{k,s,i}}} \nabla {f_i}\left( {{x_k};\xi_s } \right)} \right\|_2^2 + 2\mathbb{E}\left\| {{d_{k,s}}} \right\|_2^2\\
& \le 2\mathbb{E}\sum\limits_{i = 1}^m {{\lambda _{k,s,i}}} \left\| {\nabla {f_i}\left( {{x_k};\xi_s } \right)} \right\|_2^2 + 2\mathbb{E}\left\| {{d_{k,s}}} \right\|_2^2\\
& \le 2C_f^2 + 2C_f^2\\
& = 4C_f^2,
\end{align*}
where the first inequality follow from the fact that ${\left\| {a + b} \right\|_2^2 \le 2\left\| a \right\|_2^2 + 2\left\| b \right\|_2^2}$, and the last inequality is by the bound in (\ref{eq:dk_Pound}). For ${\mathbb{E}\left\| {\nabla {\phi _\lambda }\left( {{\lambda _{k,s}},{d_{k,s}};\xi } \right)} \right\|_2^2}$ we have
\begin{align*}
\mathbb{E}\left\| {\nabla {F_k}\left( \xi_s  \right){d_{k,s}}} \right\|_\infty ^2 &\le \mathbb{E}\left\| {\nabla {F_k}\left( \xi_s  \right){d_{k,s}}} \right\|_2^2\\
& \le\mathbb{E} \left\| {\nabla {F_k}\left( \xi_s  \right)} \right\|_2^2\left\| {{d_{k,s}}} \right\|_2^2\\
& \le mC_f^2C_f^2\\
& \le {m}C_f^4.
\end{align*}
Based on the two inequalities above we can conclude that the dual norm of the stochastic gradient ${{\nabla {\phi _z}\left( {{z_{k,s}};{\xi _s}} \right)}}$ satisfies
\begin{align*}
\mathbb{E}\left\| {\nabla {\phi _z}\left( {{z_{k,s}};{\xi _s}} \right)} \right\|_*^2 &\le \mathbb{E}\left\| {\left[ {\begin{array}{*{20}{c}}
{\nabla {\phi _d}\left( {{z_{k,s}};{\xi _s}} \right)}\\
{-\nabla {\phi _\lambda }\left( {{z_{k,s}};{\xi _s}} \right)}
\end{array}} \right]} \right\|_*^2 \le \mathbb{E}\left\| {\left[ {\begin{array}{*{20}{c}}
{\nabla {F_k}\left( \xi_s  \right){\lambda _{k,s}} + {d_{k,s}}}\\
{ - \nabla {F_k}\left( \xi_s  \right){d_{k,s}}}
\end{array}} \right]} \right\|_*^2\\
& \le C_f^2\mathbb{E}\left\| {\nabla {F_k}\left( \xi_s  \right){\lambda _{k,s}} + {d_{k,s}}} \right\|_2^2 + 2\ln m\mathbb{E}\left\| {\nabla {F_k}\left( \xi_s  \right){d_{k,s}}} \right\|_\infty ^2\\
& \le C_f^24mC_f^2 + 2\ln m{m}C_f^4\\
& = \left( {2 + \ln m} \right)2{m}C_f^4,
\end{align*}
where the penultimate inequality is according to the dual norm on $Z$ defined in (\ref{eq:dual_norm}). Similarly, for the variance of ${{\nabla {\phi _z}\left( {{z_{k,s}};{\xi _s}} \right)}}$ under the same dual norm definition, we have
\begin{align*}
\mathbb{E}\left\| {\nabla {\phi _z}\left( {{z_{k,s}};{\xi _s}} \right) - \nabla {\phi _z}\left( {{z_{k,s}}} \right)} \right\|_*^2 &= \mathbb{E}\left\| {\left[ {\begin{array}{*{20}{c}}
{\nabla {\phi _d}\left( {{z_{k,s}};{\xi _s}} \right) - \nabla {\phi _d}\left( {{z_{k,s}}} \right)}\\
{ - \nabla {\phi _\lambda }\left( {{z_{k,s}};{\xi _s}} \right) + \nabla {\phi _\lambda }\left( {{z_{k,s}}} \right)}
\end{array}} \right]} \right\|_*^2\\
 &\le C_f^2\mathbb{E}\left\| {\nabla {\phi _d}\left( {{z_{k,s}};{\xi _s}} \right) - \nabla {\phi _d}\left( {{z_{k,s}}} \right)} \right\|_2^2 + 2\ln m\mathbb{E}\left\| {\nabla {\phi _\lambda }\left( {{z_{k,s}};{\xi _s}} \right) - \nabla {\phi _\lambda }\left( {{z_{k,s}}} \right)} \right\|_\infty ^2\\
& \le C_f^2\mathbb{E}\left\| {\nabla {F_k}\left( {{\xi _s}} \right){\lambda _{k,s}} + {d_{k,s}} - \nabla {F_k}{\lambda _{k,s}} - {d_{k,s}}} \right\|_2^2 + 2\ln m\mathbb{E}\left\| { - \nabla {F_k}\left( {{\xi _s}} \right){d_{k,s}} + \nabla {F_k}{d_{k,s}}} \right\|_\infty ^2\\
& = C_f^2\mathbb{E}\left\| {\left( {\nabla {F_k}\left( {{\xi _s}} \right) - \nabla {F_k}} \right){\lambda _{k,s}}} \right\|_2^2 + 2\ln m\mathbb{E}\left\| {\left( {\nabla {F_k}\left( {{\xi _s}} \right) - \nabla {F_k}} \right){d_{k,s}}} \right\|_\infty ^2\\
& \le C_f^2\sum\limits_{i = 1}^m {{\lambda _{k,s,i}}} \mathbb{E}\left\| {\nabla {f_i}\left( {{x_k};{\xi _s}} \right) - \nabla {f_i}\left( {{x_k}} \right)} \right\|_2^2 + 2\ln m\mathbb{E}\left\| {\left( {\nabla {F_k}\left( {{\xi _s}} \right) - \nabla {F_k}} \right){d_{k,s}}} \right\|_\infty ^2\\
& \le C_f^2{\delta ^2} + 2\ln m\left\| {{d_{k,s}}} \right\|_2^2\mathbb{E}\left\| {\nabla {F_k}\left( {{\xi _s}} \right) - \nabla {F_k}} \right\|_2^2\\
& \le C_f^2{\delta ^2} + 2\ln mC_f^2m{\delta ^2}\\
& = \left( {1 + 2m\ln m} \right)C_f^2{\delta ^2},
\end{align*}
where the last of these inequalities is due to the variance of the gradient matrix in (\ref{eq:bound_marx_F}) and the upper bound of $d_{k,s}$ in (\ref{eq:dk_Pound}), thereby completing the proof.
\end{proof}

Lemma \ref{lemma:bound} demonstrates that the gradient norm and variance of the subproblem in SMOO are bounded above by two constants. This upper bound is influenced by $C_f$ and $\delta$ outlined in Assumption \ref{ass_gradient} and increases proportionally with the number of objectives $m$. Before proving the main result of our work, it is necessary to introduce the following lemma.

\begin{lemma} (\cite{nemirovski2009robust} Lemma 6.1)
\label{lemma:2}
Let ${{\zeta _1}, \cdots ,{\zeta _S}}$ be a sequence of elements of ${{\mathbb{R}^{n + m}}}$. Define the sequence ${{v_s}}$, ${s \in \left[ S \right]}$ in ${{Z^o}}$ as follows: ${{v_1} \in {Z^o}}$ and
\[{v_{s + 1}} = {P_{{v_s}}}\left( {{\zeta _s}} \right),1 \le s \le S.\]
Then, for any ${u \in Z}$, the following holds:
\begin{align}\label{eq:lemma2}
\sum\limits_{s = 1}^S {\zeta _s^T\left( {{v_s} - u} \right)}  \le V\left( {{v_1},u} \right) + \frac{1}{{2\eta }}\sum\limits_{s = 1}^S {\left\| {{\zeta _s}} \right\|_*^2}.
\end{align}
\end{lemma}

To establish a precise upper bound for the Bergman distance, we utilize the continuous differentiability of the Bergman function ${{\omega _z}\left( z \right)}$, We define the bound as follows:

\begin{equation}\label{eq:barD}
\begin{aligned}
{{\bar D}_{{\omega _z},Z}}&: = \sqrt 2 \mathop {\sup }\limits_{u \in Z,v \in Z} {\left[ {\omega \left( v \right) - \omega \left( u \right) - {{\left( {v - u} \right)}^T}\nabla \omega \left( u \right)} \right]^{1/2}}\\
 &= \mathop {\sup }\limits_{u \in Z,v \in Z} \sqrt {2V\left( {u, v} \right)}.
\end{aligned}
\end{equation}
This leads to the conclusion that for any ${u,v \in Z}$,
\begin{align*}
V\left( {u,v} \right) \le \frac{1}{2}\bar D_{{\omega _z},Z}^2 <  + \infty.
\end{align*}
This inequality implies that the Bergman distance ${V\left( {u,v} \right)}$ is bounded, which is crucial for the convergence proof of MSMD. By ensuring that the distance between any two points in $Z$ remains finite, we can analyze the stability and convergence behavior of the iterative updates generated by our method.

\begin{theorem}\label{th1}
Suppose that Assumptions \ref{ass_f} and \ref{ass_gradient} hold. Let
\begin{align}\label{eq:m*}
M_*^2 = \left( {2 + \ln m} \right){m^2}C_f^4 + \left( {\frac{1}{2} + \ln m} \right){m^2}C_f^2{\delta ^2},
\end{align}
the iteration ${{{ z}_{k,P}^S} = {\left[ {\begin{array}{*{20}{c}}
{ ( d_{k,P}^{S})^T}&{ ( \lambda _{k,P}^S)^T}
\end{array}} \right]^T}}$ generated from (\ref{eq:avg_zks}) satisfies the following
\begin{align}\label{eq:th_dk-dks}
\frac{1}{2}\mathbb{E}\left\| {{d_k} - {{ d}_{k,P}^S}} \right\|_2^2 \le {\left[ {\sum\limits_{s = P}^{S } {{\gamma_s}} } \right]^{ - 1}}\left[ {\bar D_{{\omega _z},Z}^2 + M_*^2\sum\limits_{s = P}^{S} {\gamma_s^2} } \right].
\end{align}

\end{theorem}
\begin{proof}
Considering that ${\phi \left( {d,\lambda_k } \right)}$ is strongly convex with respect to $d$ and ${{d_k}}$ is the optimal solution of Problem (\ref{eq:dk}), according to the first-order optimality condition for any ${d \in C}$ we have
\begin{align}\label{eq:d_dk>=0}
\nabla {\phi _d}\left( {{d_k},{\lambda _k}} \right)^T\left( {d - {d_k}} \right) \ge 0.
\end{align}
Recalling the generation of ${{{{ \lambda }_{k,P}^S}}}$ and ${{{{ d}_{k,P}^S}}}$ in (\ref{eq:avg_zks}) , we can get
\begin{equation}
\begin{aligned}\label{eq:th1_7}
\frac{1}{2}\left\| {{d_k} -  d_{k,P}^S} \right\|_2^2 &\le \frac{1}{2}\left\| {{d_k} -  d_{k,P}^S} \right\|_2^2 + \nabla {\phi _d}\left( {{d_k},{\lambda _k}} \right)^T\left( { d_{k,P}^S - {d_k}} \right)\\
& \le \phi \left( { d_{k,P}^S,{\lambda _k}} \right) - \phi \left( {{d_k},{\lambda _k}} \right)\\
& = \phi \left( { d_{k,P}^S,{\lambda _k}} \right) - \mathop {\max }\limits_{\lambda  \in {\Delta ^m}} \phi \left( {{d_k},\lambda } \right)\\
& \le \phi \left( { d_{k,P}^S,{\lambda _k}} \right) - \phi \left( {{d_k}, \lambda _{k,P}^S} \right)\\
& \le {\left[ {\sum\limits_{s = P}^S {{\gamma _s}} } \right]^{ - 1}}\left[ {\sum\limits_{s = P}^S {{\gamma _s}\phi \left( {{d_{k,s}},{\lambda _k}} \right) - \sum\limits_{s = P}^S {{\gamma _s}\phi \left( {{d_k},{\lambda _{k,s}}} \right)} } } \right],
\end{aligned}
\end{equation}
where the first inequality follows from (\ref{eq:d_dk>=0}) and the fact that ${ d_{k,P}^S \in C}$, the last one uses the convexity-concavity of function ${\phi}$. The first equality is directly obtained from the definition of $\lambda_k$ in (\ref{eq:moo_optimal}).
Furthermore, by the convexity of ${\phi \left( { \cdot ,{\lambda _k}} \right)}$ and concavity of ${\phi \left( {{d_k}, \cdot } \right)}$, we can derive the following inequalities:
\[{\phi \left( {{d_{k,s}},{\lambda _{k}}} \right) - \phi \left( {{d_k},{\lambda _k}} \right) \le {{\left( {{d_{k,s}} - {d_k}} \right)}^T}\nabla {\phi _d}\left( {{d_{k,s}},{\lambda _{k,s}}} \right)},\]
and a similar inequality concerning ${\lambda_k}$:
\[{\phi \left( {{d_{k}},{\lambda _k}} \right) - \phi \left( {{d_{k}},{\lambda _{k,s}}} \right) \le {{\left( {{\lambda _k} - {\lambda _{k,s}}} \right)}^T}\nabla {\phi _\lambda }\left( {{d_{k,s}},{\lambda _{k,s}}} \right)}.\]
Adding these two inequalities together provides
\begin{equation}
\begin{aligned}\label{eq:th1_5}
\phi \left( {{d_{k,s}},{\lambda _k}} \right) - \phi \left( {{d_k},{\lambda _{k,s}}} \right) &\le {\left( {{d_{k,s}} - {d_k}} \right)^T}\nabla {\phi _d}\left( {{d_{k,s}},{\lambda _{k,s}}} \right) + {\left( {{\lambda _k} - {\lambda _{k,s}}} \right)^T}\nabla {\phi _\lambda }\left( {{d_{k,s}},{\lambda _{k,s}}} \right)\\
& = {\left( {{z_{k,s}} - {z_k}} \right)^T}\nabla {\phi _z}\left( {{z_{k,s}}} \right).
\end{aligned}
\end{equation}
Applying Lemma \ref{lemma:1} with ${x = {z_{k,s}}}$, ${y = \gamma_s \nabla {\phi _z}\left( {{z_{k,s}};{\xi _s}} \right)}$, $u=z_k$ and ${{\eta_z  = 1}}$, recalling that ${{z_{k,s + 1}} = {P_{{z_{k,s}}}}\left( {{\gamma_s}\nabla {\phi _z}\left( {{z_{k,s}};{\xi _s}} \right)} \right)}$
we drive the inequality:
\[{\gamma_s}{\left( {{z_{k,s}} - {z_k}} \right)^T}\nabla {\phi _z}\left( {{z_{k,s}};{\xi _s}} \right) \le V\left( {{z_{k,s},z_k}} \right) - V\left( {{z_{k,s + 1}},z_k} \right) + \frac{{\gamma_s^2}}{2}\left\| {\nabla {\phi _z}\left( {{z_{k,s}};{\xi _s}} \right)} \right\|_*^2.\]
Subtracting ${{\gamma_s}{\left( {{z_{k,s}} - {z_k}} \right)^T}\left( {\nabla {\phi _z}\left( {{z_{k,s}};{\xi _s}} \right) - \nabla {\phi _z}\left( {{z_{k,s}}} \right)} \right)}$ on both sides yields
\begin{equation}
\begin{aligned} \label{eq:th_1_3}
{\gamma_s}{\left( {{z_{k,s}} - {z_k}} \right)^T}\nabla {\phi _z}\left( {{z_{k,s}}} \right) &\le V\left( {{z_{k,s},z_k}} \right) - V\left( {{z_{k,s + 1},z_k}} \right) + \frac{{\gamma_s^2}}{2}\left\| {\nabla {\phi _z}\left( {{z_{k,s}};{\xi _s}} \right)} \right\|_*^2\\
& - {\gamma_s}{\left( {{z_{k,s}} - {z_k}} \right)^T}\left( {\nabla {\phi _z}\left( {{z_{k,s}};{\xi _s}} \right) - \nabla {\phi _z}\left( {{z_{k,s}}} \right)} \right).
\end{aligned}
\end{equation}
Then, we apply Lemma \ref{lemma:2} with initial point ${{v_P} = {z_{k,P}}}$ and ${{\zeta _s} =  - {\gamma_s}\left( {\nabla {\phi _z}\left( {{z_{k,s}};{\xi _s}} \right) - \nabla {\phi _z}\left( {{z_{k,s}}} \right)} \right)}$, leading to:
\begin{align}\label{eq:th1_1}
\sum\limits_{s = P}^{S} {{\gamma_s}{{\left( {{z_k} - {v_s}} \right)}^T}\left( {\nabla {\phi _z}\left( {{z_{k,s}};{\xi _s}} \right) - \nabla {\phi _z}\left( {{z_{k,s}}} \right)} \right)}  \le V\left( {{z_{k,P}},{z_k}} \right) + \sum\limits_{s = P}^{S} {\frac{{\gamma_s^2}}{2}\left\| {\nabla {\phi _z}\left( {{z_{k,s}};{\xi _s}} \right) - \nabla {\phi _z}\left( {{z_{k,s}}} \right)} \right\|_*^2}.
\end{align}
Taking the expectation of (\ref{eq:th1_1}) on both side and plug (\ref{eq:bound_var_phi}) into it we have
\begin{equation}
\begin{aligned}\label{eq:th1_2}
\mathbb{E}\left[ {\sum\limits_{s = P}^{S} {{\gamma_s}{{\left( {z_k - {v_s}} \right)}^T}\left( {\nabla {\phi _z}\left( {{z_{k,s}};{\xi _s}} \right) - \nabla {\phi _z}\left( {{z_{k,s}}} \right)} \right)} } \right]  \le {V\left( {{z_{k,P}}, z_k } \right)} + \left( {\frac{1}{2} + m\ln m} \right)C_f^2{\delta ^2}\sum\limits_{s = P}^{S} {\gamma_s^2}.
\end{aligned}
\end{equation}
Summing up (\ref{eq:th_1_3}) for ${s = P, \cdots ,S}$, we can get

\begin{equation}
\begin{aligned}\label{eq:th1_4}
\sum\limits_{s = P}^{S } {{\gamma_s}{{\left( {{z_{k,s}} - z_k} \right)}^T}\nabla {\phi _z}\left( {{z_{k,s}}} \right)}  &\le V\left( {{z_{k,P}},z_k} \right) - V\left( {{z_{k,S+1}},z_k} \right) + \sum\limits_{s = P}^{S} {\frac{{\gamma_s^2}}{2}\left\| {\nabla {\phi _z}\left( {{z_{k,s}};{\xi _s}} \right)} \right\|_*^2} \\
& - \sum\limits_{s = P}^{S } {{\gamma_s}{{\left( {{z_{k,s}} - z_k} \right)}^T}\left( {\nabla {\phi _z}\left( {{z_{k,s}};{\xi _s}} \right) - \nabla {\phi _z}\left( {{z_{k,s}}} \right)} \right)} \\
& \le V\left( {{z_{k,P}},z_k} \right) + \sum\limits_{s = P}^{S } {\frac{{\gamma_s^2}}{2}\left\| {\nabla {\phi _z}\left( {{z_{k,s}};{\xi _s}} \right)} \right\|_*^2}  - \sum\limits_{s = P}^{S } {{\gamma_s}{{\left( {{z_{k,s}} - {v_s}} \right)}^T}\left( {\nabla {\phi _z}\left( {{z_{k,s}};{\xi _s}} \right) - \nabla {\phi _z}\left( {{z_{k,s}}} \right)} \right)} \\
& + \sum\limits_{s =P}^{S } {{\gamma_s}{{\left( {z_k -{v_s}} \right)}^T}\left( {\nabla {\phi _z}\left( {{z_{k,s}};{\xi _s}} \right) - \nabla {\phi _z}\left( {{z_{k,s}}} \right)} \right)},
\end{aligned}
\end{equation}
where the last inequality follows from the fact that ${V\left( {{z_{k,S+1}},{z_k}} \right) \ge 0}$.
For the giving ${{\xi _{[s - 1]}} = \left( {{\xi _1}, \cdots ,{\xi _{s - 1}}} \right)}$ and that the conditional expectation of ${{\nabla {\phi _z}\left( {{z_{k,s}};{\xi _s}} \right) - \nabla {\phi _z}\left( {{z_{k,s}}} \right)}}$, we can conclude that both  $z_{k,s}$ and $v_s$ are deterministic functions. At this point, we conclude that ${\mathbb{E}\left[ {{{\left( {{z_{k,s}} - {v_s}} \right)}^T}\left( {\nabla {\phi _z}\left( {{z_{k,s}};{\xi _s}} \right) - \nabla {\phi _z}\left( {{z_{k,s}}} \right)} \right)} \right] = 0}$. Take expectations on both sides of (\ref{eq:th1_4}), and substituting (\ref{eq:bound_gra_phi}) and (\ref{eq:th1_2}) into it yields

\begin{equation}
\begin{aligned}\label{eq:barDin}
\mathbb{E}\left[ {\sum\limits_{s = P}^{S } {{\gamma_s}{{\left( {{z_{k,s}} - {z_k}} \right)}^T}\nabla {\phi _z}\left( {{z_{k,s}}} \right)} } \right] &\le V\left( {{z_{k,P}},{z_k}} \right) + \sum\limits_{s = P}^{S} {\frac{{\gamma_s^2}}{2}\mathbb{E}\left\| {\nabla {\phi _z}\left( {{z_{k,s}};{\xi _s}} \right)} \right\|_*^2} \\
& + \sum\limits_{s = P}^{S } {{\gamma_s}\mathbb{E}\left[ {{{\left( {{z_k} - {v_s}} \right)}^T}\left( {\nabla {\phi _z}\left( {{z_{k,s}};{\xi _s}} \right) - \nabla {\phi _z}\left( {{z_{k,s}}} \right)} \right)} \right]} \\
 &\le V\left( {{z_{k,P}},{z_k}} \right) + \sum\limits_{s = P}^{S} {\frac{{\gamma_s^2}}{2}\mathbb{E}\left\| {\nabla {\phi _z}\left( {{z_{k,s}};{\xi _s}} \right)} \right\|_*^2}  + V\left( {{z_{k,P}},{z_k}} \right) + \left( {\frac{1}{2} + m\ln m} \right)C_f^2{\delta ^2}\sum\limits_{s = P}^{S } {\gamma_s^2} \\
& \le 2V\left( {{z_{k,P}},{z_k}} \right) + \left( {1 + \ln m} \right)mC_f^4\sum\limits_{s = P}^{S} {\gamma_s^2}  + \left( {\frac{1}{2} + m\ln m} \right)C_f^2{\delta ^2}\sum\limits_{s = P}^{S } {\gamma_s^2},
\end{aligned}
\end{equation}
By plugging the upper bound of Bergman distance (${\ref{eq:barD}}$) into (\ref{eq:barDin}) we can get
\begin{align*}
\mathbb{E}\left[ {\sum\limits_{s = P}^S {{\gamma _s}{{\left( {{z_{k,s}} - {z_k}} \right)}^T}\nabla {\phi _z}\left( {{z_{k,s}}} \right)} } \right] \le \bar D_{{\omega _z},Z}^2 + \underbrace {\left( {\left( {2 + \ln m} \right){m^2}C_f^4 + \left( {\frac{1}{2} + \ln m} \right){m^2}C_f^2{\delta ^2}} \right)}_{ = M_*^2}\sum\limits_{s = P}^S {\gamma _s^2} ,
\end{align*}
where the constant part of the second term on the right-hand side of the inequality is denoted as $M^*$, we can simplify to have
\[\mathbb{E}\left[ {\sum\limits_{s = P}^{S } {{\gamma_s}{{\left( {{z_{k,s}} - {z_k}} \right)}^T}\nabla {\phi _z}\left( {{z_{k,s}}} \right)} } \right] \le \bar D_{{\omega _z},Z}^2 + M_*^2\sum\limits_{s = P}^{S } {\gamma_s^2}. \]
Then implies the aforementioned inequality, summing up (\ref{eq:th1_5}) for $s=P$ to $s=S$, multiplying both sides by $\gamma_s$ and taking expectation on both sides and rearranging yields

\begin{align}\label{eq:th1_6}
\sum\limits_{s = P}^S {\mathbb{E}\left[ {{\gamma _s}\phi \left( {{d_{k,s}},{\lambda _k}} \right) - {\gamma _s}\phi \left( {{d_k},{\lambda _{k,s}}} \right)} \right]}  \le \bar D_{{\omega _z},Z}^2 + M_*^2\sum\limits_{s = P}^S {\gamma _s^2} .
\end{align}
Plugging in (\ref{eq:th1_6}) back to  (\ref{eq:th1_7}) implies that (\ref{eq:th_dk-dks}) holds, completes the proof.
\end{proof}

Theorem \ref{th1} establishes a relationship between the update direction ${{ d}_{k,P}^S}$ in MSMD and the multi-gradient direction $d_k$ under the full gradient. The distance between these two directions depends on the choice of the inner step size $\gamma_s$. Additionally, by incorporating the iterative format of $x_{k+1}$ in (\ref{eq:update_xk}), we can further derive a relationship between this distance and the expectation norm of $d_k$.

\begin{theorem}\label{th:dk2}
Suppose that Assumptions \ref{ass_f} and \ref{ass_gradient} hold, the iteration ${x_k}$ generated from Algorithm \ref{alg:smoo} satisfies
\begin{align}\label{eq:th2}
\sum\limits_{k = 0}^{K - 1} \mathbb{E}{\left( {\frac{{{\alpha _k}}}{2} - L\alpha _k^2} \right)\left\| {{d_k}} \right\|_2^2}  \le \sum\limits_{k = 0}^{K - 1} {\left( {{\alpha _k}{C_f}{{\left( {\mathbb{E}\left\| {{{ d_{k,P}^S}} - {d_k}} \right\|_2^2} \right)}^{\frac{1}{2}}} + L\alpha _k^2\mathbb{E}\left\| {{{ d_{k,P}^S}} - {d_k}} \right\|_2^2} \right)}  + {M_f},
\end{align}
where
\[{M_f} = \mathop {\max }\limits_{i \in \left[ m \right]} {f_i}\left( {{x_0}} \right) - {F_{\inf }}.\]
\end{theorem}
\begin{proof}

Recalling that ${{x_{k + 1}} = {x_k} + {\alpha _k}{{ d}_{k,P}^S}}$ and L-smoothness of each objective function, we have that for any ${i \in \left[ m \right]}$,
\begin{equation}\label{eq:th2_f}
\begin{aligned}
{f_i}\left( {{x_{k + 1}}} \right) &= {f_i}\left( {{x_{k + 1}}} \right) - {f_i}\left( {{x_k}} \right) + {f_i}\left( {{x_k}} \right)\\
& \le \mathop {\max }\limits_{i \in \left[ m \right]} \left\{ {{f_i}\left( {{x_{k + 1}}} \right) - {f_i}\left( {{x_k}} \right)} \right\} + {f_i}\left( {{x_k}} \right)\\
& \le \mathop {\max }\limits_{i \in \left[ m \right]} \left\langle {\nabla {f_i}\left( {{x_k}} \right),{\alpha _k}{{ d}_{k,P}^S}} \right\rangle  + \frac{L}{2}\left\| {{\alpha _k}{{ d_{k,P}^S}}} \right\|_2^2 + {f_i}\left( {{x_k}} \right)\\
& = \mathop {\max }\limits_{i\in \left[ m \right]} \left\langle {\nabla {f_i}\left( {{x_k}} \right),{\alpha _k}{d_k}} \right\rangle  + \mathop {\max }\limits_{i\in \left[ m \right]} \left\langle {\nabla {f_i}\left( {{x_k}} \right),{\alpha _k}{{ d_{k,P}^S}} - {\alpha _k}{d_k}} \right\rangle \\
& + \frac{L}{2}\left\| {{\alpha _k}{{ d_{k,P}^S}} - {\alpha _k}{d_k} + {\alpha _k}{d_k}} \right\|_2^2 + {f_i}\left( {{x_k}} \right)\\
& \le {\alpha _k}\mathop {\max }\limits_{i\in \left[ m \right]} \left\langle {\nabla {f_i}\left( {{x_k}} \right),{d_k}} \right\rangle  + {\alpha _k}\mathop {\max }\limits_{i\in \left[ m \right]} \left\langle {\nabla {f_i}\left( {{x_k}} \right),{{ d_{k,P}^S}} - {d_k}} \right\rangle \\
& + L\alpha _k^2\left\| {{{ d_{k,P}^S}} - {d_k}} \right\|_2^2 + L\alpha _k^2\left\| {{d_k}} \right\|_2^2 + {f_i}\left( {{x_k}} \right)\\
 &\le  - \frac{{{\alpha _k}}}{2}\left\| {{d_k}} \right\|_2^2 + L\alpha _k^2\left\| {{d_k}} \right\|_2^2 + L\alpha _k^2\left\| {{{ d_{k,P}^S}} - {d_k}} \right\|_2^2\\
& + {\alpha _k}\mathop {\max }\limits_{i\in \left[ m \right]} \left\langle {\nabla {f_i}\left( {{x_k}} \right),{{ d_{k,P}^S}} - {d_k}} \right\rangle + {f_i}\left( {{x_k}} \right),
\end{aligned}
\end{equation}
where the last inequality follows directly from \cite{fliege2000steepest} by definition (\ref{eq:dk}) of ${d_k}$. Further, from Cauchy-Schwarz inequality, the inner product term of the last inequality above yields
\[{\left( {\mathop {\max }\limits_{i \in \left[ m \right]} \left\langle {\nabla {f_i}\left( {{x_k}} \right), d_{k,P}^S - {d_k}} \right\rangle } \right)^2} \le \mathop {\max }\limits_{i \in \left[ m \right]} \left\| {\nabla {f_i}\left( {{x_k}} \right)} \right\|_2^2\left\| { d_{k,P}^S - {d_k}} \right\|_2^2,\]
taking expectation on both sides, we can get
\begin{align*}
{\left( {\mathbb{E}\left[ {\mathop {\max }\limits_{i\in \left[ m \right]} \left\langle {\nabla {f_i}\left( {{x_k}} \right),{{ d_{k,P}^S}} - {d_k}} \right\rangle } \right]} \right)^2} & \le \mathbb{E}\left[ {\mathop {\max }\limits_{i \in \left[ m \right]} \left\| {\nabla {f_i}\left( {{x_k}} \right)} \right\|_2^2\left\| { d_{k,P}^S - {d_k}} \right\|_2^2} \right]\\
& \le C_f^2\mathbb{E}\left\| {{{ d_{k,P}^S}} - {d_k}} \right\|_2^2.
\end{align*}
Substituting the above inequality back into (\ref{eq:th2_f}) and take total expectation, we have
\begin{align}\label{eq:th_expf}
\mathbb{E}\left[ {{f_i}\left( {{x_{k + 1}}} \right)} \right] \le \left( {L\alpha _k^2 - \frac{{{\alpha _k}}}{2}} \right)\mathbb{E}\left\| {{d_k}} \right\|_2^2 + {\alpha _k}{C_f}{\left( {\mathbb{E}\left\| {{{ d_{k,P}^S}} - {d_k}} \right\|_2^2} \right)^{\frac{1}{2}}} + L\alpha _k^2\mathbb{E}\left\| {{{ d_{k,P}^S}} - {d_k}} \right\|_2^2 + \mathbb{E}\left[ {{f_i}\left( {{x_k}} \right)} \right].
\end{align}
Rearranging and taking telescoping sum on both sides of the (\ref{eq:th_expf}) gives
\begin{align*}
\sum\limits_{k = 0}^{K - 1} {\left( {\frac{{{\alpha _k}}}{2} - L\alpha _k^2} \right)\mathbb{E}\left\| {{d_k}} \right\|_2^2} &  \le \sum\limits_{k = 0}^{K - 1} {\left( {{\alpha _k}{C_f}{{\left( {\mathbb{E}\left\| {{{ d_{k,P}^S}} - {d_k}} \right\|_2^2} \right)}^{\frac{1}{2}}} + L\alpha _k^2\mathbb{E}\left\| {{{ d_{k,P}^S}} - {d_k}} \right\|_2^2} \right)}  + {f_i}\left( {{x_0}} \right) - {f_i}\left( {{x_K}} \right)\\
& \le \sum\limits_{k = 0}^{K - 1} {\left( {{\alpha _k}{C_f}{{\left( {\mathbb{E}\left\| {{{ d_{k,P}^S}} - {d_k}} \right\|_2^2} \right)}^{\frac{1}{2}}} + L\alpha _k^2\mathbb{E}\left\| {{{ d_{k,P}^S}} - {d_k}} \right\|_2^2} \right)}  + \mathop {\max }\limits_{i \in \left[ m \right]} {f_i}\left( {{x_0}} \right) - {F_{\inf }},
\end{align*}
let ${{M_f} = \mathop {\max }\limits_{i\in \left[ m \right]} {f_i}\left( {{x_0}} \right) - {F_{\inf }}}$ we have (\ref{eq:th2}) satisfies, this completes the proof.
\end{proof}
By Combining Theorem \ref{th1} and Theorem \ref{th:dk2} and selecting different outer step size $\alpha_k$ and inner step size $\gamma_s$, we can derive the following convergence results.
\begin{corollary}
\label{cor:fixa_fixg}
Set fixed inner step-size ${{\gamma_s} = \frac{\theta}{{{M_*}\sqrt S }}}$, ${\theta>0}$, fixed outer step size ${{\alpha _k} = \alpha  \le \frac{1}{{4L}}}$  and ${P = \left\lceil {rS} \right\rceil }$, ${r \in \left( {0,1} \right)}$ in Theorems \ref{th1} and \ref{th:dk2}.
Let ${{M_*}}$ satisfies (\ref{eq:m*}) and
\begin{align}\label{eq:M1}
{M_1} = \left( {\frac{{2\bar D_{{\omega _z},Z}^2S}}{{S - \left\lceil {rS} \right\rceil  + 1}} + 2} \right)\max \left\{ {\theta ,{\theta ^{ - 1}}} \right\}{M_*},
\end{align}
we have

(a) If we choose ${{S} = {\left( {K + 1} \right)^2}}$, the iteration ${x_k}$ generated from the proposed MSMD method satisfies
\begin{align}\label{eq:cor1_K2}
\frac{1}{K}\sum\limits_{k = 0}^{K - 1}\mathbb{E} {\left\| {{d_k}} \right\|_2^2}  \le {\rm{O}}\left( {4{C_f}{{M_1}^{\frac{1}{2}}}\frac{1}{{\sqrt K }}} \right).
\end{align}
(b) If we choose ${{S} = {(k+1)^2}}$ and total iteration step $K \le 4$, the iteration ${x_k}$ generated from the proposed MSMD method satisfies
\begin{align}\label{eq:cor1_k2}
\frac{1}{K}\sum\limits_{k = 0}^{K - 1} \mathbb{E} \left\| {{d_k}} \right\|_2^2 \le {\rm{O}}\left( {\max \left\{ {8{C_f}{M_1}^{\frac{1}{2}},4L\alpha {M_1}} \right\}\frac{1}{{\sqrt K }}} \right).
\end{align}
\end{corollary}

\begin{proof}
We begin the proof by revisiting the result Theorems \ref{th1} and \ref{th:dk2}. Setting ${{\gamma_s} = \frac{\theta }{{{M_*}\sqrt S }}}$ in (\ref{eq:th_dk-dks}) we obtain

\begin{equation}\label{eq:cor1_dks}
\begin{aligned}
\mathbb{E}\left\| {{d_k} -  d_{k,P}^S} \right\|_2^2 &\le 2{\left[ {\sum\limits_{s = P}^S {{\gamma_s}} } \right]^{ - 1}}\left[ {\bar D_{{\omega _z},Z}^2 + M_*^2\sum\limits_{s = P}^{S} {\gamma_s^2} } \right]\\
& \le 2{\left[ {\frac{{\theta \left( {S - P + 1} \right)}}{{{M_*} \sqrt{S}}}} \right]^{ - 1}}\left[ {\bar D_{{\omega _z},Z}^2 + \frac{{{\theta ^2}\left( {S - P + 1} \right)}}{S}} \right]\\
& \le \frac{{{M_*}}}{{\sqrt S }}\left[ {\frac{{2\bar D_{{\omega _z},Z}^2S}}{{\theta \left( {S - P + 1} \right)}} + 2\theta } \right]\\
& \le \max \left\{ {\theta ,{\theta ^{ - 1}}} \right\}\frac{{{M_*}}}{{\sqrt S }}\left[ {\frac{{2\bar D_{{\omega _z},Z}^2S}}{{S - P + 1}} + 2} \right]\\
 &\le \frac{{{M_1}}}{{\sqrt S }}.
\end{aligned}
\end{equation}

where the last inequality can be obtained by direct substitution of ${P = \left\lceil {rS} \right\rceil }$ and the definition of ${M_1}$ in (\ref{eq:M1}). Meanwhile, since ${\alpha_k = \alpha  \le \frac{1}{{4L}}}$ , substituting (\ref{eq:cor1_dks}) into (\ref{eq:th2}), we can obtain
\begin{equation}
\begin{aligned}\label{eq:cor1_f}
\frac{\alpha }{4}\sum\limits_{k = 0}^{K - 1} \mathbb{E}{\left\| {{d_k}} \right\|_2^2} & \le \sum\limits_{k = 0}^{K - 1} {\left( {\alpha {C_f}{{\left[ {\mathbb{E}\left\| {{{ d}_{k,P}^S} - {d_k}} \right\|_2^2} \right]}^{\frac{1}{2}}} + L{\alpha ^2}\mathbb{E}\left\| {{{ d}_{k,P}^S} - {d_k}} \right\|_2^2} \right)}  + {M_f}\\
& \le \sum\limits_{k = 0}^{K - 1} {\left( {\alpha {C_f}{{\left[ {\frac{{{M_1}}}{{\sqrt S }}} \right]}^{\frac{1}{2}}} + L{\alpha ^2}\frac{{{M_1}}}{{\sqrt S }}} \right)}  + {M_f}\\
& \le \sum\limits_{k = 0}^{K - 1} {\left( {\alpha {C_f}{M_1}^{\frac{1}{2}}{S^{ - \frac{1}{4}}} + L{\alpha ^2}{M_1}{S^{ - \frac{1}{2}}}} \right)}  + {M_f}\\
& \le \alpha {C_f}{M_1}^{\frac{1}{2}}\sum\limits_{k = 0}^{K - 1} {{S^{ - \frac{1}{4}}}}  + L{\alpha ^2}{M_1}\sum\limits_{k = 0}^{K - 1} {{S^{ - \frac{1}{2}}}}  + {M_f}.
\end{aligned}
\end{equation}
Then, dividing both sides simultaneously by $\frac{\alpha }{4}$ and the number of iterations ${K}$, rearranging the above inequality, yields
\begin{align}
\label{eq:cor1_dk}
\frac{1}{K}\sum\limits_{k = 0}^{K - 1} \mathbb{E}{\left\| {{d_k}} \right\|_2^2}  \le {C_f}{M_1^{\frac{1}{2}}}\frac{4}{K}\sum\limits_{k = 0}^{K - 1} {{S^{ - \frac{1}{4}}}}  + L\alpha M_1\frac{4}{K}\sum\limits_{k = 0}^{K - 1} {{S^{ - \frac{1}{2}}}}  + \frac{{4{M_f}}}{{\alpha K}}.
\end{align}
(a) Setting ${S = {\left( {K + 1} \right)^2}}$, plugging in the above inequality gives
\begin{align*}
\frac{1}{K}\sum\limits_{k = 0}^{K - 1} \mathbb{E}{\left\| {{d_k}} \right\|_2^2}  \le 4{C_f}{M_1}^{\frac{1}{2}}{K^{ - \frac{1}{2}}} + 4L\alpha {M_1}{K^{ - 1}} + \frac{{4{M_f}}}{{K\alpha }} \le {\rm{O}}\left( {4{C_f}{M_1}^{\frac{1}{2}}\frac{1}{{\sqrt K }}} \right).
\end{align*}

(b) Setting ${S = {\left( {k + 1} \right)^2}}$, similarly plugging in (\ref{eq:cor1_dk}), with simple calculations we can arrive at
\begin{align*}
\frac{1}{K}\sum\limits_{k = 0}^{K - 1}\mathbb{E} {\left\| {{d_k}} \right\|_2^2}  &\le {C_f}{M_1}^{\frac{1}{2}}\frac{4}{K}\sum\limits_{k = 0}^{K - 1} {\frac{1}{{\sqrt {k + 1} }}}  + L\alpha {M_1}\frac{4}{K}\sum\limits_{k = 0}^{K - 1} {\frac{1}{{k + 1}}}  + \frac{{4{M_f}}}{{\alpha K}}\\
& \le {C_f}{M_1}^{\frac{1}{2}}\frac{4}{K}\sum\limits_{k = 1}^K {\frac{1}{{\sqrt k }}}  + L\alpha {M_1}\frac{4}{K}\sum\limits_{k = 1}^K {\frac{1}{k}}  + \frac{{4{M_f}}}{{\alpha K}}\\
& \le {C_f}{M_1}^{\frac{1}{2}}\frac{4}{K}\left( {2\sqrt K  - 1} \right) + L\alpha {M_1}\frac{4}{K}\left( {\sqrt K } \right) + \frac{{4{M_f}}}{{\alpha K}}\\
& \le {C_f}{M_1}^{\frac{1}{2}}\frac{8}{{\sqrt K }} + L\alpha {M_1}\frac{4}{{\sqrt K }} + \frac{{4{M_f}}}{{\alpha K}}\\
& \le {\rm{O}}\left( {\max \left\{ {8{C_f}{M_1}^{\frac{1}{2}},4L\alpha {M_1}} \right\}\frac{1}{{\sqrt K }}} \right),
\end{align*}
where the third inequality is derived from the given conditions  ${\sum\limits_{k = 1}^K {\frac{1}{{\sqrt k }} \le 2\sqrt K  - 1} }$ and ${\sum\limits_{k = 1}^K {\frac{1}{k} \le \sqrt K } }$ (${K \ge 4}$). The proof is completed.
\end{proof}
Corollary \ref{cor:fixa_fixg} furnishes a convergence guarantee for our MSMD method, where we introduce a hyperparameter ${\theta>0}$ to better regulate the inner step size of ${\gamma_s}$. The per objective sampling complexity of our method is ${{\rm O}\left( {{\varepsilon ^{ - 2}}} \right)}$, significantly lower than that of the algorithm SDMGrad, where the exact solution of subproblem (\ref{eq:dk_1}) is also not used. This is attributed to the fact that the SMD method we adopt for solving the saddle point problem (\ref{eq:dk_1}) requires only one sample in each subproblem iteration, being much more efficient than SDMGrad's three samples when the number of iteration steps is large. The above corollary states the case where both the inner and outer step sizes are fixed. Similarly, with the following three Corollaries we present convergence rate results for several fixed and variable step size settings.
\begin{corollary}
\label{cor:fixa_alterg}
Set varying inner step size ${{\gamma_s} = \frac{\theta}{{{M_*}\sqrt s }}}$, $\theta>0$, fixed outer step size ${\alpha  \le \frac{1}{{4L}}}$  and ${P = \left\lceil {rS} \right\rceil }$, ${r \in \left( {0,1} \right)}$ in Theorems \ref{th1} and \ref{th:dk2}. Let ${K \ge 4}$, ${{M_*}}$ satisfies (\ref{eq:m*}) and
\begin{align}\label{eq:M2}
{M_2} = \left( {\frac{{2\bar D_{{\omega _z},Z}^2S}}{{S - \left\lceil {rS} \right\rceil  + 1}} + 2\sqrt {\frac{S}{{\left\lceil {rS} \right\rceil }}} } \right)\max \left\{ {\theta ,{\theta ^{ - 1}}} \right\}{M_*},
\end{align}
we have

(a) If we choose ${{S} = {\left( {K + 1} \right)^2}}$, the iteration ${x_k}$ generated from the proposed MSMD method satisfies
\begin{align}\label{eq:cor2_K2}
\frac{1}{K}\sum\limits_{k = 0}^{K - 1}\mathbb{E} {\left\| {{d_k}} \right\|_2^2}  \le {\rm{O}}\left( {4{C_f}{M_2}^{\frac{1}{2}}\frac{1}{{\sqrt K }}} \right).
\end{align}
(b) If we choose  ${{S} = {(k+1)^2}}$ and total iteration step $K \le 4$, the iteration ${x_k}$ generated from the proposed MSMD method satisfies
\begin{align}\label{eq:cor2_k2}
\frac{1}{K}\sum\limits_{k = 0}^{K - 1} \mathbb{E}{\left\| {{d_k}} \right\|_2^2}  \le {\rm{O}}\left( {8{C_f}{M_2}^{\frac{1}{2}}\frac{1}{{\sqrt K }}} \right).
\end{align}

\end{corollary}
\begin{proof}
Setting ${{\gamma_s} = \frac{\theta }{{{M_*}\sqrt s }}}$  in (\ref{eq:th_dk-dks}) we can get
\begin{equation}\label{eq:cor2_dks}
\begin{aligned}
\mathbb{E}\left\| {{d_k} -  d_{k,P}^S} \right\|_2^2 &\le 2{\left[ {\sum\limits_{s = P}^S {{\gamma_s}} } \right]^{ - 1}}\left[ {\bar D_{{\omega _z},Z}^2 + M_*^2\sum\limits_{s = P}^{S} {\gamma_s^2} } \right]\\
& \le 2{\left[ {\sum\limits_{s = P}^S {\frac{\theta }{{{M_*}\sqrt s }}} } \right]^{ - 1}}\left[ {\bar D_{{\omega _z},Z}^2 + \sum\limits_{s = P}^{S } {\frac{{{\theta ^2}}}{s}} } \right]\\
& \le \max \left\{ {\theta ,{\theta ^{ - 1}}} \right\}\frac{{{M_*}}}{{\sqrt S }}\left[ {\frac{{2\bar D_{{\omega _z},Z}^2S}}{{S - P + 1}} + 2\sqrt {\frac{S}{P}} } \right]\\
& \le \frac{{{M_2}}}{{\sqrt S }}.
\end{aligned}
\end{equation}
By substituting equation (\ref{eq:cor2_dks}) into equation (\ref{eq:th2}), the subsequent steps follow the same procedure as outlined in Corollary \ref{cor:fixa_fixg}, leading to the verification of results (\ref{eq:cor2_K2}) and (\ref{eq:cor2_k2}), the proof is completed.
\end{proof}

\begin{corollary}
\label{cor:altera_fixg}
Set fixed inner step size ${{\gamma_s} = \frac{\theta}{{{M_*}\sqrt S }}}$, $\theta>0$, varying outer step size ${{\alpha _k} = \frac{\rho }{{k + 1}}}$  and ${P = \left\lceil {rS} \right\rceil }$, ${r \in \left( {0,1} \right)}$ in Theorems \ref{th1} and \ref{th:dk2}. Let (\ref{eq:m*}) and  (\ref{eq:M1})  satisfy, we have

(a) If we choose ${{S} = {\left( {K + 1} \right)^2}}$, the iteration ${x_k}$ generated from the proposed MSMD method satisfies
\begin{align}\label{cor:altera_fixg_K2}
\frac{1}{{{A_{K - 1}}}}\sum\limits_{k = 0}^{K - 1} {{\alpha _k}\mathbb{E}\left\| {{d_k}} \right\|_2^2} \le {\rm{O}}\left( {\frac{{\max \left\{ {LmC_f^2\rho {\pi ^2}{3^{ - 1}},2{M_f}{\rho ^{ - 1}}} \right\}}}{{\ln \left( {K + 1} \right)}}} \right).
\end{align}
(b) If we choose ${{S} = {(k+1)^2}}$, there exist finite constants $\zeta(\frac{3}{2})$, ${\zeta (3) \in \mathbb{R}}$ such that the iteration ${x_k}$ generated from the proposed MSMD method satisfies
\begin{align}\label{eq:cor:altera_fixg_k2}
\frac{1}{{{A_{K - 1}}}}\sum\limits_{k = 0}^{K - 1} {{\alpha _k}\mathbb{E}\left\| {{d_k}} \right\|_2^2}  \le {\rm{O}}\left( {\max \left\{ {LmC_f^2\rho {\pi ^2}{3^{ - 1}},2\zeta(\frac{3}{2}){C_f}{M_1}^{\frac{1}{2}},2\zeta(3)L\rho {M_1},2{M_f}{\rho ^{ - 1}}} \right\}\frac{1}{{\ln \left( {K + 1} \right)}}} \right).
\end{align}
\end{corollary}
\begin{proof}
Recalling from (\ref{eq:th2}) we have
\begin{align*}
\sum\limits_{k = 0}^{K - 1} {{\alpha _k}\mathbb{E}\left\| {{d_k}} \right\|_2^2} & \le \sum\limits_{k = 0}^{K - 1} {2L\alpha _k^2\mathbb{E}\left\| {{d_k}} \right\|_2^2}  + \sum\limits_{k = 0}^{K - 1} {2{\alpha _k}{C_f}{M_1}^{\frac{1}{2}}{S^{ - \frac{1}{4}}}}  + \sum\limits_{k = 0}^{K - 1} {2L\alpha _k^2{M_1}{S^{ - \frac{1}{2}}}}  + 2{M_f}\\
& \le 2LmC_f^2\sum\limits_{k = 0}^{K - 1} {\alpha _k^2}  + 2{C_f}{M_1}^{\frac{1}{2}}\sum\limits_{k = 0}^{K - 1} {{\alpha _k}{S^{ - \frac{1}{4}}}}  + 2L{M_1}\sum\limits_{k = 0}^{K - 1} {\alpha _k^2{S^{ - \frac{1}{2}}}}  + 2{M_f},
 \end{align*}
where the last inequality is due to the bounds on ${{\left\| {{d_k}} \right\|_2^2}}$.
Then, plugging in the chosen ${{\alpha _k} = \frac{\rho }{{k + 1}}}$ and dividing both sides by ${{A_{K - 1}} = \sum\limits_{k = 0}^{K - 1} {{\alpha _k}}  }$, we have

\begin{equation}
\label{eq:cor3_Aksum}
\begin{aligned}
\frac{1}{{{A_{K - 1}}}}\sum\limits_{k = 0}^{K - 1} {{\alpha _k}\mathbb{E}\left\| {{d_k}} \right\|_2^2}  \le \frac{{2LmC_f^2{\rho ^2}}}{{{A_{K - 1}}}}\sum\limits_{k = 0}^{K - 1} {\frac{1}{{{{\left( {k + 1} \right)}^2}}}}  + \frac{{2{C_f}\rho }}{{{A_{K-1}}}}{M_1}^{\frac{1}{2}}\sum\limits_{k = 0}^{K - 1} {\frac{{{S^{ - \frac{1}{4}}}}}{{k + 1}}}  + \frac{{2L{\rho ^2}}}{{{A_{K - 1}}}}{M_1}\sum\limits_{k = 0}^{K - 1} {\frac{{{S^{ - \frac{1}{2}}}}}{{{{\left( {k + 1} \right)}^2}}}}  + \frac{{2{M_f}}}{{{A_{K - 1}}}}.
\end{aligned}
\end{equation}

(a) Setting ${{S} = (K+1)^2}$, substitution of ${\frac{1}{{{A_{K - 1}}}} = {\left( {\sum\limits_{k = 1}^K {\frac{\rho }{k}} } \right)^{ - 1}}}$ into the right-hand side of inequality (\ref{eq:cor3_Aksum}) leads to the following
\begin{align*}
\frac{1}{{{A_{K - 1}}}}\sum\limits_{k = 0}^{K - 1} {{\alpha _k}\mathbb{E}\left\| {{d_k}} \right\|_2^2} & \le 2LmC_f^2{\rho ^2}{\left( {\sum\limits_{k = 1}^K {\frac{\rho }{k}} } \right)^{ - 1}}\sum\limits_{k = 1}^K {\frac{1}{{{k^2}}}}  + {M_1}^{\frac{1}{2}}2{C_f}\rho {\left( {\sum\limits_{k = 1}^K {\frac{\rho }{k}} } \right)^{ - 1}}\frac{1}{{\sqrt {K + 1} }}\sum\limits_{k = 1}^K {\frac{1}{k}} \\
& + 2L{\rho ^2}{M_1}{\left( {\sum\limits_{k = 1}^K {\frac{\rho }{k}} } \right)^{ - 1}}\frac{1}{{K + 1}}\sum\limits_{k = 1}^K {\frac{1}{{{k^2}}}}  + {\left( {\sum\limits_{k = 1}^K {\frac{\rho }{k}} } \right)^{ - 1}}2{M_f}\\
& \le \frac{{LmC_f^2\rho {\pi ^2}}}{{3\ln \left( {K + 1} \right)}} + {M_1}^{\frac{1}{2}}\frac{{2{C_f}}}{{\sqrt {K + 1} }} + \frac{{L\rho {M_1}{\pi ^2}}}{{3\left( {K + 1} \right)\ln \left( {K + 1} \right)}} + \frac{{2{M_f}}}{{\rho \ln (K + 1)}}\\
& \le {\rm{O}}\left( {\frac{{\max \left\{ {LmC_f^2\rho {\pi ^2}{3^{ - 1}},2{M_f}{\rho ^{ - 1}}} \right\}}}{{\ln \left( {K + 1} \right)}}} \right),
\end{align*}
where the second inequality is derived from the given conditions  ${\sum\limits_{k = 1}^K {\frac{1}{{{k^2}}}}  \le \sum\limits_{k = 1}^\infty  {\frac{1}{{{k^2}}}}  = \frac{{{\pi ^2}}}{6}}$ and ${\sum\limits_{k = 1}^K {\frac{1}{k} \ge \ln \left( {K + 1} \right)} }$.

(b) Similarly, by setting ${{S} = {\left( {k + 1} \right)^2}}$ we have
\begin{align*}
\frac{1}{{{A_{K - 1}}}}\sum\limits_{k = 0}^{K - 1} {{\alpha _k}\mathbb{E}\left\| {{d_k}} \right\|_2^2} & \le 2LmC_f^2{\rho ^2}{\left( {\sum\limits_{k = 1}^K {\frac{\rho }{k}} } \right)^{ - 1}}\sum\limits_{k = 1}^K {\frac{1}{{{k^2}}}}  + {M_1}^{\frac{1}{2}}2{C_f}\rho {\left( {\sum\limits_{k = 1}^K {\frac{\rho }{k}} } \right)^{ - 1}}\sum\limits_{k = 1}^K {\frac{1}{{k\sqrt k }}} \\
& + 2L{\rho ^2}{M_1}{\left( {\sum\limits_{k = 1}^K {\frac{\rho }{k}} } \right)^{ - 1}}\sum\limits_{k = 1}^K {\frac{1}{{{k^3}}}}  + {\left( {\sum\limits_{k = 1}^K {\frac{\rho }{k}} } \right)^{ - 1}}2{M_f}\\
&\le \frac{{LmC_f^2\rho {\pi ^2}}}{{3\ln \left( {K + 1} \right)}} + \frac{{2\zeta(\frac{3}{2}){M_1}^{\frac{1}{2}}{C_f}}}{{\ln \left( {K + 1} \right)}} + \frac{{2\zeta(3)L\rho {M_1}}}{{\ln \left( {K + 1} \right)}} + \frac{{2{M_f}}}{{\rho \ln (K + 1)}}\\
&\le {\rm{O}}\left( {\max \left\{ {LmC_f^2\rho {\pi ^2}{3^{ - 1}},2\zeta(\frac{3}{2}){C_f}{M_1}^{\frac{1}{2}},2\zeta(3)L\rho {M_1},2{M_f}{\rho ^{ - 1}}} \right\}\frac{1}{{\ln \left( {K + 1} \right)}}} \right),
\end{align*}
where there exist finite constants $\zeta(\frac{3}{2})$, ${\zeta (3) \in \mathbb{R}}$ based on the Riemann zeta function \cite{ivic2012riemann} such that ${\sum\limits_{k = 1}^\infty  {\frac{1}{{k\sqrt k }}}  = \zeta (\frac{3}{2})}$ and ${\sum\limits_{k = 1}^\infty  {\frac{1}{{{k^3}}}}  = \zeta (3)}$, thus the second inequality holds. The proof is completed.
\end{proof}
\begin{corollary}
\label{cor:altera_alterg}
Set varying inner step size ${{\gamma_s} = \frac{\theta}{{{M_*}\sqrt s }}}$, $\theta>0$ and varying outer step size ${{\alpha _k} = \frac{\rho }{{k + 1}}}$ in Theorems \ref{th1} and \ref{th:dk2}. Let (\ref{eq:m*}) and  (\ref{eq:M2}) satisfy, we have

(a)  If we choose ${{S} = {\left( {K + 1} \right)^2}}$, the iteration ${x_k}$ generated from the proposed MSMD method satisfies
\begin{align}\label{cor:altera_alterg_K2}
\frac{1}{{{A_{K - 1}}}}\sum\limits_{k = 0}^{K - 1} {{\alpha _k}\mathbb{E}\left\| {{d_k}} \right\|_2^2}  \le {\rm{O}}\left( {\frac{{\max \left\{ {LmC_f^2\rho {\pi ^2}{3^{ - 1}},2{M_f}{\rho ^{ - 1}}} \right\}}}{{\ln \left( {K + 1} \right)}}} \right).
\end{align}
(b)  If we choose ${{S} = {(k+1)^2}}$, there exist finite constants $\zeta(\frac{3}{2})$, ${\zeta (3) \in \mathbb{R}}$ such that the iteration ${x_k}$ generated from the proposed MSMD method satisfies
\begin{align}\label{eq:cor:altera_alterg_k2}
\frac{1}{{{A_{K - 1}}}}\sum\limits_{k = 0}^{K - 1} {{\alpha _k}\mathbb{E}\left\| {{d_k}} \right\|_2^2}  \le {\rm{O}}\left( {\max \left\{ {LmC_f^2\rho {\pi ^2}{3^{ - 1}},2\zeta(\frac{3}{2}){C_f}{M_2}^{\frac{1}{2}},2\zeta(3)L\rho {M_2},2{M_f}{\rho ^{ - 1}}} \right\}\frac{1}{{\ln \left( {K + 1} \right)}}} \right).
\end{align}
\end{corollary}
\begin{proof}
The proof of this corollary follows analogously from the proofs of Corollary \ref{cor:fixa_alterg} and Corollary \ref{cor:altera_fixg}, and thus we do not repeat the detailed steps here.
\end{proof}

Corollaries \ref{cor:fixa_fixg}-\ref{cor:altera_alterg} present the convergence rates for our MSMD method under both fixed and varying inner and outer step sizes. The inner step size primarily influences the convergence of ${\mathbb{E}\left\| {{d_k} -  d_{k,P}^S} \right\|_2^2}$, and different settings for the inner step size affect the parameters $M_1$ and ${M_2}$ in the convergence results.  Specifically, a fixed $\gamma_s$ corresponds to $M_1$, while a varying ${\gamma_s}$ corresponds to $M_2$. The outer step size ${\alpha_k}$ determines the overall convergence rate. A fixed ${{\alpha _k} = \alpha  \le 1/4L}$ yields a convergence rate of ${{\rm O}\left( {1/\sqrt K } \right)}$ , while a varying $\alpha_k$ results in a slower rate of ${{\rm O}\left( {1/\ln \left( {K + 1} \right)} \right)}$. This difference stems from the fact that the objective function in SMOO is not upper-bounded. For bounded SMOO problems, applying a moving average technique, similar to that used for $d_{k,s}$, enables the method to achieve a convergence rate of ${{\rm O}\left( {1/\sqrt K } \right)}$.


\section{MOO Problems with Preferences}
\label{sec:preferences}
In some situations, the goal of MOO is not only to find a common descent direction, but also to optimize a specific objective, which is usually a linear combination of the objective functions ${{f_0}\left( {{x_k}} \right) = \sum\limits_{i = 1}^m {{w_i}{f_i}\left( {{x_k}} \right)} }$ for some given preference vector $w \in W \subset {\mathbb{R}^m}$, we denote such problem as a MOO problem with preferences.
In this case, Xiao et al. \cite{xiao2024direction} proposed the following subproblem
\begin{align}\label{eq:pre_sub}
{d_k} = \mathop {\arg \min }\limits_{d \in {\mathbb{R}^n}} \mathop {\max }\limits_{i \in \left[ m \right]} \left\langle {\nabla {f_i}\left( {{x_k}} \right) + \mu {f_0}\left( {{x_k}} \right),d} \right\rangle  + \frac{1}{2}\left\| {{d}} \right\|_2^2.
\end{align}
We can easily express the dual problem of (\ref{eq:pre_sub}) with stochastic vector ${\xi  \in \Xi  \subset {\mathbb{R}^n}}$ as follows:
\begin{align}\label{eq:pre_dk}
\left( {{d_k},{\lambda _k}} \right) = \mathop {\arg \min }\limits_{d \in {\mathbb{R}^n}} \mathop {\max }\limits_{\lambda  \in {\Delta ^m}} \left\{ {\mathbb{E}\left[ {\varphi \left( {d,\lambda ;\xi } \right)} \right] = \left\langle {\mathbb{E}\left[ {\nabla F\left( {{x_k};\xi } \right)\lambda  + \mu {g_0}\left( {{x_k};\xi } \right)} \right],d} \right\rangle  + \frac{1}{2}\left\| d \right\|_2^2} \right\},
\end{align}
where ${\mu}$ is the parameter and ${{g_0}\left( {{x_k};\xi } \right) = \sum\limits_{i = 1}^m {{w_i}\nabla {f_i}\left( {{x_k};\xi } \right)} }$ represents the weighted gradient of this preference vector. Assuming an unbiased gradient, we obtain
\[{g_0}\left( {{x_k}} \right) = \mathbb{E}\left[ {{g_0}\left( {{x_k};\xi } \right)} \right] = \mathbb{E}\left[ {\sum\limits_{i = 1}^m {{w_i}\nabla {f_i}\left( {{x_k};\xi } \right)} } \right] = \sum\limits_{i = 1}^m {{w_i}\nabla {f_i}\left( {{x_k}} \right)}  = \nabla {f_0}\left( {{x_k}} \right).\]
To ensure that the common descent direction $d_k$ does not deviate from the objective direction ${g_0}\left( {{x_k}} \right)$, we add an inner product regularization term of ${{\mu \left\langle {{{g_0}\left( {{x_k}} \right)},d_k} \right\rangle }}$ to (\ref{eq:dk}). In the following, we present the iteration format of MSMD under this scenario.
Similar to the previous discussion, $x_k$ is updated in the direction of
\[{d_k} =  - \nabla F\left( {{x_k};\xi } \right){\lambda _k} - \mu {{g_0}\left( {{x_k}} \right)}~~\mbox{and}~~{\lambda _k} = \mathop {\arg \min }\limits_{\lambda  \in {\Delta ^m}} \left\| {\nabla F\left( {{x_k};\xi } \right)\lambda  + \mu {{g_0}\left( {{x_k}} \right)}} \right\|_2^2.\]
We present the following assumptions for the objective function and the gradient component with preference.
\begin{assumption}\label{ass_gradient_pre}
For all iterates ${k \in [K]}$, i.i.d stochastic variable ${\xi  \in \Xi }$ and preference vector $w \in W \subset {\mathbb{R}^m}$ with ${{g_0}\left( {{x_k}} \right) = \sum\limits_{i = 1}^m {{w_i}\nabla {f_i}\left( {{x_k}} \right)} }$, we have access to the stochastic preference gradient is bounded by ${{C_g}\ge 0}$, i.e.,
\[\mathbb{E}\left[ {\left\| { {g_0}({x_k},\xi )} \right\|_2^2} \right] \le C_g^2,\]
and the variance is also bounded by ${\delta_0 \ge 0}$ i.e.
\begin{align*}
\mathbb{E   }\left[ {\left\| {{g_0}\left( {{x_k};\xi } \right) - {g_0}\left( {{x_k}} \right)} \right\|_2^2} \right] \le \delta _0^2.
\end{align*}
\end{assumption}
Then, relations with Assumption \ref{ass_gradient} imply that
\begin{equation}\label{eq:bound_pre_dk}
\begin{aligned}
\left\| {{d_k}} \right\|_2^2 = \left\| {\nabla F\left( {{x_k}} \right){\lambda _k} + \mu {{g_0}\left( {{x_k}} \right)}} \right\|_2^2 \le 2\left\| {\nabla F\left( {{x_k}} \right){\lambda _k}} \right\|_2^2 + 2\left\| {\mu {{g_0}\left( {{x_k}} \right)}} \right\|_2^2 \le 2C_f^2 + 2{\mu ^2}C_g^2.
\end{aligned}
\end{equation}
Thus we have ${C_0 = \left\{ {d \in {\mathbb{R}^n}|\left\| d \right\|_2^2 \le 2C_f^2 + 2{\mu ^2}C_g^2} \right\}}$ in the problem (\ref{eq:dk_c}).
Meanwhile, we consider (\ref{eq:wlamda}) and (\ref{eq:wd}) as the distance generating function, allowing us to derive the norm definition for MSMD in Problem (\ref{eq:pre_dk}) as shown in (\ref{eq:norm}) and (\ref{eq:dual_norm}). The only difference lies in the value of ${R_d}$, which can be expressed as
\[R_d^2 = C_f^2 + {\mu ^2}C_g^2.\]
By modifying the SMDM presented in Section \ref{sec:ourmethod} at this stage,  we can obtain the iterative formulation of the MSDM method for solving subproblems with preference as follows:
\begin{align}\label{eq:pre_zks}
{z_{k,s + 1}} = {P_{{z_{k,s}}}}\left( {{\gamma_s}{\nabla _z}\varphi \left( {{z_{k,s}};{\xi _s}} \right)} \right),
\end{align}
where the gradient of the subproblem is
\[\nabla \varphi \left( {{z_{k,s}};{\xi _s}} \right) = \left[ {\begin{array}{*{20}{c}}
{{\nabla _d}\varphi \left( {{z_{k,s}};{\xi _s}} \right)}\\
{{-\nabla _\lambda }\varphi \left( {{z_{k,s}};{\xi _s}} \right)}
\end{array}} \right] = \left[ {\begin{array}{*{20}{c}}
{\nabla {F_k}\left( {{\xi _s}} \right){\lambda _{k,s}} +\mu {{{g_0}\left( {{x_k};\xi_s } \right)}}+ {d_{k,s}} }\\
{ - \nabla {F_k}\left( {{\xi _s}} \right){d_{k,s}}}
\end{array}} \right].\]
The iterative format of ${\lambda_{k,s+1}}$ is the same as in (\ref{eq:lamdaks}), and with respect to $d_{k,{s+1}}$, the following iteration format can be obtained
\begin{align}\label{eq:pre_dks}
{d_{k,s + 1}} = {\Pi _{C_0}}\left( {{d_{k,s}} - {\gamma_s}\left( {\nabla {F_k}\left( {{\xi _s}} \right){\lambda _{k,s}} + \mu {{{g_0}\left( {{x_k};{\xi _s}} \right)}} + {d_{k,s}} } \right)} \right).
\end{align}
Bringing (\ref{eq:pre_zks}) to Algorithm \ref{alg:smoo} yields the algorithmic framework for problem (\ref{eq:pre_dk}).

\begin{lemma}
\label{lemma:bound_pre}
Suppose that Assumptions \ref{ass_f} and \ref{ass_gradient} hold. We get the following upper bounds
\begin{align}\label{eq:bound_pre_gra_phi}
\mathbb{E}\left\| {\nabla {\varphi _z}\left( {{z_{k,s}};{\xi _s}} \right)} \right\|_*^2  \le 18{\left( {C_f^2 + {\mu ^2}C_g^2} \right)^2} + 4m\ln mC_f^2\left( {C_f^2 + {\mu ^2}C_g^2} \right).
\end{align}
and
\begin{align}\label{eq:bound_pre_var_phi}
\mathbb{E}\left\| {\nabla {\varphi _z}\left( {{z_{k,s}};{\xi _s}} \right) - \nabla {\phi _z}\left( {{z_{k,s}}} \right)} \right\|_*^2 \le 4\left( {C_f^2 + {\mu ^2}C_g^2} \right)\left( {{\delta ^2} + {\mu ^2}\delta _0^2 + m{\delta ^2}\ln m} \right).
\end{align}
\end{lemma}
\begin{proof}
Following the inequality
\[\left\| {a + b + c} \right\|_2^2 \le 3\left\| a \right\|_2^2 + 3\left\| b \right\|_2^2 + 3\left\| c \right\|_2^2,\]
we can obtain that
\begin{align*}
\mathbb{E}\left\| {\nabla {\varphi _z}\left( {{z_{k,s}};{\xi _s}} \right)} \right\|_*^2 &\le 2\left( {C_f^2 + {\mu ^2}C_g^2} \right)\mathbb{E}\left\| {{\nabla _d}\varphi \left( {{z_{k,s}};{\xi _s}} \right)} \right\|_2^2 + 2\ln m\mathbb{E}\left\| {{\nabla _\lambda }\varphi \left( {{z_{k,s}};{\xi _s}} \right)} \right\|_2^2\\
& \le 2\left( {C_f^2 + {\mu ^2}C_g^2} \right)\mathbb{E}\left\| {\nabla {F_k}\left( {{\xi _s}} \right){\lambda _{k,s}} + \mu {g_0}\left( {{x_k};{\xi _s}} \right) + {d_{k,s}}} \right\|_2^2 + 2\ln m\mathbb{E}\left\| {\nabla {F_k}\left( {{\xi _s}} \right){d_{k,s}}} \right\|_\infty ^2\\
& \le 18{\left( {C_f^2 + {\mu ^2}C_g^2} \right)^2} + 4m\ln mC_f^2\left( {C_f^2 + {\mu ^2}C_g^2} \right).
\end{align*}
For the second inequality, we can write
\begin{align*}
&\mathbb{E}\left\| {\nabla {\varphi _z}\left( {{z_{k,s}};{\xi _s}} \right) - \nabla {\varphi _z}\left( {{z_{k,s}}} \right)} \right\|_*^2 = 2\left( {{m^2}C_f^2 + m{\mu ^2}C_g^2} \right)\mathbb{E}\left\| {\nabla {\varphi _d}\left( {{z_{k,s}};{\xi _s}} \right) - \nabla {\varphi _d}\left( {{z_{k,s}}} \right)} \right\|_2^2 \\
& + 2\ln{m}\mathbb{E}\left\| {\nabla {\varphi _\lambda }\left( {{z_{k,s}};{\xi _s}} \right) - \nabla {\varphi _\lambda }\left( {{z_{k,s}}} \right)} \right\|_\infty ^2\\
& \le 2\left( {C_f^2 + {\mu ^2}C_g^2} \right)\mathbb{E}\left\| {\nabla {F_k}\left( {{\xi _s}} \right){\lambda _{k,s}} + \mu {g_0}\left( {{x_k};{\xi _s}} \right) + {d_{k,s}} - \nabla {F_k}{\lambda _{k,s}} - \mu {g_0}\left( {{x_k}} \right) - {d_{k,s}}} \right\|_2^2\\
& + 2\ln m\mathbb{E}\left\| {\nabla {F_k}\left( {{\xi _s}} \right){d_{k,s}} - \nabla {F_k}{d_{k,s}}} \right\|_\infty ^2\\
&\le 4\left( {C_f^2 + {\mu ^2}C_g^2} \right)\mathbb{E}\left\| {\nabla {F_k}\left( {{\xi _s}} \right){\lambda _{k,s}} - \nabla {F_k}{\lambda _{k,s}}} \right\|_2^2 + 4\left( {C_f^2 + {\mu ^2}C_g^2} \right)\mathbb{E}\left\| {\mu {g_0}\left( {{x_k};{\xi _s}} \right) - \mu {g_0}\left( {{x_k}} \right)} \right\|_2^2\\
& + 2\ln m\mathbb{E}\left\| {\nabla {F_k}\left( {{\xi _s}} \right){d_{k,s}} - \nabla {F_k}{d_{k,s}}} \right\|_\infty ^2\\
&\le 4\left( {C_f^2 + {\mu ^2}C_g^2} \right){\delta ^2} + 4\left( {C_f^2 + {\mu ^2}C_g^2} \right){\mu ^2}\delta _0^2 + 4m{\delta ^2}\ln m\left( {C_f^2 + {\mu ^2}C_g^2} \right)\\
&= 4\left( {C_f^2 + {\mu ^2}C_g^2} \right)\left( {{\delta ^2} + {\mu ^2}\delta _0^2 + m{\delta ^2}\ln m} \right),
\end{align*}
completes the proof.
\end{proof}
Both upper bounds obtained in Lemma \ref{lemma:bound_pre} are dependent on the preference vectors and grow as the norm and variance of the preference gradient increase. Following the analysis in Theorems \ref{th1}, \ref{th:dk2}, and Corollary \ref{cor:fixa_fixg}, by selecting fixed inner and outer step sizes, we can derive the convergence rate results for the MSMD method with preferences.
\begin{theorem}\label{th3}
Suppose that Assumptions \ref{ass_f}, \ref{ass_gradient} and \ref{ass_gradient_pre} hold. Let ${{S} = {K^2}}$, ${\alpha  \le \frac{1}{{4L}}}$, ${{\gamma_s} = \frac{\theta}{{{M_*}\sqrt S }}}$, ${P = \left\lceil {rS} \right\rceil }$ and ${r \in \left( {0,1} \right)}$. Set $M_0$ be a given constant according to the following:
\begin{align*}
{M_0} = \left( {\frac{{2\bar D_{{\omega _z},Z}^2S}}{{S - \left\lceil {rS} \right\rceil  + 1}} + 2} \right)\max \left\{ {\theta ,{\theta ^{ - 1}}} \right\}{M_{0,*}},
\end{align*}
where
\begin{align} \label{eq:m*2}
M_{0,*}^2 = 9{\left( {C_f^2 + {\mu ^2}C_g^2} \right)^2} + 2\left( {m\ln mC_f^2 + {\delta ^2} + {\mu ^2}\delta _0^2 + m\ln m{\delta ^2}} \right)\left( {C_f^2 + {\mu ^2}C_g^2} \right).
\end{align}
We have the iteration ${x_k}$ generated from (\ref{eq:pre_zks}) satisfies
\begin{align}\label{eq:th3}
\frac{1}{K}\sum\limits_{k = 0}^{K - 1} {\mathbb{E}\left\| {{d_k}} \right\|_2^2}  \le {\rm O}\left( {\frac{{4\sqrt {2{M_0}\left( {C_f^2 + {\mu ^2}C_g^2} \right)} }}{{\sqrt K }}} \right).
\end{align}
\end{theorem}
\begin{proof}
Taking the expectation of (\ref{eq:th1_1}) and plugging the bound of variance (\ref{eq:bound_pre_var_phi}) into it has
\begin{align}\label{eq:th3_2}
\mathbb{E}\left[ {\sum\limits_{s = P}^S {{\gamma _s}{{\left( {{z_k} - {v_s}} \right)}^T}\left( {\nabla {\phi _z}\left( {{z_{k,s}};{\xi _s}} \right) - \nabla {\phi _z}\left( {{z_{k,s}}} \right)} \right)} } \right] \le V\left( {{z_{k,P}},{z_k}} \right) + 2\left( {C_f^2 + {\mu ^2}C_g^2} \right)\left( {{\delta ^2} + {\mu ^2}\delta _0^2 + m\ln m{\delta ^2}} \right)\sum\limits_{s = P}^S {\gamma _s^2} .
\end{align}
By plugging (\ref{eq:th3_2}) into (\ref{eq:th1_4}) and taking expectations on both sides, we can obtain
\begin{align*}
&\mathbb{E}\left[ {\sum\limits_{s = P}^S {{\gamma _s}} {{\left( {{z_{k,s}} - {z_k}} \right)}^T}\nabla {\varphi _z}\left( {{z_{k,s}}} \right)} \right] \le V\left( {{z_{k,P}},{z_k}} \right) + \sum\limits_{s = P}^S {\frac{{\gamma _s^2}}{2}} \mathbb{E}\left\| {\nabla {\varphi _z}\left( {{z_{k,s}};{\xi _s}} \right)} \right\|_*^2\\
& - \sum\limits_{s = P}^S {{\gamma _s}} \mathbb{E}\left[ {{{\left( {{v_s} - {z_k}} \right)}^T}\left( {\nabla {\varphi _z}\left( {{z_{k,s}};{\xi _s}} \right) - \nabla {\varphi _z}\left( {{z_{k,s}}} \right)} \right)} \right]\\
& \le V\left( {{z_{k,P}},{z_k}} \right) + \sum\limits_{s = P}^S {\frac{{\gamma _s^2}}{2}} \mathbb{E}\left\| {\nabla {\varphi _z}\left( {{z_{k,s}};{\xi _s}} \right)} \right\|_*^2 + V\left( {{z_{k,P}},{z_k}} \right) + 2\left( {C_f^2 + {\mu ^2}C_g^2} \right)\left( {{\delta ^2} + {\mu ^2}\delta _0^2 + m\ln m{\delta ^2}} \right)\sum\limits_{s = P}^S {\gamma _s^2} \\
& \le 2V\left( {{z_{k,P}},{z_k}} \right) + \left( {9{{\left( {C_f^2 + {\mu ^2}C_g^2} \right)}^2} + 2m\ln mC_f^2\left( {C_f^2 + {\mu ^2}C_g^2} \right)} \right)\sum\limits_{s = P}^S {\gamma _s^2}  + 2\left( {C_f^2 + {\mu ^2}C_g^2} \right)\left( {{\delta ^2} + {\mu ^2}\delta _0^2 + m\ln m{\delta ^2}} \right)\sum\limits_{s = P}^S {\gamma _s^2} \\
& \le \bar D_{{\omega _z},Z}^2 + \underbrace {\left( {9{{\left( {C_f^2 + {\mu ^2}C_g^2} \right)}^2} + 2\left( {m\ln mC_f^2 + {\delta ^2} + {\mu ^2}\delta _0^2 + m\ln m{\delta ^2}} \right)\left( {C_f^2 + {\mu ^2}C_g^2} \right)} \right)}_{ = M_{0,*}^2}\sum\limits_{s = P}^S {\gamma _s^2} \\
& = \bar D_{{\omega _z},Z}^2 + M_{0,*}^2\sum\limits_{s = P}^S {\gamma _s^2} ,
\end{align*}
where the third inequality can be directly obtained by Lemma \ref{lemma:bound_pre}.
Follow from (\ref{eq:cor1_dks}), let ${{\gamma_s} = \frac{\theta }{{{M^*}\sqrt S }}}$ we can obtain
\begin{align}\label{eq:th3_dks}
\mathbb{E}\left\| {{{ d}_{k,P}^S} - {d_k}} \right\|_2^2 \le {\frac{{{M_0}}}{{\sqrt S }}},
\end{align}
where
\[{M_0} = \left( {\frac{{2\bar D_{{\omega _z},Z}^2S}}{{S - \left\lceil {rS} \right\rceil  + 1}} + 2} \right)\max \left\{ {\theta ,{\theta ^{ - 1}}} \right\}{M_{0,*}}.\]
Thus, by combining (\ref{eq:th3_dks}) with the bound on the gradient of ${f(\cdot)}$, we can conclude
\begin{align*}
{\left( {\mathbb{E}\left[ {\mathop {\max }\limits_{i \in \left[ m \right]} \left\langle {\nabla {f_i}\left( {{x_k}} \right) + \mu {g_0}\left( {{x_k}} \right), d_{k,P}^S - {d_k}} \right\rangle } \right]} \right)^2} & \le \mathbb{E}\left[ {\mathop {\max }\limits_{i \in \left[ m \right]} \left\| {\nabla {f_i}\left( {{x_k}} \right) + \mu {g_0}\left( {{x_k}} \right)} \right\|_2^2\left\| { d_{k,P}^S - {d_k}} \right\|_2^2} \right]\\
& \le \left( {2C_f^2 + 2{\mu ^2}C_g^2} \right)\mathbb{E}\left\| { d_{k,P}^S - {d_k}} \right\|_2^2\\
& \le 2\left( {C_f^2 + {\mu ^2}C_g^2} \right)\frac{{{M_0}}}{{\sqrt S }}.
\end{align*}
Then, similar to Theorem \ref{th:dk2}, we consider the following inequality

\begin{equation}\label{eq:pre_maxf}
\begin{aligned}
\mathop {\max }\limits_{i \in \left[ m \right]} {f_i}\left( {{x_{k + 1}}} \right) + \mu {f_0}\left( {{x_{k + 1}}} \right) &\le \mathop {\max }\limits_{i \in \left[ m \right]} \left\langle {\nabla {f_i}\left( {{x_k}} \right) + \mu {f_0}\left( {{x_k}} \right),{\alpha _k} d_{k,P}^S} \right\rangle  + \frac{L}{2}\left\| {{\alpha _k} d_{k,P}^S} \right\|_2^2 + \mathop {\max }\limits_{i \in \left[ m \right]} {f_i}\left( {{x_k}} \right) + \mu {f_0}\left( {{x_k}} \right)\\
& \le {\alpha _k}\mathop {\max }\limits_{i \in \left[ m \right]} \left\langle {\nabla {f_i}\left( {{x_k}} \right) + \mu {f_0}\left( {{x_k}} \right),{d_k}} \right\rangle  + {\alpha _k}\mathop {\max }\limits_{i \in \left[ m \right]} \left\langle {\nabla {f_i}\left( {{x_k}} \right) + \mu {f_0}\left( {{x_k}} \right), d_{k,P}^S - {d_k}} \right\rangle \\
& + \frac{L}{2}\alpha _k^2\left\| { d_{k,P}^S - {d_k} + {d_k}} \right\|_2^2 + \mathop {\max }\limits_{i \in \left[ m \right]} {f_i}\left( {{x_k}} \right) + \mu {f_0}\left( {{x_k}} \right)\\
& \le  - \frac{{{\alpha _k}}}{2}\left\| {{d_k}} \right\|_2^2 +  + {\alpha _k}\mathop {\max }\limits_{i \in \left[ m \right]} \left\langle {\nabla {f_i}\left( {{x_k}} \right) + \mu {f_0}\left( {{x_k}} \right), d_{k,P}^S - {d_k}} \right\rangle  + L\alpha _k^2\left\| { d_{k,P}^S - {d_k}} \right\|_2^2\\
& + L\alpha _k^2\left\| {{d_k}} \right\|_2^2 + \mathop {\max }\limits_{i \in \left[ m \right]} {f_i}\left( {{x_k}} \right) + \mu {f_0}\left( {{x_k}} \right),
\end{aligned}
\end{equation}

where according to (\ref{eq:pre_dk}), we can obtain
\[\mathop {\max }\limits_{i \in \left[ m \right]} \left\langle {\nabla {f_i}\left( {{x_k}} \right) + \mu {f_0}\left( {{x_k}} \right),{d_k}} \right\rangle  + \frac{1}{2}\left\| {{d_k}} \right\|_2^2 \le 0,\]
which implies that the last inequality holds. By taking expectations for both sides of (\ref{eq:pre_maxf}) and reorganizing, we obtain
\begin{align*}
\left( {\frac{{{\alpha _k}}}{2} - L\alpha _k^2} \right)\mathbb{E}\left\| {{d_k}} \right\|_2^2 &\le {\alpha _k}\mathbb{E}\left[ {\mathop {\max }\limits_{i \in \left[ m \right]} \left\langle {\nabla {f_i}\left( {{x_k}} \right) + \mu {f_0}\left( {{x_k}} \right), d_{k,P}^S - {d_k}} \right\rangle } \right] + L\alpha _k^2\mathbb{E}\left\| { d_{k,P}^S - {d_k}} \right\|_2^2\\
& + \mathbb{E}\left[ {\mathop {\max }\limits_{i \in \left[ m \right]} {f_i}\left( {{x_k}} \right) + \mu {f_0}\left( {{x_k}} \right)} \right] - \mathbb{E}\left[ {\mathop {\max }\limits_{i \in \left[ m \right]} {f_i}\left( {{x_{k + 1}}} \right) + \mu {f_0}\left( {{x_{k + 1}}} \right)} \right]\\
& \le {\alpha _k}\sqrt {2{M_0}\left( {C_f^2 + {\mu ^2}C_g^2} \right)} {S^{ - \frac{1}{4}}} + \alpha _k^2L{M_0}{S^{ - \frac{1}{2}}} + \mathbb{E}\left[ {\mathop {\max }\limits_{i \in \left[ m \right]} {f_i}\left( {{x_k}} \right) + \mu {f_0}\left( {{x_k}} \right)} \right]\\
& - \mathbb{E}\left[ {\mathop {\max }\limits_{i \in \left[ m \right]} {f_i}\left( {{x_{k + 1}}} \right) + \mu {f_0}\left( {{x_{k + 1}}} \right)} \right].
\end{align*}

Taking outer step as a fixed step ${\alpha_k = \alpha  \le \frac{1}{{4L}}}$, by telescoping the above inequality from ${k=0}$ to ${k=K-1}$ and dividing both sides by the step size $\frac{\alpha }{4}$ and the number of iterations ${K}$ simultaneously, we have
\begin{align*}
\frac{1}{K}\sum\limits_{k = 0}^{K - 1} {\mathbb{E}\left\| {{d_k}} \right\|_2^2} & \le 4\sqrt {2{M_0}\left( {C_f^2 + {\mu ^2}C_g^2} \right)} {S^{ - \frac{1}{4}}} + 4\alpha L{M_0}{S^{ - \frac{1}{2}}}\\
& + \frac{4}{{K\alpha }}\left( {\mathbb{E}\left[ {\mathop {\max }\limits_{i \in \left[ m \right]} {f_i}\left( {{x_0}} \right) + \mu {f_0}\left( {{x_0}} \right)} \right] - \mathbb{E}\left[ {\mathop {\max }\limits_{i \in \left[ m \right]} {f_i}\left( {{x_K}} \right) + \mu {f_0}\left( {{x_K}} \right)} \right]} \right)\\
& \le 4\sqrt {2{M_0}\left( {C_f^2 + {\mu ^2}C_g^2} \right)} {S^{ - \frac{1}{4}}} + 4\alpha L{M_0}{S^{ - \frac{1}{2}}} + \frac{4}{{K\alpha }}\left( {1 + \mu \sum\limits_{i = 1}^m {{w_i}} } \right)\left( {\mathop {\max }\limits_{i \in \left[ m \right]} {f_i}\left( {{x_0}} \right) - {F_{\inf }}} \right).
\end{align*}
Plugging ${{S} = {K^2}}$ and ${{M_{0,f}} = \left( {1 + \mu \sum\limits_{i = 1}^m {{w_i}} } \right)\left( {\mathop {\max }\limits_{i \in \left[ m \right]} {f_i}\left( {{x_0}} \right) - {F_{\inf }}} \right)}$ into the above inequality yields
\begin{align*}
\frac{1}{K}\sum\limits_{k = 0}^{K - 1} {\mathbb{E}\left\| {{d_k}} \right\|_2^2} & \le \frac{{4\sqrt {2{M_0}\left( {C_f^2 + {\mu ^2}C_g^2} \right)} }}{{\sqrt K }} + \frac{{4\alpha L{M_0}}}{K} + \frac{{4{M_{0,f}}}}{{K\alpha }}\\
& \le {\rm O}\left( {\frac{{4\sqrt {2{M_0}\left( {C_f^2 + {\mu ^2}C_g^2} \right)} }}{{\sqrt K }}} \right),
\end{align*}
completes the proof.
\end{proof}
For the SMOO problem with preferences, we provide convergence rate results for fixed inner and outer step sizes. Analogous proofs can be derived for other settings.
In conclusion, our MSMD achieves sublinear convergence rates for both SMOO with and without preferences, matching the results of state-of-the-art SMOO methods.

\section{Numerical experiments}
\label{sec:experment}
The numerical experiment comprises two primary components. The first is a benchmarking problem in MOO, where the objective of the optimizer is to approximate the Pareto front. The second component is an image classification task in MTL, where the optimizer is tasked with training a neural network to minimize classification loss.
All experiments were conducted on a 64-bit PC equipped with an Intel(R) i5-10600KF CPU running at 4.10GHz and having 16GB RAM along with a GeForce RTX 3060 Ti GPU.
Our method was compared against deterministic MOO methods such as MGDA \cite{sener2018multi} and PCGrad \cite{yu2020gradient},  as well as recent SMOO methods including CR-MOGM and SDMGrad. Additionally, we compared our method with traditional network optimizers such as SGD \cite{robbins1951stochastic}, Adam \cite{kingma2014adam} and root mean square prop (RMSProp) \cite{tieleman2012lecture}, which optimize the weighted sum of loss functions. For the traditional optimizer, we set the loss weighting parameter to ${\frac{1}{m}}$, where $m$ represents the number of objective functions. For all the methods, we use greedy search to select the best step size $\alpha_k$ and $\gamma_s$ in ${\{ 0.001,}$ ${0.01,}$ ${0.02,}$ ${0.03,}$ ${0.04,}$ ${0.5\}}$ and ${\{0.005,}$ ${0.01,}$ ${0.05\}}$, respectively, such that the obtained best performance. All other parameters, except for the step sizes, were left at their default values in PyTorch.

\subsection{MOO benchmark test function}
We conduct a comparative analysis of six MOO problems, including four 2-objective functions (BK1, FF1, Lov1, Lov3, Toi4) and one 3-objective function (MOP5). Each problem is initialized with box constraints $[lb,ub]$ as outlined in \cite{assunccao2021conditional}.
For all methods, the number of iteration steps is fixed at $K = 100$, and the initial number of points is set to $100$. For both SDMGrad and our MSMD method, which do not require a solver for subproblem optimization, we choose $S = 300$ steps for solving the subproblem, i.e., three times the value of $K$.


To introduce stochastic noise during the optimization process for SMOO, we sample independent variables from a normal distribution with mean 0 and standard deviation ${\rho = 0.5}$ times the range $[lb,ub]$, where ${\sigma = \rho[lb ,ub]}$ i.e. ${N(0,{\sigma ^2})}$. It is important to note that our method requires only a single sampling at each inner iteration s, while the SDMGrad method performs the sampling process three times per iteration.

\begin{figure}[h!]
\centering
\subfloat[BK1]{\includegraphics[scale=0.3]{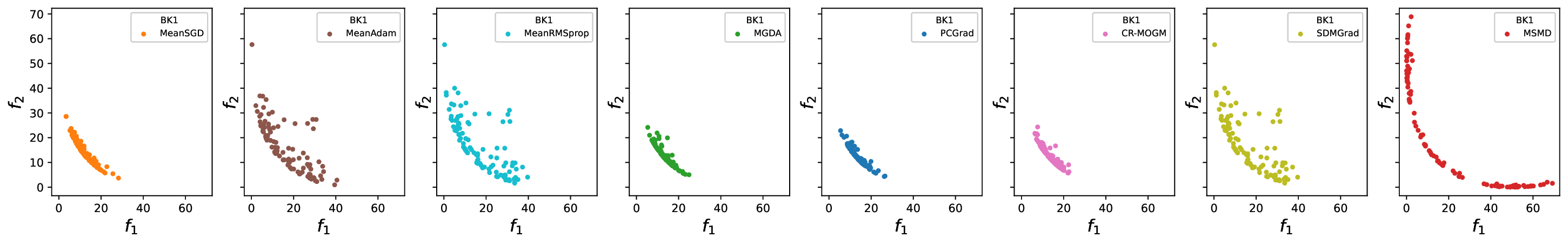}}\\
\subfloat[FF1]{\includegraphics[scale=0.3]{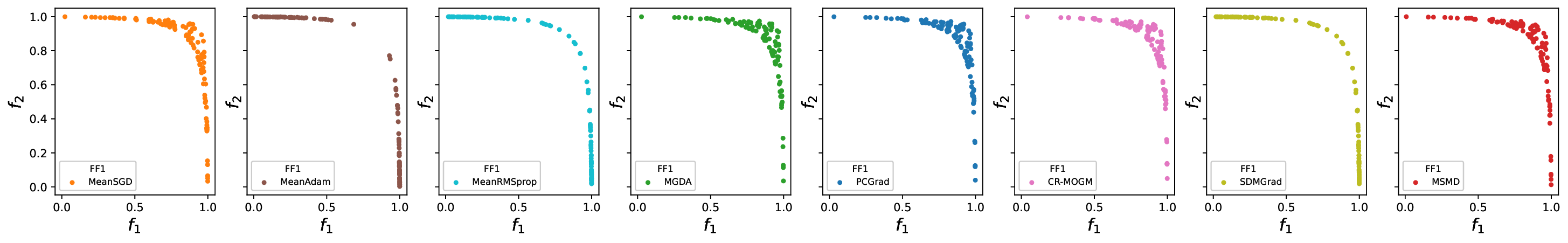}}\\
\subfloat[Lov1]{\includegraphics[scale=0.3]{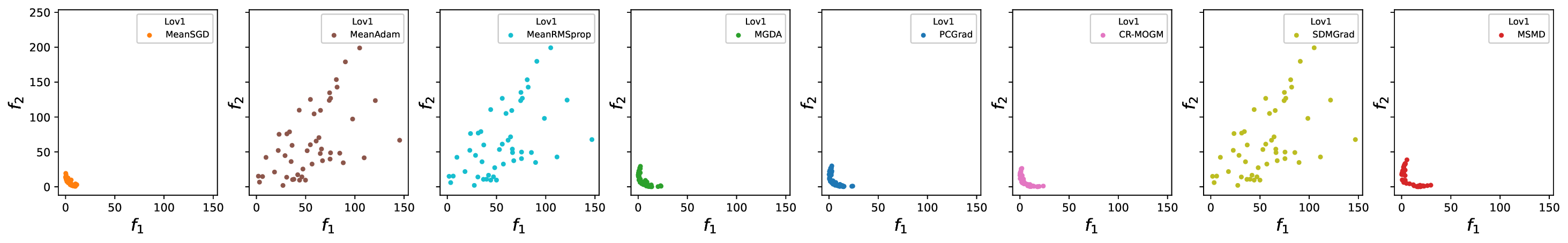}}\\
\subfloat[Lov3]{\includegraphics[scale=0.3]{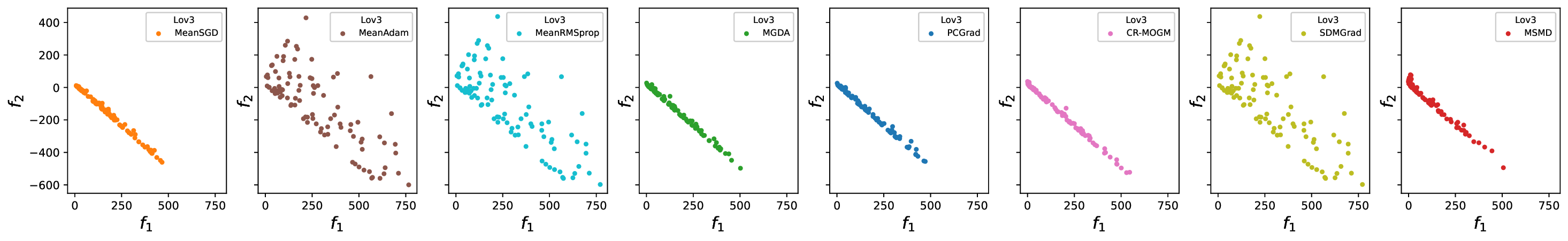}}\\
\subfloat[Toi4]{\includegraphics[scale=0.3]{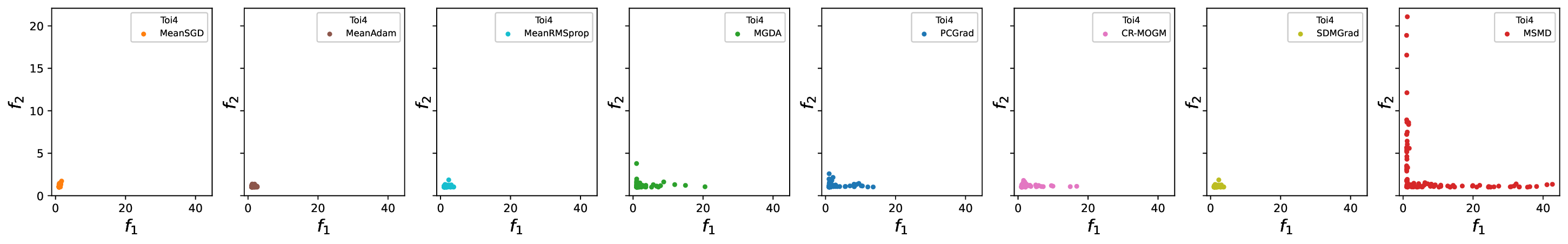}}\\

\caption{Comparison on the Pareto fronts of 2-dimensional benchmark MOO test functions.}
\label{fig:2d}
\end{figure}

\begin{figure}[h!]
\centering
\subfloat{\includegraphics[scale=0.3]{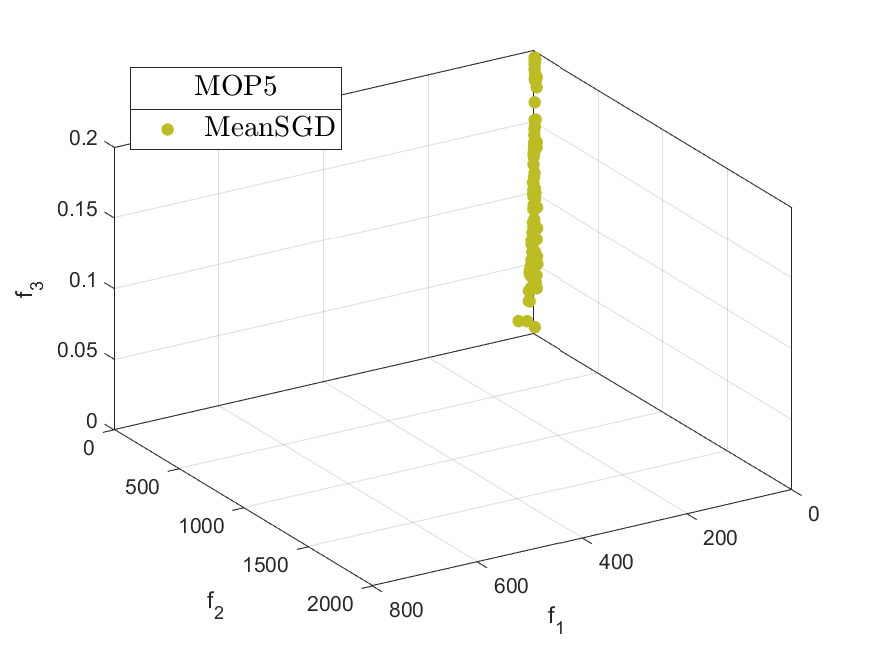}}
\subfloat{\includegraphics[scale=0.3]{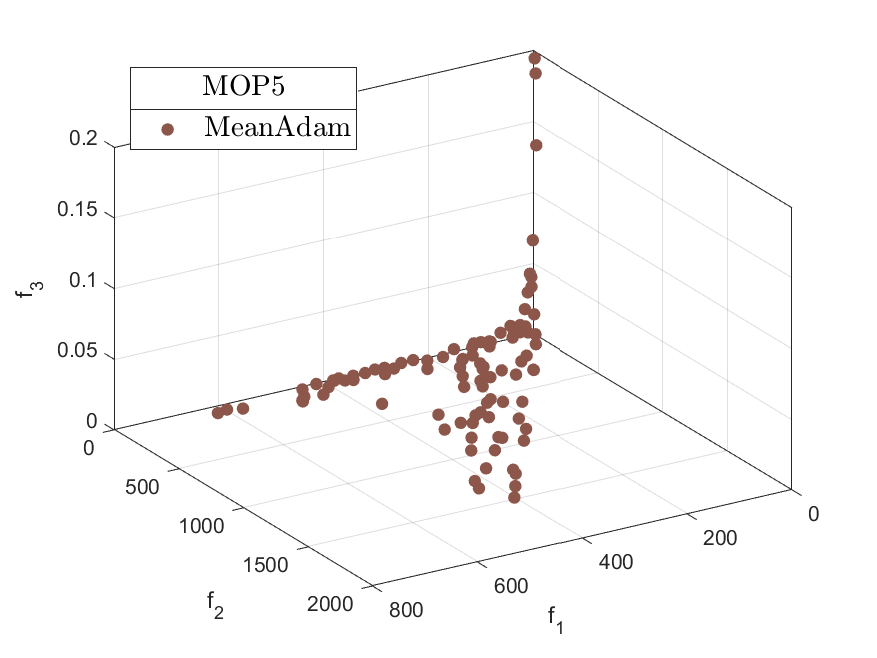}}
\subfloat{\includegraphics[scale=0.3]{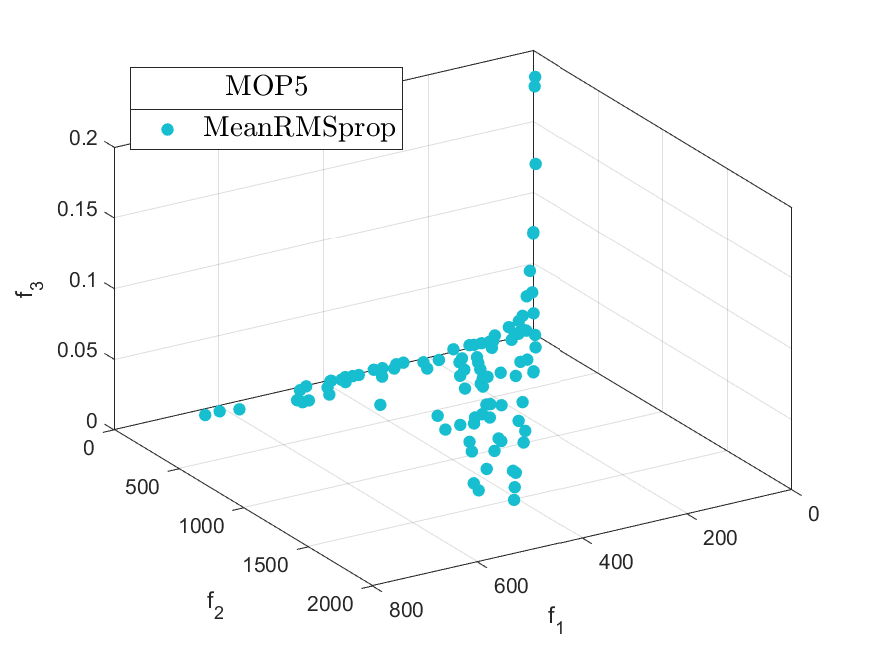}}
\subfloat{\includegraphics[scale=0.3]{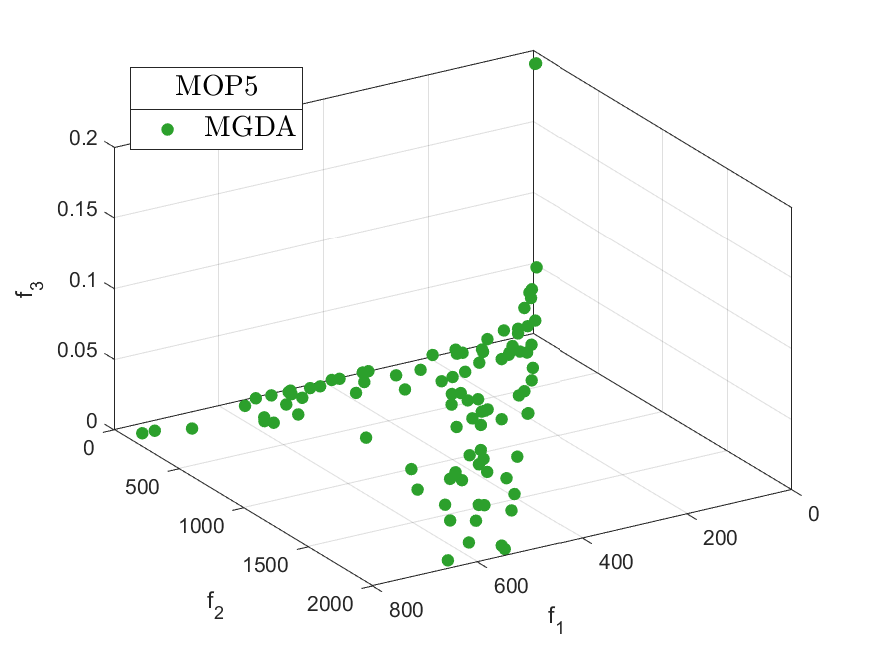}}\\
\subfloat{\includegraphics[scale=0.3]{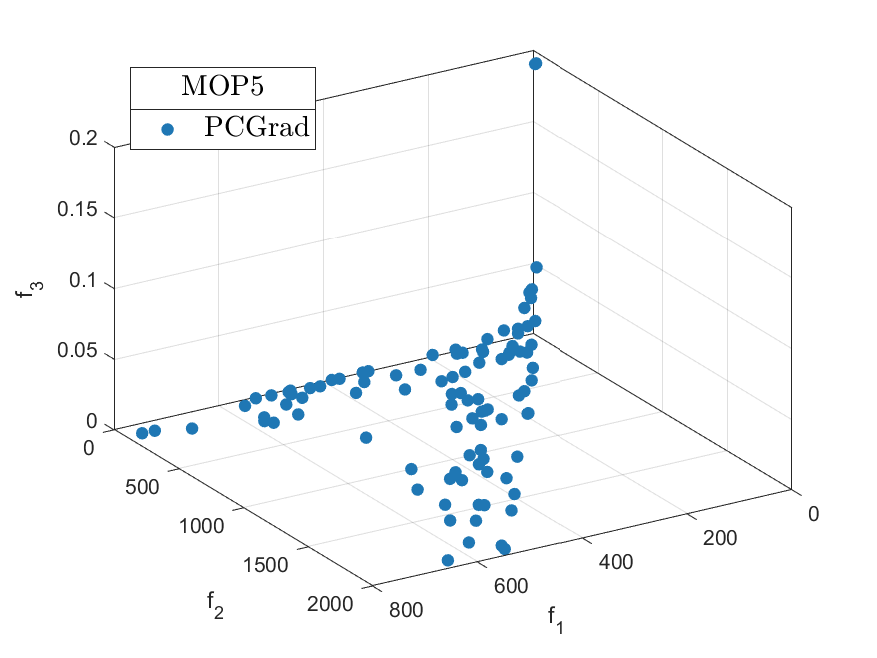}}
\subfloat{\includegraphics[scale=0.3]{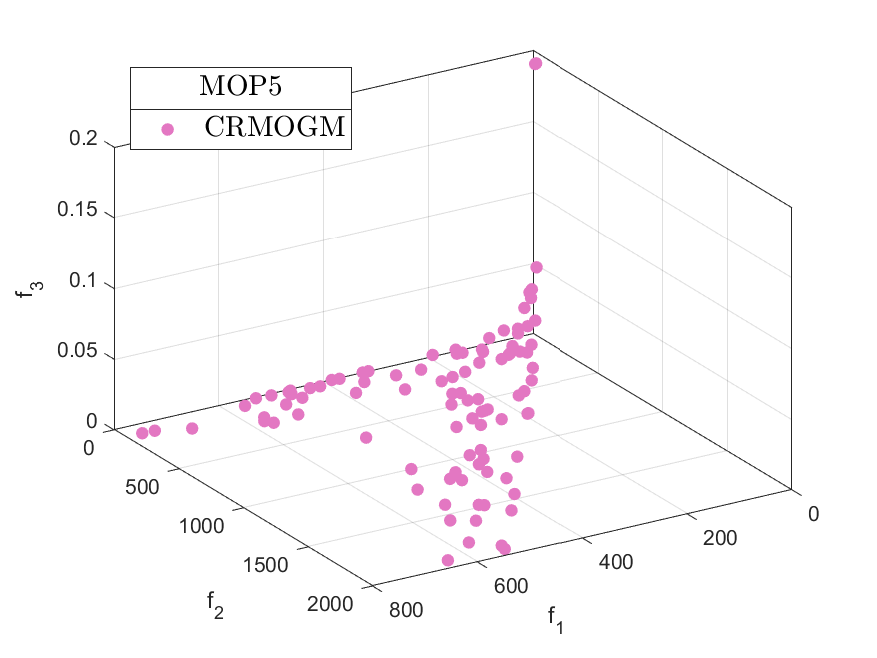}}
\subfloat{\includegraphics[scale=0.3]{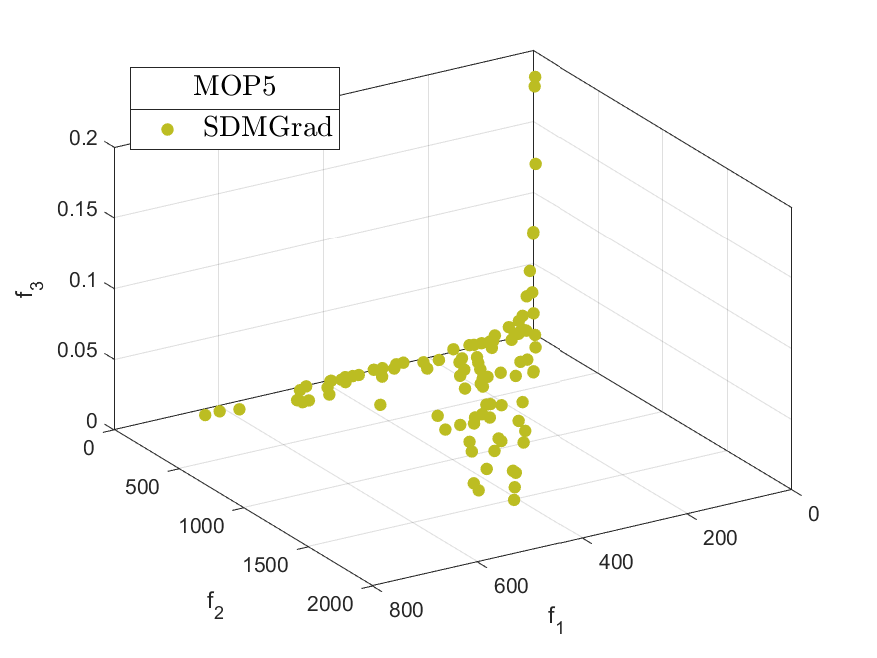}}
\subfloat{\includegraphics[scale=0.3]{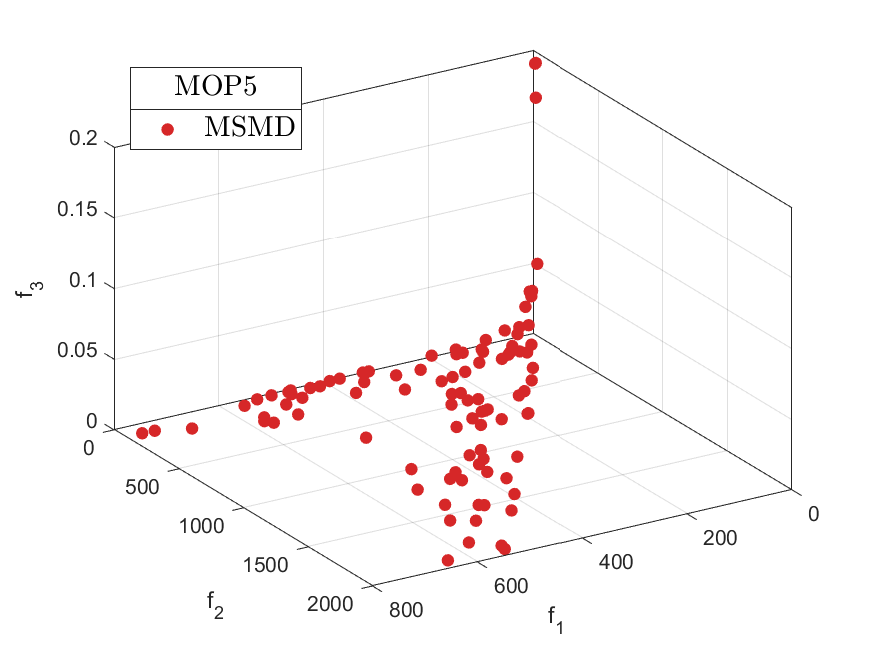}}
\caption{Comparison on the Pareto fronts of 3-dimensional benchmark MOO test function MOP5.}
\label{fig:mop5}
\end{figure}

The competitive Pareto fronts obtained by our method across all five 2-objective test functions, as shown in Figure \ref{fig:2d}, consistently outperform those of the seven comparison methods. Notably, on the BK1 problem, which involves convex quadratic functions for both objectives, our method produces the most comprehensive fronts and demonstrates remarkable stability against stochastic noise.
In the BK1 problem, the Pareto front obtained by the weighted sum method using SGD is concentrated primarily in the lower-left corner. In contrast, both Adam and RMSProp are significantly affected by the introduction of stochastic noise, likely due to their reliance on historical step information. This reliance results in accumulated stochastic noise, causing oscillations in the front surface.
The multi-gradient methods, which manipulate gradients, also produce Pareto fronts concentrated around solutions with smaller objective values. However, these methods fail to identify larger Pareto solutions for individual objectives, which our method successfully captures. Additionally, the SDMGrad method demonstrates poor stability in this problem.
Similar trends are observed in the Lov1, Lov3, and Toi4 problems, where our method identifies more complete front surfaces, particularly in Toi4, where the results significantly surpass those of the comparison methods.
For the 3-objective test function MOP5, as illustrated in Figure \ref{fig:mop5}, our method consistently identifies favorable Pareto fronts.
In conclusion, our MSMD approach outperforms both the traditional weighted and stochastic methods, as well as stochastic multi-gradient methods, in discovering superior Pareto front surfaces across these benchmark problems

\subsection{Multi-task learning}
In this subsection, we evaluate our method by training neural networks on the MultiMNIST dataset. In the specific experimental setup, we adopt the hard parameter sharing scheme of MTL, where the parameters of the task-sharing part are derived by solving for optimal parameters using the SMOO method, and the task-specific part is updated directly through SGD.
We can formulate the loss function of MTL as the following bi-objective optimization problem
\[\mathop {\min }\limits_{x \in X} L = \left( {{L^1}\left( {x,\xi } \right),{L^2}\left( {x,\xi } \right)} \right)\]
where ${x \in X}$ is the parameter of the neural network, $\xi$ is a batch of data, and we chose ${L^1}$ and ${L^2}$ are both cross-entropy loss function.
The following elaborates on the architecture of the neural network we are about to train.
We have two convolutional neural networks (CNN) with different shared part structures and the same task-specific structure. The structure of the shared encoder of the convolutional neural network are \textit{c-c-p} (CNN1) and \textit{c-c-p-c-c-p} (CNN-2), where \textit{c} and \textit{p} represent the convolution and max-pooling, respectively. The task-specific network with structures of \textit{c-f}, where $f$ represent fully connected layer with $50$ hidden units. Convolution kernel is with size of ${3 \times 3}$ and the max-pooling layer is with size of ${2 \times 2}$ and stride $2$. Both CNN1 and CNN2 are trained with batch normalization.
Finally, for two commonly used neural network architectures, we  utilized  the same setup as in \cite{yang2023learning} to obtain a modified LeNet5 and a modified VGG.

The initial values of the parameters in the neural network are sampled independently from a Gaussian distribution.
During the training process, the batch size of the training optimiser is set to $16$. Similar to the setting in the previous subsection, we report the final results as the average of $10$ training processes. We choose $S = 400$ for training tasks with $K = 200$ and $S = 500$ for tasks with $K \ge 500$. Following these settings, we can demonstrate the performance of our method and compare it to other methods on different training tasks.

In addition, we followed the steps outlined in \cite{ma2020efficient} and \cite{yang2023learning} to generate the MultiMNIST dataset.
We placed two ${28 \times 28}$ MNIST images in the upper left and lower right corners and introduced random offsets of up to 2 pixels in each direction to generate a synthetic image of size ${36 \times 36}$. Finally, we resized the synthetic image to ${28 \times 28}$ and normalized it with a mean of $0.1307$ and a standard deviation of $0.3081$. Each dataset consists of 60,000 training images and 10,000 test images, illustrating the MultiMNIST dataset.
Figure \ref{fig:multimnist} shows the synthetic MultiMNIST dataset, where each image has two labels. One label corresponds to the handwritten digit in the upper left corner, denoted as L, while the other represents the handwritten digit in the lower right corner, denoted as R.

\begin{figure}[ht]
\centering
\includegraphics[scale=0.45]{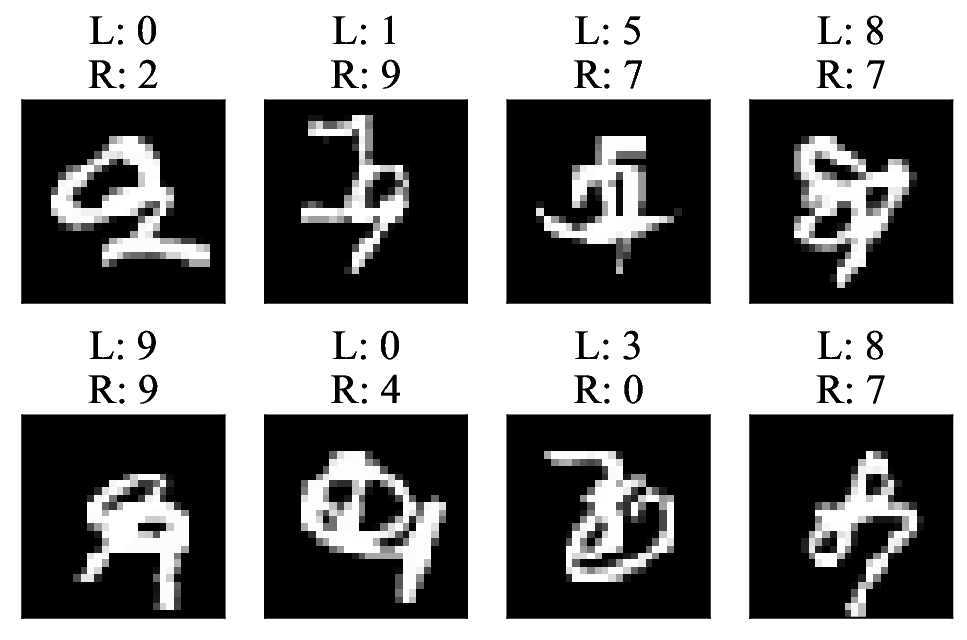}
\caption{MultiMNIST image samples.}
\label{fig:multimnist}
\end{figure}

In Table \ref{tab:cnn}, we provide the loss function, Top1 classification accuracy, and Top5 classification accuracy for each of the eight methods under different training tasks.  The metrics are separated by task1 and task2, where task1 refers to classifying the left image and task2 refers to classifying the right image.
Table \ref{tab:cnn} evaluates the performance of several SMOO methods, including MSMD, applied to CNN1 and CNN2 architectures on the Multi-MNIST dataset. CNN2 features a more complex structure compared to CNN1, with double the convolutional and pooling layers, which  allows for a deeper representation of the input data.
MSMD exhibits several notable strengths. First, it consistently achieves the lowest loss across different architectures and iteration steps. For instance, in CNN1 with K=200, MSMD reports a loss of 0.11 for Task1 and 0.22 for Task2, which outperforms other methods like MGDA and CR-MOGM. This low loss indicates faster convergence and greater efficiency in optimizing the model, an essential characteristic in complex MTL problems.

Additionally, MSMD performs exceptionally well in terms of accuracy, particularly in the Top5 accuracy. For CNN1 with K=200, MSMD achieves nearly perfect scores of 99.52\% for Task1 and 99.64\% for Task2, outperforming all other methods. This trend persists when the number of iterations is increased to K=1000, where MSMD still leads with 99.64\% accuracy for both tasks. Such robustness across different iterations and architectures indicates that MSMD is highly stable and effective, even with more complex tasks and network structures.

However, a slight decline in performance is observed when moving from CNN1 to the more complex CNN2 architecture. In CNN2 with K=200, the Top1 accuracy of MSMD decreases to 81.23\% for Task1 and 73.63\% for Task2, while in CNN1, it reached 89.05\% and 86.56\%, respectively. This drop could be attributed to the increased complexity of CNN2, where the additional layers might require further fine-tuning or a larger number of iterations to fully leverage the deeper network capabilities. Nonetheless, MSMD still outperforms other methods in this scenario, highlighting its adaptability to complex network structures.

Figure \ref{fig:cnn} illustrates the decreasing trend of the loss function for the SMOO optimizer during training, with each curve representing the average of 10 experiments. It is evident that the loss function for the gradient-based approach is significantly lower than that of the weighting method, demonstrating its effectiveness in neural network training. Although the proposed MSMD method is not as fast as several gradient-based methods (MGDA, PCGrad, CR-MOGM, SDMGrad) in CNN2, the decline in its loss function is substantial and closely parallels the top-performing methods.

In conclusion, MSMD demonstrates clear advantages, particularly in minimizing loss and maximizing accuracy across various tasks and architectures. Its performance in more complex network architectures like CNN2 suggests that further optimization, such as increasing the number of iterations or adjusting the hyperparameters, could enhance its effectiveness. Nevertheless, the strong overall performance of MSMD, especially in terms of accuracy, underscores its potential as a reliable method for SMOO in training simple neural network.

\begin{table}[htbp]
  \centering
  \caption{ Performance of different MOO optimizers of CNN on Multi-MNIST test dataset.}
    \begin{tabular}{c|c|c|cccccccc}
    \toprule
    \multicolumn{3}{c|}{Neural NetWork} & SGD   & ADAM  & RMSProp & MGDA  & PCGrad & CR-MOGM & SDMGrad & MSMD \\
    \midrule
          & \multirow{2}[2]{*}{Loss} & Task1 & 2.28  & 2.96  & 2.78  & 0.42  & 0.48  & 0.32  & 0.46  & \textbf{0.11} \\
          &       & Task2 & 2.47  & 2.49  & 2.67  & 0.04  & 0.41  & 0.49  & 0.41  & \textbf{0.22} \\
\cmidrule{2-11}    CNN1  & Top1  & Task1 & 13.05 & 9.63  & 13.33 & 79.52 & 79.00 & 80.97 & 81.23 & \textbf{89.05} \\
    K=200 & Acc(\%) & Task2 & 10.04 & 12.54 & 10.01 & 74.14 & 73.86 & 77.95 & 74.63 & \textbf{86.56} \\
\cmidrule{2-11}          & Top5  & Task1 & 51.99 & 52.91 & 52.29 & 98.62 & 98.53 & 98.78 & 98.84 & \textbf{99.52} \\
          & Acc(\%) & Task2 & 49.65 & 50.27 & 50.49 & 98.20 & 98.40 & 98.64 & 98.69 & \textbf{99.64} \\
    \midrule
          & \multirow{2}[2]{*}{Loss} & Task1 & 2.28  & 2.96  & 2.78  & 0.42  & 0.42  & 0.48  & 0.41  & \textbf{0.40} \\
          &       & Task2 & 2.47  & 2.49  & 2.67  & \textbf{0.40} & 0.41  & 0.41  & 0.66  & 0.41 \\
\cmidrule{2-11}    CNN2  & Top1  & Task1 & 9.63  & 9.63  & 12.33 & 79.52 & 79.00 & 79.00 & 80.97 & \textbf{81.23} \\
    K=200 & Acc(\%) & Task2 & 12.54 & 12.54 & 10.01 & \textbf{74.14} & 73.86 & 73.86 & 68.05 & 73.63 \\
\cmidrule{2-11}          & Top5  & Task1 & 52.19 & 52.91 & 52.29 & 98.62 & 98.53 & 98.53 & \textbf{98.78} & 98.24 \\
          & Acc(\%) & Task2 & 49.95 & 50.57 & 50.94 & 98.20 & \textbf{98.40} & \textbf{98.40} & 97.69 & 98.26 \\
    \midrule
          & \multirow{2}[2]{*}{Loss} & Task1 & 2.48  & 2.39  & 2.41  & 0.18  & 0.24  & 0.17  & \textbf{0.15} & \textbf{0.15} \\
          &       & Task2 & 2.45  & 2.94  & 2.37  & 0.16  & 0.34  & 0.39  & 0.41  & \textbf{0.33} \\
\cmidrule{2-11}    CNN1  & Top1  & Task1 & 11.90 & 11.87 & 10.03 & \textbf{89.80} & 89.42 & 88.70 & 89.35 & 89.63 \\
    K=1000 & Acc(\%) & Task2 & 9.47  & 11.99 & 11.36 & 86.42 & 85.53 & 86.53 & 86.05 & \textbf{86.72} \\
\cmidrule{2-11}          & Top5  & Task1 & 47.26 & 47.56 & 52.13 & 99.63 & \textbf{99.64} & 99.52 & 99.52 & \textbf{99.64} \\
          & Acc(\%) & Task2 & 48.98 & 50.65 & 20.77 & \textbf{99.51} & 99.35 & 99.42 & 99.42 & \textbf{99.51} \\
    \bottomrule
    \end{tabular}%
  \label{tab:cnn}%
\end{table}%

\begin{figure}[h]
\centering
\subfloat{\includegraphics[scale=0.4]{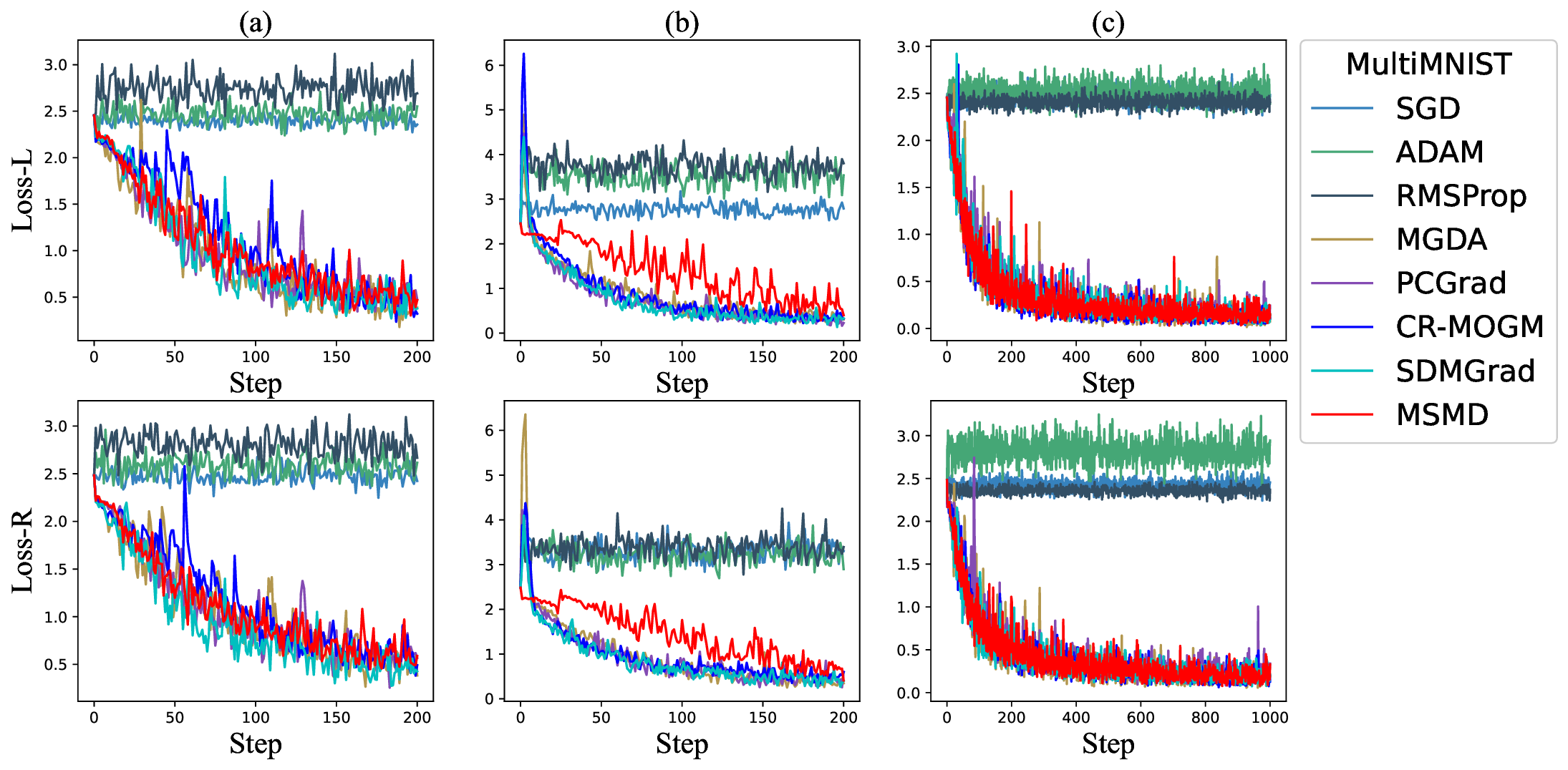}}\\

\caption{Comparison of loss function curves for training CNN. (a) CNN1 with step size $K = 200$ and $S = 400$; (b) CNN2 with step size $K = 200$ and $S = 400$; (c) CNN1 with step size $K = 1000$ and $S = 500$.}
\label{fig:cnn}
\end{figure}


In Table \ref{tab:vgg} and Figure \ref{fig:vgg}, the modified LeNet5 and modified VGG networks, MSMD shows competitive performance. For Modified LeNet5, MSMD achieves a loss of 0.22 for Task1 and 0.24 for Task2, slightly higher than MGDA (0.18, 0.30) but still lower than most of the other optimizers. When it comes to Top1 accuracy, MSMD achieves 85.49\% and 81.20\% for Task1 and Task2, respectively. While MGDA and PCGrad slightly outperform MSMD in Top1 accuracy, with MGDA reaching 85.17\% and 80.28\%, MSMD remains competitive and delivers higher Top5 accuracy. In Top5 accuracy, MSMD achieves 99.16\% for Task1 and 99.00\% for Task2, the highest values across all optimizers, reflecting its strength in capturing more correct predictions within the top predictions.

For the modified VGG, MSMD again demonstrates robust performance, particularly in the loss and Top5 accuracy. It records the lowest loss among all methods with values of 0.26 for Task1 and 0.21 for Task2, indicating effective convergence. While in Top1 accuracy, MSMD attains 86.18\% for Task1 and 82.69\% for Task2, which is competitive but slightly lower than 86.94\% and 84.24\% of SDMGrad, MSMD shines in Top5 accuracy. It achieves 99.22\% and 99.31\% for Task1 and Task2, respectively, the highest in the table, demonstrating superior capability in reducing misclassifications when considering the top five predictions.
Thus, MSMD consistently performs well across multiple metrics, with particular strengths in loss minimization and Top5 accuracy. Its versatility across different architectures, highlights its adaptability and robustness in MTL environments.
\begin{table}[htbp]
  \centering
  \caption{Performance of different MOO optimizers of modified LeNet5 and modified VGG on Multi-MNIST test dataset.}
    \begin{tabular}{c|c|c|cccccccc}
    \toprule
    \multicolumn{3}{c|}{Neural NetWork} & SGD   & ADAM  & RMSProp & MGDA  & PCGrad & CR-MOGM & SDMGrad & MSMD \\
    \midrule
          & \multirow{2}[2]{*}{Loss} & Task1 & 2.30  & 0.25  & 2.30  & \textbf{0.18} & 0.24  & 0.34  & 0.20  & 0.22 \\
          &       & Task2 & 2.30  & 2.26  & 2.30  & 0.30  & 0.27  & 0.62  & 0.28  & \textbf{0.24} \\
\cmidrule{2-11}    Modified & Top1  & Task1 & 14.06 & 10.27 & 13.29 & 85.17 & \textbf{86.00} & 82.19 & 84.53 & 85.49 \\
    LeNet5 & Acc(\%) & Task2 & 9.84  & 11.26 & 10.05 & 80.28 & 79.69 & 75.95 & 80.28 & \textbf{81.20} \\
\cmidrule{2-11}          & Top5  & Task1 & 49.98 & 51.66 & 52.64 & 99.10 & 99.11 & 99.05 & 99.12 & \textbf{99.16} \\
          & Acc(\%) & Task2 & 51.28 & 48.85 & 50.99 & 98.75 & 98.67 & 98.48 & 98.94 & \textbf{99.00} \\
    \midrule
          & \multirow{2}[2]{*}{Loss} & Task1 & 2.32  & 2.30  & 2.34  & 0.40  & 0.35  & 0.25  & \textbf{0.21} & 0.26 \\
          &       & Task2 & 2.28  & 2.34  & 2.30  & 0.54  & 0.50  & 0.46  & 0.28  & \textbf{0.21} \\
\cmidrule{2-11}    Modified & Top1  & Task1 & 10.43 & 11.75 & 10.78 & 82.89 & 85.38 & 84.62 & \textbf{86.94} & 86.18 \\
    Vgg   & Acc(\%) & Task2 & 10.69 & 10.38 & 10.21 & 78.08 & 79.88 & 78.78 & \textbf{84.24} & 82.69 \\
\cmidrule{2-11}          & Top5  & Task1 & 50.08 & 50.13 & 50.20 & 99.19 & 99.34 & 99.22 & \textbf{99.26} & 99.22 \\
          & Acc(\%) & Task2 & 51.18 & 50.94 & 50.90 & 99.00 & 99.15 & 99.06 & 99.00 & \textbf{99.31} \\
    \bottomrule
    \end{tabular}%
  \label{tab:vgg}%
\end{table}%

\begin{figure}[h]
\centering
\subfloat{\includegraphics[scale=0.4]{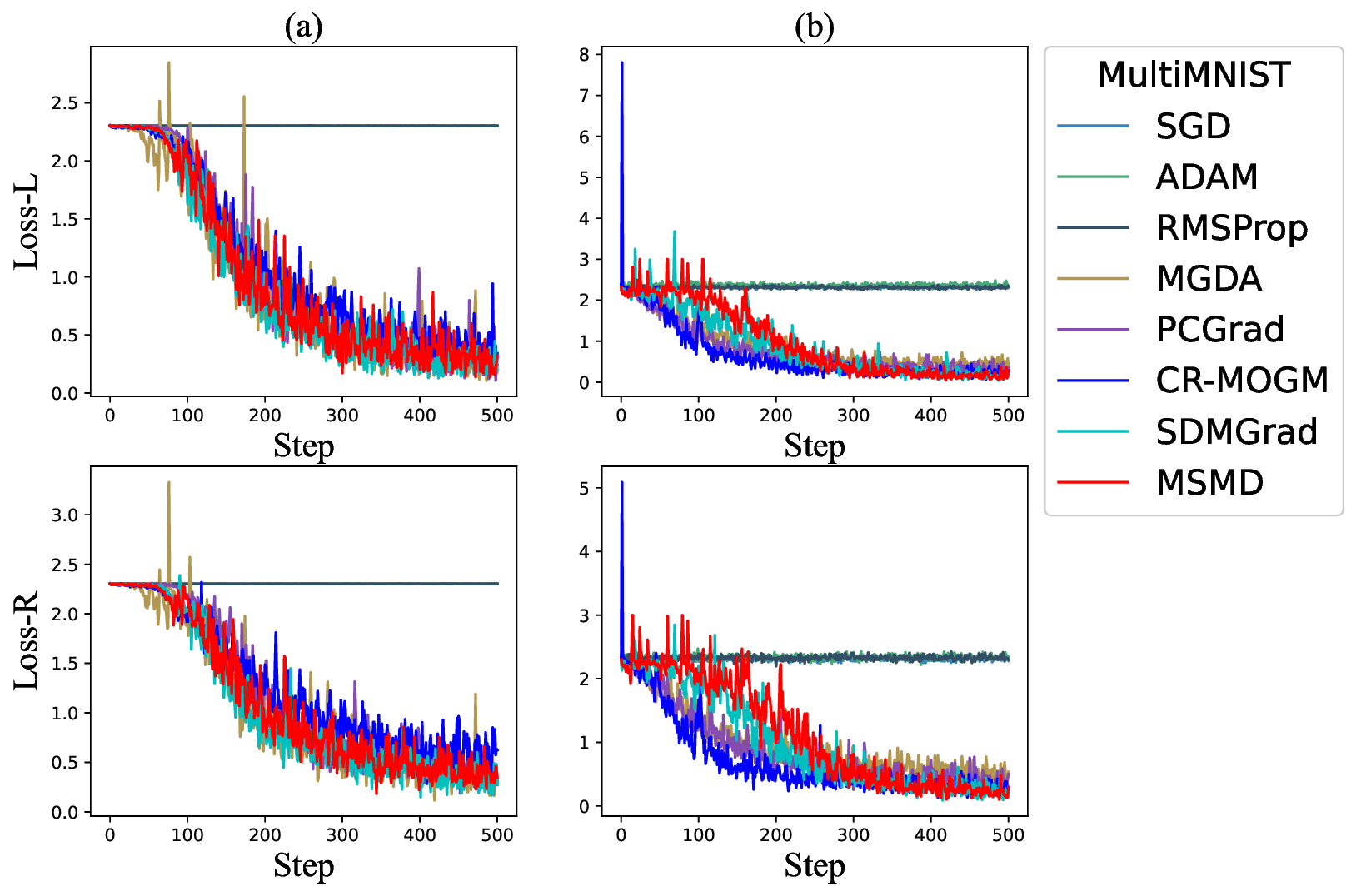}}\\

\caption{Comparison of loss function curves for training modified LeNet5 and modified VGG. (a) modified LeNet5 with step size $K = 500$ and $S = 500$; (b) modified VGG with step size $K = 500$ and $S = 500$.}
\label{fig:vgg}
\end{figure}

The previous experimental results indicate that SDMGrad performs comparably to our method on most of the training tasks. Specifically, SDMGrad outperforms our method on Modified VGG. However, it is worth noting that the iteration process of SDMGrad requires 3 samples, which is significantly more time-consuming compared to the single sample requirement of our method.
We present the time comparisons between SDMGrad and MSMD for different training tasks in Figure \ref{tab:time}. Each time is the average of the $S$ steps of the internal iterations to solve the subproblem, with the same value of $S$ chosen for both SDMGrad and MSMD. Finally, we calculate the average of these averaged times over the total number of external iterations $K$ to obtain the values reported in Table \ref{tab:time}.

For the CNN1 model with K=200, MSMD takes only 0.20 seconds compared to SDMGrad's 3.99 seconds. Similarly, for the CNN2 model with K=200, MSMD completes in 3.97 seconds, significantly faster than 11.35 seconds of SDMGrad. Even with a larger model, CNN1 K=1000, MSMD remains more efficient, taking 2.11 seconds compared to 6.62 seconds of SDMGrad. Furthermore, on the modified LeNet5 and VGG architectures, MSMD continues to show superior performance with 1.82 seconds and 13.66 seconds, respectively, while SDMGrad requires 7.33 seconds and 29.35 seconds.

Overall, MSMD achieves an average subproblem-solving time of 4.32 seconds, significantly lower than 11.73 seconds (SDMGrad), and almost $1/3$ of SDMGrad, which corresponds to our sample size of only $1/3$ of SDMGrad. This highlights the superior computational efficiency of MSMD, making it highly advantageous, especially for more complex models and tasks.

\begin{table}[htbp]
  \centering
  \caption{Comparison of average time for solving subproblems between SDMGrad and MSMD on different training tasks}
    \begin{tabular}{c|cccccccc|c}
    \toprule
    \multirow{2}[2]{*}{Method}   & CNN1  & CNN2  & CNN1  & Modified & Modified & \multirow{2}[2]{*}{Average} \\
           & K=200 & K=200 & K=1000 & LeNet5 & Vgg   &  \\
    \midrule
    SDMGrad  & 3.99s & 11.35s & 6.62s & 7.33s & 29.35s & 11.73s\\
    \midrule
    MSMD  & 0.20s & 3.97s & 2.11s & 1.82s & 13.66s & \textbf{4.32s} \\
    \bottomrule
    \end{tabular}%
  \label{tab:time}%
\end{table}%

\section{Conclusion}

In this paper, we propose the MSMD method, a novel SMOO method with convergence guaranteed. This method effectively addresses the challenge of analyzing SMOO convergence in the presence of stochastic noise. Our key strategy is to leverage the SMD method for solving the SMOO subproblem and obtaining an update direction with reduced bias. We demonstrate the sublinear convergence rate of MSMD under four different step size setups. Furthermore, we present a variant of MSMD tailored for SMOO with preferences and briefly outline its convergence. Finally, our numerical experiments on benchmark test functions and neural network training validate that our method excels in generating Pareto fronts and achieving high classification accuracy.

In our future work, we aim to further apply advanced optimization methods for the stochastic min-max problem in the design of the SMOO method. Apart from the SMD method utilized in this work, there exist several other techniques for addressing the stochastic min-max problem, such as the stochastic gradient descent ascen method without relying on noise strength assumption \cite{loizou2021stochastic}, the variance decay method \cite{palaniappan2016stochastic} and the coordinate randomization method \cite{sadiev2021zeroth}. These methods are based on different assumptions and have been analyzed under various scenarios to enhance upon classical SMD method \cite{nemirovski2009robust,dem1972numerical}. Therefore, it is valuable to analyze the convergence of SMOO based on our proposed novel idea and explore the performance associated with different stochastic methods when solving the SMOO problem.
\subsubsection*{Acknowledgments}
This work was funded by the Major Program of the National Natural Science Foundation of China (Grant Nos. 11991020, 11991024); the National Natural Science Foundation of China (Grant Nos. 12431010, 12171060); NSFC-RGC (Hong Kong) Joint Research Program  (Grant No. 12261160365); the Team Project of Innovation Leading Talent in Chongqing (Grant No. CQYC20210309536); the Natural Science Foundation of Chongqing of China (Grant No. ncamc2022-msxm01), the Major Project of Science and Technology Research Rrogram of Chongqing Education Commission of China (Grant No. KJZD-M202300504) and the Foundation of Chongqing Normal University (Grant Nos. 22XLB005, 22XLB006).

\end{document}